\documentclass[10pt,leqno]{amsart}
\usepackage[a4paper, margin=2.3cm]{geometry}
\usepackage{eulervm} 

\usepackage{newcent}

\usepackage{ulem}
\usepackage{amsthm}
\usepackage{mathrsfs}
\usepackage{lineno}
\usepackage{amsaddr}

\usepackage{float}
\usepackage[all]{xy}\usepackage{xypic}
\usepackage{braket}
\usepackage[colorlinks = true,citecolor=black]{hyperref}
\hypersetup{
  pdftitle   = {},
  pdfauthor  = {},
  pdfcreator = {\LaTeX\ with package \flqq hyperref\frqq}
}

\DeclareMathAlphabet{\mathpzc}{OT1}{pzc}{m}{it}
\usepackage{tikz-cd}
\usepackage{tkz-graph}
\usetikzlibrary{shapes.geometric,arrows, shapes}

\newtheorem{theorem}{Theorem}[section]
\newtheorem*{theorem*}{Theorem}
\newtheorem{theorem-non}{Theorem}
\newtheorem{proposition}[theorem]{Proposition}

\newtheorem*{lemma*}{Lemma}
\newtheorem{corollary}[theorem]{Corollary}
\newtheorem{conjecture}[theorem]{Conjecture}
\newtheorem*{conjecture*}{Conjecture}

\theoremstyle{definition}
\newtheorem{definition}[theorem]{Definition}
\newtheorem{example}[theorem]{Example}

\theoremstyle{remark}
\newtheorem{remark}[theorem]{Remark}

\DeclareMathOperator{\im}{im}

\numberwithin{equation}{section}

\begin{document}
\title[Homogeneous contact manifolds and resolutions of Calabi-Yau cones]{Homogeneous contact manifolds and resolutions of Calabi-Yau cones}

\author{Eder M. Correa}
\address{\resizebox{12cm}{.2cm}{ \textit{IMPA \ - \ Instituto de Matem\'{a}tica Pura e Aplicada,  Estr. Dona Castorina, 110, Rio de Janeiro, 22460-320, Brasil}} }

\thanks{Eder M. Correa was supported by CNPq grant 150899/2017-3}

\thanks{ E-mail: \rm edermoraes@impa.br}

%\linenumbers

\begin{abstract} 
In the present work we provide a constructive method to describe contact structures on compact homogeneous contact manifolds. The main feature of our approach is to describe the Cartan-Ehresmann connection (gauge field) for principal ${\rm{U}}(1)$-bundles over complex flag manifolds by using elements of representation theory of simple Lie algebras. This description allows us to compute explicitly the expression of the contact form for any Boothby-Wang fibration over complex flag manifolds \cite{BW} as well as their underlying Sasaki structures. By following \cite{COLON}, \cite{RESOLUTIONCOMPSUP}, and \cite{GOTO}, as an application of our results we use the Cartan-Remmert reduction \cite{GRAUERT} and the Calabi Ansatz technique \cite{CALABIANSATZ} to provide many explicit examples of crepant resolutions of Calabi-Yau cones with certain homogeneous Sasaki-Einstein manifolds realized as links of isolated singularities. These concrete examples illustrate the existence part of the conjecture introduced in \cite{ADSCFT}.

\end{abstract}

\maketitle

\hypersetup{linkcolor=black}
\tableofcontents

\hypersetup{linkcolor=black}

\section{Introduction}

\subsection{An overview on contact geometry}
As stated in \cite{GEIGES}, the roots of Contact Geometry can be traced back to 1872, when Sophus Lie introduced the notion of contact transformation \cite{SOPHUS} as a geometric tool to study systems of differential equations. The subject has manifold connections with other fields of pure mathematics, and a significant place in applied areas such as mechanics, optics, thermodynamics, and control theory. 

According to \cite{GEIGESII}, the study of contact manifolds in the modern sense can be traced back to the work of Georges Reeb \cite{REEB}, who referred to a strict contact manifold $(M,\eta)$ as a ``syst\`{e}me dynamique avec invariant int\'{e}gral de Monsieur Elie Cartan". The relation with dynamical systems comes from the fact that a contact form $\eta$ gives rise to a vector field $\xi$ defined uniquely by the equations

\begin{center}

$d\eta (\xi, \cdot) = 0,$ \ \ $\eta(\xi) = 1.$

\end{center}
This vector field is nowadays called the {\textit{Reeb vector field}} of $\eta$, see for instance \cite{GEIGES}, \cite{GEIGESII}, \cite{LUTZ}.

Since its foundation, contact geometry has been seen to underlie many physical phenomena and be related to many other mathematical structures. An important feature of contact geometry is its connection with symplectic geometry. Actually, given a contact manifold $(M,\eta)$, it is straightforward to check that the cone 
\begin{equation}
\label{conesymplectic}
\big (\mathscr{C}(M) = \mathbb{R}^{+} \times M, \omega_{\mathscr{C}} = \frac{1}{2}d(r^{2}\eta) \big),
\end{equation}
is a symplectic manifold, also known as {\textit{symplectization}} of $(M,\eta)$, see for instance \cite{MCDUFF}. Likewise, the Reeb field $\xi$ associated to $\eta$ defines a foliation $\mathcal{F}_{\eta}$ on $M$ called {\textit{characracteristic foliation}}. When this foliation is regular and $M$ is compact, the transverse space is a smooth symplectic manifold $(N,\omega_{N})$ giving a projection $\pi$ over the space of leaves $N = M/\mathcal{F}_{\eta}$ called Boothby-Wang fibration \cite{BW}, such that $\pi^{\ast}\omega_{N} = \frac{1}{2}d\eta$. In this last case, we have that $\pi \colon (M,\eta) \to (N,\omega_{N})$ defines a principal ${\rm{U}}(1)$-bundle over $(N,\omega_{N})$ and $\eta$ induces a connection 1-form on $M$. The following diagram illustrate how symplectic geometry arises from contact geometry through of these two different perspectives.

\begin{center}

\begin{tikzcd}
(\mathscr{C}(M),\omega_{\mathscr{C}}) \arrow[r, "\iota", hookleftarrow ] & (M,\eta) \arrow[d, "\pi", twoheadrightarrow] \\
                                                                & (N,\omega_{N}) 
\end{tikzcd}

\end{center}

The basic setting in which the Boothby-Wang fibration becomes even more interesting is when the transverse space $N = M/\mathcal{F}_{\eta}$ is a K\"{a}hler manifold. In this setting is quite reasonable to ask if there is a Riemannian metric $g_{M}$ on $M$ which ``best fits" into the diagram above. Alternatively, one could ask for a Riemannian metric $g_{M}$ on $M$ which would define a K\"{a}hler metric $g_{N}$ on $N$ via Riemannian submersion. Surprisingly, in both cases the answer to these questions leads naturally and uniquely to Sasakian geometry \cite{SASAKI}, \cite{BOYERGALICKI}. Thus, Sasakian geometry can be seen in some sense as the odd-dimensional analogue of K\"{a}hler geometry.

In fact the latter, for positive Ricci curvature, is strictly contained in the former; Sasaki-Einstein geometry is thus a generalization of K\"{a}hler-Einstein geometry. From this point of view, it is quite clear that K\"{a}hler and Sasaki geometries are inseparable \cite{SPARKS}. 

Another remarkable feature of the Sasaki-Einstein condition is that it implies that the cone $(\mathscr{C}(M),\omega_{\mathscr{C}})$ is a (non-compact) Calabi-Yau manifold, namely, $\omega_{\mathscr{C}}$ defines a K\"{a}hler Ricci-flat metric $g_{\mathscr{C}}$ on $\mathscr{C}(M)$, see for instance \cite{BOYERGALICKI}.

Sasakian geometry has recently proven to be a rich source for the production of Einstein metrics, see for instance \cite{FIVEDIM}, \cite{NOVE}, \cite{OITO}, \cite{S2S3}, \cite{ONSPHERES}. Moreover, there has been particular interest in Sasaki-Einstein manifolds recently in string theory and conformal field theories (AdS/CFT correspondence), e.g., \cite{MALDACENA}, \cite{NEWCLASSSASAKI}, \cite{OBSTRUCTION}, \cite{ADSCFT} and the references therein.

With the previous ideas in mind, this work is devoted to study in a broad sense the geometry of homogeneous contact manifolds, i.e., contact manifolds $(M,\eta)$ on which a connected Lie group $G$ acts transitively and effectively as a group of diffeomorphisms which leave $\eta$ invariant. 

In the homogeneous context, we always have $\xi$ regular \cite{BW}, and if we assume that $(M,\eta)$ is compact and simply connected, we can also suppose that $G$ is compact and semisimple \cite{MONT}, \cite{WANG}. Therefore, under the assumption of the Einstein condition in the associated Boothby-Wang fibration $\pi \colon (M,\eta) \to (N,\omega_{N})$, i.e. ${\text{Ric}}(\omega_{N}) = k\omega_{N}$, $k \in \mathbb{Z}_{>0}$, we have

\begin{center}

$N = G^{\mathbb{C}}/P = G/G \cap P,$ \ \ and \ \ $M = Q(K_{N}^{\otimes \frac{1}{I(N)}})$,

\end{center}
where $G^{\mathbb{C}}$ is a complexification of $G$, $P \subset G^{\mathbb{C}}$ is a parabolic Lie subgroup, $I(N)$ is the Fano index of $N$, and $Q(K_{N}^{\otimes \frac{1}{I(N)}})$ is the principal circle bundle associated to the holomorphic line bundle defined by the $I(N)$-root $K_{N}^{\otimes \frac{1}{I(N)}}$ of the canonical bundle. As we see, it is suitable to denote $N = X_{P}$ in order to emphasize the parabolic Lie subgroup $P \subset G^{\mathbb{C}}$ and the underlying parabolic Cartan geometry defined by the pair $(G^{\mathbb{C}},P)$ cf. \cite{PARABOLICTHEORY}.

The description above of compact simply connected contact manifolds leads to the correspondence between parabolic Cartan geometry and homogeneous contact geometry. It is worth pointing out that, additionally, since $X_{P}$ is a K\"{a}hler-Einstein Fano manifold, we have that $Q(K_{X_{P}}^{\otimes \frac{1}{I(X_{P})}})$ is a compact simply connected Sasaki-Einstein manifold, and the associated cone $\mathscr{C}(Q(K_{X_{P}}^{\otimes \frac{1}{I(X_{P})}}))$ is a Calabi-Yau manifold.

By considering the last comments, the goal of this work is to provide a precise description of the relation between homogeneous contact geometry and Lie theory. The main tool to be considered in our approach is the representation theory which underlies the parabolic Cartan geometry of such a pair $(G^{\mathbb{C}},P)$ which defines $X_{P} = G^{\mathbb{C}}/P$.

\subsection{Main results} Our main results can be organized as follows:

\begin{enumerate}

\item Description of contact structure for any compact homogeneous contact manifold;

\item Description of $G$-invariant Sasaki-Einstein structures for certain compact homogeneous contact manifolds;

\item Description of Calabi-Yau metrics on cones with compact homogeneous Sasaki-Einstein manifolds as links of isolated singularities;

\item Description of crepant resolutions of Calabi-Yau cones with certain compact homogeneous Sasaki-Einstein manifolds as links of isolated singularities (via Calabi Ansatz).

\end{enumerate}

It is worth noting that our approach to study homogeneous contact manifolds is based on K\"{a}hler geometry of complex flag manifolds. Thus, the description of the structures listed above are related to the $G$-invariant geometry of flag manifolds in a quite natural manner. 

We also observe that, since every semisimple Lie algebra is given by a direct sum of its simple components, our study of homogeneous contact manifolds reduces to Boothby-Wang fibrations over flag manifolds associated to simple Lie groups. 

The first result listed above is the content of the following theorem.

\begin{theorem-non}
\label{Theo1}
Let $(M,\eta,G)$ be a compact connected homogeneous contact manifold, then $M$ is the principal $S^{1}$-bundle given by the sphere bundle 
\begin{equation}
M = \Big \{ u \in L \ \Big | \ \sqrt{H(u,u)} = 1 \Big\} ,
\end{equation}
for some ample line bundle $L^{-1} \in {\text{Pic}}(X_{P})$, where $X_{P} = G^{\mathbb{C}}/P$ is a flag manifold defined by some parabolic Lie subgroup $P \subset G^{\mathbb{C}}$. Furthermore, if $c_{1}(L^{-1})$ defines a K\"{a}hler-Einstein metric on $X_{P}$, it follows that $M = Q(K_{X_{P}}^{\otimes \frac{\ell}{I(X_{P})}})$, for some $\ell \in \mathbb{Z}_{>0}$, and its contact structure $\eta$ is (locally) given by
\begin{equation}
\label{eqtheo1}
\eta = \displaystyle - \frac{\ell \sqrt{-1}}{2I(X_{P})}\big ( \partial - \overline{\partial} \big )\log \big | \big |s_{U}v_{\delta_{P}}^{+} \big| \big |^{2} + d\theta_{U},
\end{equation}
for some local section $s_{U} \colon U \subset X_{P} \to G^{\mathbb{C}}$, where $v_{\delta_{P}}^{+}$ denotes the highest weight vector of weight $\delta_{P}$ associated to the irreducible $\mathfrak{g}^{\mathbb{C}}$-module $V(\delta_{P})$.
\end{theorem-non}

The result above provides an additional information for Boothby-Wang fibrations over flag manifolds, namely, the expression \ref{eqtheo1} of the associated contact structure (see Remark \ref{contactample}). In fact, it takes into account elements of representation theory of simple Lie algebras which control the K\"{a}hler geometry of the base manifold (transverse space) as well as its projectve algebraic geometry via Kodaira embedding. As we can see, under the assumption of Einstein condition on the induced K\"{a}hler metric on $X_{P}$, every compact homogeneous contact manifold is obtained from the universal covering space

\begin{center}

$\mathcal{Q}_{P}  := Q(K_{X_{P}}^{\otimes \frac{1}{I(X_{P})}}).$

\end{center}
In other words, the circle bundle underlying the compact homogeneous contact manifold, whose the cone \ref{conesymplectic} is Ricci-flat, is given by a principal circle bundle associated to some $\ell$-fold covering (Maslov covering, e.g. \cite{URBANO}), i.e. $M = \mathcal{Q}_{P} / \pi_{1}(M)$, where $ \pi_{1}(M) = \mathbb{Z}_{\ell} \subset {\rm{U}}(1)$ ($\ell$-roots of unity).

The second result of the previous list is concerned to provide a complete description of the invariant Sasaki-Einstein structure which we can endow certain compact homogeneous contact manifolds. Actually, according to \cite{HATAKEYMA}, we can equip a principal circle bundle, defined by a regular compact contact manifold, with a ${\text{K}}$-contact structure \cite{BLAIR}. In the setting of compact homogeneous contact manifolds, since the base manifold associated to the Boothby-Wang fibration is a homogeneous Hodge manifold \cite{BW}, it follows that the induced ${\text{K}}$-contact structure provided in \cite{HATAKEYMA} is in fact a Sasaki structure. Moreover, under the assumption of the Einstein condition in the basic Hodge metric, this Sasaki structure is Sasaki-Einstein. 

Although there are many results in the literature on Sasaki-Einstein manifolds, explicit metrics are rather rare. Our next result provides a constructive method to describe explicitly a huge class of homogeneous Sasaki-Einstein metrics.

\begin{theorem-non}
\label{Theo2}
Let $(M = \mathcal{Q}_{P} / \mathbb{Z}_{\ell},\eta,G)$ be a compact connected homogeneous contact manifold. Then, $(M = \mathcal{Q}_{P} / \mathbb{Z}_{\ell},\eta,G)$ admits a homogeneous Sasaki-Einstein structure $(g_{M}, \phi,\xi = \frac{\ell(n+1)}{I(X_{P})}\frac{\partial}{\partial \theta}, \frac{I(X_{P})}{\ell(n+1)}\eta )$, such that 
\begin{equation}
\label{metricsasakieinstein}
g_{M} = \displaystyle \frac{I(X_{P})}{\ell(n+1)} \Bigg ( \frac{1}{2}d \eta ({\rm{id}} \otimes \phi)  + \frac{I(X_{P})}{\ell(n+1)}\eta \otimes \eta \Bigg ),
\end{equation}
where 
\begin{center}

$\eta = \displaystyle - \frac{\ell\sqrt{-1}}{2I(X_{P})}\big ( \partial - \overline{\partial} \big )\log \big | \big |s_{U}v_{\delta_{P}}^{+} \big| \big |^{2} + d\theta_{U},$
\end{center}
for some local section $s_{U} \colon U \subset X_{P} \to G^{\mathbb{C}}$, where $v_{\delta_{P}}^{+}$ denotes the highest weight vector of weight $\delta_{P}$ associated to the irreducible $\mathfrak{g}^{\mathbb{C}}$-module $V(\delta_{P})$. Furthermore, we also have $\phi \in {\text{End}}(TM)$ completely determined by the invariant complex structure of $X_{P}$ and the horizontal lift of the Cartan-Ehresmann connection $ \frac{I(X_{P})\sqrt{-1}}{\ell(n+1)}\eta \in \Omega^{1}(M;\mathfrak{u}(1))$.
\end{theorem-non}

 It is worth pointing out that the metric \ref{metricsasakieinstein} is a prototype which allows us to get a huge class of constructive explicit examples of invariant Einstein metrics with positive scalar curvature. These metrics are obtained via Kaluza-Klein ansatz in the setting of Boothby-Wang fibrations over flag manifolds, see for instance \cite{KOBAYASHI}, \cite{WANGZILLER}, \cite[Section 5.4]{FALCITELLI}. 

Our third result is concerned to describe the Calabi-Yau structure which we have associated to the Riemannian cones (symplectizations) over homogeneous Sasaki-Einstein manifolds.

\begin{theorem-non}
\label{Theo3}
Let $(M,\eta,G)$ be a compact homogeneous contact manifold such that $M = \mathcal{Q}_{P} / \mathbb{Z}_{\ell}$, for some parabolic Lie subgroup $P \subset G^{\mathbb{C}}$. Then, the cone $\mathscr{C}(M)$ admits a Calabi-Yau metric $\omega_{\mathscr{C}} = \frac{1}{2}d\Phi$ such that 

\begin{equation}
\Phi = \displaystyle - \frac{r^{2} \sqrt{-1}}{2(n+1)}\big ( \partial - \overline{\partial} \big )\log \big | \big |s_{U}v_{\delta_{P}}^{+} \big| \big |^{2} + \frac{r^{2}I(X_{P})}{\ell(n+1)}d\theta_{U},    
\end{equation}
for some local section $s_{U} \colon U \subset X_{P} \to G^{\mathbb{C}}$, where $v_{\delta_{P}}^{+}$ denotes the highest weight vector of weight $\delta_{P}$ associated to the irreducible $\mathfrak{g}^{\mathbb{C}}$-module $V(\delta_{P})$.
\end{theorem-non}

The result above provides a constructive method to obtain explicit examples of Ricci-flat K\"{a}hler metrics on Riemannian cones. Since there are no explicit Ricci-flat metrics known on compact Calabi-Yau manifolds, metric cones over Sasaki-Einstein spaces provide a testing ground for Calabi-Yau compactifications.

As mentioned before, due to their importance in the AdS/CFT correspondence, Sasaki-Einstein manifolds have been widely studied by mathematicians and physicists. Inspired by some of these applications, and by following \cite{COLON}, \cite{RESOLUTIONCOMPSUP}, \cite{GOTO}, our forth result provides a constructive method to describe the resolution of Calabi-Yau cones, with certain homogeneous Sasaki-Einstein manifolds realized as links of isolated singularities, by means of the Cartan-Remmert reduction \cite{GRAUERT} and the Calabi Ansatz technique \cite{CALABIANSATZ}. The result is precisely the following. 

\begin{theorem-non}
\label{Theo4}
Let $(M,\eta,G)$ be a compact homogeneous contact manifold such that $M = \mathcal{Q}_{P} / \mathbb{Z}_{I(X_{P})}$, i.e., $M = Q(K_{X_{P}})$ for some parabolic Lie subgroup $P \subset G^{\mathbb{C}}$. Then, the Cartan-Remmert reduction $\mathscr{R} \colon K_{X_{P}} \to Y = {\mathscr{C}}(M) \cup \{o\}$ provides a crepant resolution for the Calabi-Yau cone $({\mathscr{C}}(M), \omega_{\mathscr{C}})$ such that the complete Calabi-Yau metric $\omega_{CY}$ on $K_{X_{P}}$, defined by the Calabi Ansatz
\begin{equation}
\label{ansatzcompact}
\omega_{CY} = \displaystyle (2\pi r^{2} + C)^{\frac{1}{n+1}} \Bigg (\omega_{X_{P}} - \frac{\sqrt{-1}}{n+1} \frac{( db_{U} + b_{U} A_{U})\wedge ( d\overline{b}_{U} + \overline{b}_{U}  \overline{A}_{U})}{(2\pi r^{2} + C)} \Bigg ),
\end{equation}
provides a resolution for the singular cone metric defined on $Y={\mathscr{C}}(M) \cup \{o\}$ by
\begin{equation}  
\omega_{\mathscr{C}} = \displaystyle r dr \wedge \Bigg (\frac{\sqrt{-1}(\overline{A}_{U} - A_{U})}{2(n+1)} + \frac{d\theta_{U}}{n+1} \Bigg )  + \frac{\pi r^{2}}{n+1} \omega_{X_{P}},      
\end{equation}
such that $\omega_{X_{P}} =  -\frac{\sqrt{-1}}{2\pi}dA_{U}$ and 
\begin{equation}
A_{U} =  \displaystyle \partial \log \big | \big |s_{U}v_{\delta_{P}}^{+} \big| \big |^{2},        
\end{equation}
for some local section $s_{U} \colon U \subset X_{P} \to G^{\mathbb{C}}$, where $v_{\delta_{P}}^{+}$ denotes the highest weight vector of weight $\delta_{P}$ associated to the irreducible $\mathfrak{g}^{\mathbb{C}}$-module $V(\delta_{P})$. Furthermore, there is a Ricci-flat complete K\"{a}hler metric for every K\"{a}hler class of $K_{X_{P}}$.

\end{theorem-non}

\begin{remark}
It is important to observe that the expression provided in Equation \ref{ansatzcompact} corresponds to an element in the compactly supported cohomology group of $K_{X_{P}}$ given by the Calabi Ansatz \cite{CALABIANSATZ}. In general, if $b_{2}(X_{P}) > 1$, by following \cite[Theorem 5.1]{GOTO}, we can find a Ricci-flat conical K\"{a}hler class which does not belong to the compactly supported cohomology group of $K_{X_{P}}$. Thus, the last statement of Theorem \ref{Theo4} is a particular consequence of the existence result provided in \cite{GOTO}.
\end{remark}

This last result allows us to describe a huge class of new explicit examples for the existence part of the conjecture introduced in \cite{ADSCFT}. Actually, the result above provides a constructive method to describe concrete realizations for \cite[Theorem 5.1]{GOTO}, see also \cite{RESOLUTIONCOMPSUP}, and \cite[Example 4.1]{COLON}. 

It is worth pointing out that, besides the results above, in this work we also provide a detailed exposition about connections and curvature on principal circle bundles and holomorphic line bundles over flag manifolds. We also provide several examples for each result in order to illustrate their direct applications.   

\subsection{Outline of the paper} The content and main ideas  of this paper are organized as follows:

In Section \ref{Sec2}, we cover the basic material about contact manifolds, Sasaki manifolds and their symplectizations. We also establish some basic notations and conventions. In Section \ref{sec3}, we describe how to apply elements of representation theory of simple Lie algebras in order to describe the Chern connection and the Cartan-Ehresmann connection, respectively, for holomorphic line bundles and principal ${\rm{U}}(1)$-bundles over generalized complex flag manifolds. In Section \ref{sec4}, we apply the machinery developed in Section \ref{sec3} to describe Sasaki-Einstein structures on homogeneous Sasaki manifolds as well as the Ricci-flat K\"{a}hler metrics on their symplectizations. The goals are to prove Theorem \ref{Theo1}, Theorem \ref{Theo2}, and Theorem \ref{Theo3}. After these, in Section \ref{sec5}, we use the content developed throughout the paper to provide a huge class of examples of crepant resolutions of Calabi-Yau cones with certain homogeneous Sasaki-Einstein manifolds realized as links of isolated singularities. The main goal in this last section is to prove Theorem \ref{Theo4}. 

\section{Generalities on contact manifolds}
\label{Sec2}

In this section we shall cover the basic generalities about contact geometry, Sasakian geometry and some related topics. After to discuss the relation between Sasaki-Einstein geometry and positive scalar K\"{a}hler-Einstein geometry, we provide a complete description of homogeneous contact manifolds. Proofs of the results presented in this section can be found in \cite{BOYERGALICKI}, \cite{BLAIR}, \cite{BW}, \cite{HATAKEYMA}.  

\begin{definition}
Let $M$ be a smooth connected manifold of dimension $2n + 1$. A contact structure on $M$ is a $1$-form $\eta \in \Omega^{1}(M)$ which satisfies $\eta \wedge (d\eta)^{n} \neq 0$.
\end{definition}

When a smooth connected $(2n+1)$-dimensional manifold $M$ admits a contact structure $\eta \in \Omega^{1}(M)$ the pair $(M,\eta)$ is called contact manifold. Given a contact manifold $(M,\eta)$, at each point $p \in M$ we have from the condition $\eta \wedge (d\eta)^{n} \neq 0$ that $(d\eta)_{p}$ is a quadratic form of rank $2n$ in the Grassman algebra $\bigwedge T_{p}^{\ast}M$, thus we obtain 
\begin{center}
$T_{p}M = \mathscr{D}_{p} \oplus \mathcal{F}_{\eta_{p}},$
\end{center}
such that $\mathscr{D} = \ker(\eta)$, and
\begin{center}
$p \in M \mapsto \mathcal{F}_{\eta_{p}} = \bigg\{ X \in T_{p}M \ \bigg | \ \iota_{X}(d\eta)_{p} = 0 \bigg \} \subset T_{p}M,$ 
\end{center}
defines the characteristic foliation.

Let $(M,\eta)$ be a contact manifold. From the condition $\eta \wedge (d\eta)^{n} \neq 0$, we have that there exists $\xi \colon C^{\infty}(M) \to C^{\infty}(M)$, such that 
\begin{equation}
\label{derivation}
df \wedge (d\eta)^{n} = \xi(f)\eta \wedge (d\eta)^{n},
\end{equation}
$\forall f \in  C^{\infty}(M)$. From this, a straightforward computation shows that $\xi$ is a $\mathbb{R}$-linear derivation on $C^{\infty}(M)$, hence $\xi \in \Gamma(TM)$. Now, from Equation \ref{derivation} we can show that 
\begin{center}
$\beta(\xi)\eta \wedge (d\eta)^{n} = \beta \wedge (d\eta)^{n}$, 
\end{center}
$\forall \beta \in \Omega^{1}(M)$. By using the last fact above we have
\begin{center}
$\eta(\xi)\eta \wedge (d\eta)^{n} = \eta \wedge (d\eta)^{n}$, \ \ and \ \ $d\eta(X,\xi)\eta \wedge (d\eta)^{n} = \displaystyle \frac{1}{n+1}\iota_{X}(d\eta)^{n+1}=0,$ 
\end{center}
for all $X \in \Gamma(TM)$. Therefore, we obtain $\xi \in \Gamma(TM)$ which satisfies
\begin{equation}
\eta(\xi) = 1, \ \ {\text{and}} \ \ d\eta(\xi,\cdot) = 0,
\end{equation}
see for instance \cite{TAKIZAWA}. The vector field $\xi$ is called the characteristic vector field, or Reeb vector field, of the contact structure $\eta$.

A contact structure $\eta \in \Omega^{1}(M)$ is regular if the associated characteristic vector field $\xi \in \Gamma(TM)$ is regular, namely, if every point of the manifold has a neighborhood such that any integral curve of the vector field passing through the neighborhood passes through only once (cf. \cite{PALAIS}). In this case $(M,\eta)$ is called a regular contact manifold. 

Given a regular compact contact manifold $(M,\eta)$, we can suppose without loss of generality that the associated characteristic vector field $\xi \in \Gamma(TM)$ generates a ${\rm{U}}(1)$-action on $M$, see for instance \cite[Theorem 1]{BW}. Therefore, we have the following well-known result.

\begin{theorem}[Boothby-Wang, \cite{BW}]
\label{BWT}
Let $\eta$ be a regular contact structure on a compact smooth manifold $M$, then

\begin{enumerate}

    \item $M$ is a principal ${\rm{U}}(1)$-bundle over $N = M/{\rm{U}}(1)$,\\
    
    \item $\eta' = \sqrt{-1}\eta$ defines a connection on this bundle, and\\
    
    \item the manifold $N$ is a symplectic manifold whose the symplectic form $\omega$ determines an integral cocycle on $N$ which satisfies $d\eta = \pi^{\ast}\omega$, where $\pi \colon M \to N$.
\end{enumerate}
\end{theorem}
The following result states that, in fact, the converse of Theorem \ref{BWT} is also true. 

\begin{theorem}[Kobayashi, \cite{TOROIDAL}]
\label{CONVBW}
Let $(N,\omega_{N})$ be a symplectic manifold such that $[\omega_{N}] \in H^{2}(N,\mathbb{Z})$, then there exists a principal ${\rm{U}}(1)$-bundle $\pi \colon M \to N$ with a connection $1$-form $\eta' \in \Omega^{1}(M;\mathfrak{u}(1))$ which determines a regular contact structure $\eta = -\sqrt{-1}\eta'$ on $M$ satisfying $d\eta = \pi^{\ast}\omega_{N}$.

\end{theorem}

We are particularly interested in the following setting. 

\begin{definition} A contact manifold $(M,\eta)$ is said to be homogeneous if there is a connected Lie group $G$ acting transitively and
 effectively as a group of diffeomorphisms on $M$ which leave $\eta$ invariant, i.e. $g^{\ast}\eta = \eta$, $\forall g \in G$.

\end{definition}

We denote a homogeneous contact manifold by $(M,\eta,G)$. From this, we have the following important result of Boothby and Wang \cite{BW}.

\begin{theorem}[Boothby-Wang, \cite{BW}]
\label{BWHOMO}
Let $(M,\eta,G)$ be a homogeneous contact manifold. Then the contact form $\eta$ is regular. Moreover, $M = G/K$ is a fiber bundle over $G/H_{0}K$ with fiber $H_{0}K/K$, where $H_{0}$ is the connected component of a $1$-dimensional Lie group $H$, and $H_{0}$ is either diffeomorphic to ${\rm{U}}(1)$ or $\mathbb{R}$.

\end{theorem}

If we suppose that $(M,\eta,G)$ is compact and simply connected, then according to \cite{MONT}, without loss of generality, we can suppose that $G$ is compact. Furthermore, according to \cite{WANG} we can in fact suppose that $G$ is a semisimple Lie group. Hence, we have the following theorem.

\begin{theorem}[Boothby-Wang, \cite{BW}]
\label{BWhomo}
Let $(M,\eta,G)$ be a compact simply connected homogeneous contact manifold. Then $M$ is a circle bundle over a complex flag manifold $(N,\omega_{N})$ such that $\omega_{N}$ defines a $G$-invariant Hodge metric which satisfies $d\eta = \pi^{\ast}\omega_{N}$, where $\pi \colon M \to (N,\omega_{N})$.
\end{theorem}

Since every complex flag manifold is a Hodge manifold, from Theorem \ref{CONVBW} it implies that we can always associate to a complex flag manifold a contact manifold. Before we describe how to construct this contact manifold, let us introduce some basic definitions and results related to contact metric structures.

\begin{definition}
\label{ContMetric}
Let $(M,g_{M})$ be a Riemannian manifold of dimension $2n+1$. A contact metric structure on $(M,g_{M})$ is a triple $(\phi,\xi,\eta)$ where $\phi$ is a $(1,1)$-tensor, $\xi$ is a vector field, and $\eta$ is a $1$-form such that 

\begin{enumerate}

    \item $\eta \wedge (d\eta)^{n} \neq 0$, \  $\eta(\xi) = 1$,
    
    \item $\phi \circ \phi = - {\rm{id}} + \eta \otimes \xi$,
    
    \item $g_{M}(\phi \otimes \phi) = g_{M} - \eta \otimes \eta$,
    
    \item $d\eta = 2g_{M}(\phi \otimes {\rm{id}})$.
    
\end{enumerate}
 
\end{definition}

\begin{remark} 
Notice that the first condition in the definition above shows us that every contact metric structure defines a contact structure. Unless otherwise stated, in what follows we shall suppose that this contact structure is a regular contact structure. Many of the results which we will cover in this section can be performed for quasi-regular contact structures on which the characteristic foliation has compact leaves. In this latter situation, the space of leaves  $N = M/\mathcal{F}_{\eta}$ has an orbifold structure. For equivalent results in the quasi-regular case, see for instance \cite{BOYERGALICKI}.

\end{remark}

We denote a contact metric structure on $M$ by $(g_{M},\phi,\xi,\eta)$. From this, we have the following definition.

\begin{definition}

A contact metric structure $(g_{M},\phi,\xi,\eta)$ is called ${\text{K}}$-contact if $\mathscr{L}_{\xi}g_{M} = 0$, i.e. if $\xi \in \Gamma(TM)$ is a Killing vector field.

\end{definition}

In the setting of ${\text{K}}$-contact structures there is a special class which is defined as follows.

\begin{definition}

A ${\text{K}}$-contact structure $(g_{M},\phi,\xi,\eta)$ on a smooth manifold $M$ is called Sasakian if 
\begin{equation}
\label{sasakicondition}
\big [ \phi , \phi \big ] + d \eta \otimes \xi = 0,
\end{equation}
where 
\begin{center}
$\big [ \phi , \phi \big ](X,Y) := \phi^{2}\big [ X,Y\big ] + \big [ \phi X, \phi Y\big ] - \phi \big [\phi X,Y \big] - \phi \big [X,\phi Y \big],$
\end{center}
for every $X,Y \in \Gamma(TM)$. A Sasaki manifold is a Riemannian manifold $(M,g_{M})$ with a $K$-contact structure $(g_{M},\phi,\xi,\eta)$ which satisfies \ref{sasakicondition}. 
\end{definition}

There are two alternative characterizations for Sasaki manifolds, the first one can be described as follows. Given a ${\text{K}}$-contact structure $(g_{M},\phi,\xi,\eta)$ on a smooth manifold $M$, we can consider the manifold defined by its cone 

\begin{center}
$\mathscr{C}(M) = \mathbb{R}^{+} \times M.$
\end{center}
By taking the coordinate $r$ on $\mathbb{R}^{+}$ we can define the warped product Riemannian metric
\begin{equation}
g_{\mathscr{C}} = dr \otimes dr + r^{2}g_{M},
\end{equation}
furthermore, from $(\phi,\xi,\eta)$ we have an almost-complex structure defined on $\mathscr{C}(M)$ by 
\begin{equation}
\label{complexcone}
J_{\mathscr{C}}(Y) = \phi(Y) - \eta(Y)r \displaystyle \frac{d}{dr}, \ \ \ \ \ J_{\mathscr{C}}(r \displaystyle \frac{d}{d r}) = \xi.
\end{equation}

From the last comments we have the following characterization for Sasaki manifolds. 

\begin{definition}
\label{sasakikahler}
A contact metric structure $(g_{M},\phi,\xi,\eta)$ on a smooth manifold $M$ is called Sasaki if $(\mathscr{C}(M),g_{\mathscr{C}},J_{\mathscr{C}})$ is a K\"{a}hler manifold.
\end{definition}

\begin{remark}
Note that in the last definition it was not required to $\xi$ being a Killing vector field. Actually, the integrability of the complex structure $J_{\mathscr{C}}$ implies the Sasaki condition \ref{sasakicondition}, which in turn implies that $\xi$ is Killing \cite[Theorem 6.2]{BLAIR}. The definition above is perhaps the closest to the original definition of Sasaki \cite{SASAKI}.
\end{remark}

In the setting above we have the K\"{a}hler structure on $\mathscr{C}(M)$ defined by
\begin{center}
$\omega_{\mathscr{C}} = g_{\mathscr{C}}(J_{\mathscr{C}} \otimes {\rm{id}}).$
\end{center}
In general, we can always associate to any contact manifold $(M,\eta)$ a symplectic manifold $(\mathscr{C}(M),\omega_{\mathscr{C}} = \frac{d(r^{2}\eta)}{2})$. This last manifold is called symplectization of $(M,\eta)$. When $(M,\eta)$ is a compact regular contact manifold, from Theorem \ref{BWT} we have that $(M,\eta)$ is a principal ${\rm{U}}(1)$-bundle over a symplectic manifold, in this case we can endow $(M,\eta)$ with a ${\text{K}}$-contact structure $(g_{M},\phi,\xi,\eta)$, see for instance \cite{HATAKEYMA}.

The second way to characterize Sasaki manifolds is by means of the transverse geometry of $N = M/\mathcal{F}_{\eta}$. In fact, a straightforward computation shows that the structure tensors $(g_{M},\phi,\xi,\eta)$ on a smooth compact $K$-contact manifold $M$ induce an almost-K\"{a}hler structure $(\omega_{N},J)$ on $N$, where 

\begin{center}

$\displaystyle \pi^{\ast}\omega_{N} = \frac{d\eta}{2}$ \ \ and \ \ $J = \phi|_{\mathscr{D}},$

\end{center}
here we consider the identification $TN \cong \mathscr{D}$, and $\pi \colon M \to M/\mathcal{F}_{\eta}$. From these we can show that Equation \ref{sasakicondition} is equivalent to $N_{J} \equiv 0$, where $N_{J}$ is the Nijenhuis tensor associated to $J$. Therefore, the Sasaki condition is equivalent to $(N,\omega_{N},J)$ being a K\"{a}hler manifold, e.g. \cite{HATAKEYMA}. For the case when $M$ is a non-regular $K$-contact, the Sasaki condition is equivalent to $(\phi|_{\mathscr{D}},\frac{d\eta}{2}|_{\mathscr{D}},\mathscr{D})$ being a K\"{a}hler foliation, see for instance \cite{SPARKS}, \cite[Corollary 6.5.11]{BOYERGALICKI}.

\begin{remark}
\label{riccitensorsasaki}
It is worth pointing out that if $(\phi,\xi,\eta)$ is a Sasakian structure on a complete Riemannian manifold $(M,g_{M})$, if we denote by $R^{\nabla}$ the curvature tensor associated to the Levi-Civita connection $\nabla$ of $g_{M}$, then we have

\begin{center}
$R^{\nabla}(X,Y)\xi = \eta(Y)X -\eta(X)Y$,
\end{center}
$\forall X,Y \in \Gamma(TM)$. Moreover, we can show that 

\begin{itemize}

\item ${\text{Ric}}_{M}(X,\xi) = 2n \eta(X)$, $\forall X \in TM$.

\end{itemize}

In particular, the scalar curvature $S_{g_{M}}$ of a Sasaki-Einstein manifold of dimension $2n+1$ is $S_{g_{M}} = 2n(2n+1)$, see for instance \cite{BAUM}, which implies from Myers's theorem that $M$ is compact. In this last case we have
\begin{itemize}

\item ${\text{Ric}}_{M}(X,Y) = {\text{Ric}}_{N}(X,Y) - 2g_{M}(X,Y)$, $\forall X,Y \in \mathscr{D} \cong TN$.

\end{itemize}
Thus, we see that $(M,g_{M})$ is a complete Sasaki-Einstein manifold if and only if $(N,g_{N})$ is K\"{a}hler-Einstein Fano, namely 
\begin{center}
${\text{Ric}}_{N}(X,Y) = 2(n+1)g_{N}(X,Y)$, $\forall X,Y \in TN$. 
\end{center}
\end{remark}

The remark above shows how K\"{a}hler-Einstein geometry with positive scalar curvature arises from Sasaki-Einstein geometry. In the next example below we show how to construct Sasaki-Einstein manifolds with prescribed transverse geometry, in other words, we shall describe how Sasaki-Einstein geometry arises from K\"{a}hler-Einstein geometry with positive scalar curvature.

\begin{example}
\label{EXAMPLESASAKIAN}
Let $(N,\omega_{N},g_{N},J)$ be a K\"{a}hler-Einstein Fano manifold with scalar curvature $S_{g_{N}} = 4n(n+1)$, where $\dim_{\mathbb{C}}(N) = n$. Consider $I(N) \in \mathbb{Z}_{+}$ as being the maximal integer such that $\frac{1}{I(N)}c_{1}(N) \in H^{2}(N,\mathbb{Z})$, i.e., the Fano index of $N$, here $c_{1}(N)$ denotes the first Chern class of $N$. If we take the principal ${\rm{U}}(1)$-bundle $\pi \colon M \to N$ with Euler class $\mathrm{e}(M) = -\frac{1}{I(N)}c_{1}(N)$, and connection $\eta'$ satisfying 

$$d\eta' = \displaystyle \frac{2(n+1)}{I(N)}\sqrt{-1}\pi^{\ast}\omega_{N},$$
then we can define a Riemannian metric on $M$ by setting

$$g_{M} = \pi^{\ast}g_{N} - \displaystyle \frac{I(N)^{2}}{(n+1)^{2}} \eta' \otimes \eta'.$$

From this, the Riemannian manifold $(M,g_{M})$ can be endowed with a Sasakian structure $(g_{M},\phi,\xi,\eta)$ defined as follows:

\begin{center}
    
$\eta := \displaystyle \frac{I(N)}{(n+1)\sqrt{-1}} \eta'$, \ \ \ \ $g_{M}(\xi,\cdot) = \eta,$ \ \ \ \ $ \phi(X) := \begin{cases}
    (J\pi_{\ast}X)^{H}, \ \ \ {\text{if}}  \ \ X \bot \xi.\\
    \ \ \ \ \  0 \  \  \ \ \ \ \ ,  \ \  \ {\text{if}} \ \ X \parallel \xi.                    \\
  \end{cases}$
    
\end{center}
here we denote by $(J\pi_{\ast}X)^{H}$ the horizontal lift of $J\pi_{\ast}X$ relative to $\eta'$, $\forall X \in \Gamma(TM)$. The proof that $(g_{M},\phi,\xi,\eta)$ defines a Sasakian structure on $M$ follows directly from the definition of the structure tensors $(g_{M},\phi,\xi,\eta)$, we shall cover the details of the proof later, the reader also can look at \cite[Example 1, p. 84]{BAUM}, \cite{HATAKEYMA}, \cite[Theorem 6]{MORIMOTO}. Notice that the metric $g_{M}$ also can be written as 
\begin{equation}
g_{M} = \pi^{\ast}\omega_{N}({\rm{id}} \otimes \phi) + \eta \otimes \eta = \displaystyle \frac{1}{2} d \eta ({\rm{id}} \otimes \phi) + \eta \otimes \eta,    
\end{equation}
where $g_{N} = \omega_{N}({\rm{id}}\otimes J)$. 

Now, since ${\text{Ric}}(\omega_{N}) = 2(n+1)\omega_{N}$, we have that $(M,g_{M})$ is Sasaki-Einstein, see Remark \ref{riccitensorsasaki}, notice that from Myers's theorem we have that $M$ is compact if $g_{M}$ is complete. Moreover, from the exact sequence

\begin{center}
\begin{tikzcd}

\cdots\arrow[r] & \pi_{2}(N) \arrow[r,"\Delta"] & \pi_{1}({\rm{U}}(1)) \arrow[r] & \pi_{1}(M) \arrow[r] & \pi_{1}(N) \arrow[r] & 0, 

\end{tikzcd}
\end{center}
since $N$ is simply connected \cite{KAHLERSIMPLY}, we have that $\pi_{1}(M)$ is trivial or a cyclic group. However, once $M$ is given by a principal circle bundle defined by an indivisible integral class, it follows from  the Thom-Gysin sequence associated to the principal ${\rm{U}}(1)$-bundle ${\rm{U}}(1) \hookrightarrow M \to N$ that $M$ is simply connected, e.g. \cite[Example 1, p. 84]{BAUM}.
\end{example}

The construction above also can be understood in terms of holomorphic line bundles in the following way. Given a K\"{a}hler-Einstein Fano manifold $(N,\omega_{N},g_{N},J)$ with scalar curvature $S_{g_{N}} = 4n(n+1)$. If we take the line bundle $L \in {\text{Pic}}(N)$ as being the $I(N)$-root of $K_{N} = \bigwedge ^{n,0}(N)$, i.e.

\begin{center}
$L = K_{N}^{\otimes \frac{1}{I(N)}}$, 
\end{center}
by fixing a Hermitian structure $H$ on $L$, we can consider the circle bundle defined by
\begin{center}
$Q(L) = \Big \{ u \in L \ \big | \ \sqrt{H(u,u)} = 1 \Big\}$.
\end{center}
Since we suppose $S_{g_{N}} = 4n(n+1)$, it follows that

\begin{center}

${\text{Ric}}(\omega_{N}) = 2(n+1)\omega_{N}$.    
    
\end{center}
We can take a Hermitian connection $\nabla^{L}$ on $L$ such that its curvature $F_{\nabla^{L}}$ satisfies 

\begin{center}

$\displaystyle \frac{\sqrt{-1}}{2\pi}F_{\nabla^{L}} = - \frac{(n+1)}{\pi I(N)}\omega_{N}$.    
    
\end{center}
Since $\nabla^{L}$ induces a connection $\eta'$ on $Q(L)$ which satisfies $d\eta' = \pi^{\ast}F_{\nabla^{L}}$, we have that $M = Q(L)$ is the manifold described in Example \ref{EXAMPLESASAKIAN}.

The next proposition states that from the example above we can obtain the description of all compact simply connected homogeneous contact manifolds by using holomorphic line bundles.

\begin{proposition}
\label{CONTACTSIMPLY}
Let $(M,\eta,G)$ be a compact simply connected homogeneous contact manifold, then $M$ is a principal ${\rm{U}}(1)$-bundle $\pi \colon Q(L) \to N$ over a complex flag manifold, for some ample line bundle $L^{-1} \in {\text{Pic}}(N)$. Moreover, if $d\eta = \pi^{\ast}\omega_{N}$, where $[\omega_{N}] \in H^{2}(N,\mathbb{Z})$ is a $G$-invariant K\"{a}hler-Einstein metric, it follows that
\begin{center}
$L = K_{N}^{\otimes \frac{1}{I(N)}}$,
\end{center}
where $I(N)$ is the fano index of $N$. 
\end{proposition}

\begin{proof}
From Theorem \ref{BWhomo} we have that $M = Q(L)$, for some line bundle $L \in {\text{Pic}}(N)$, such that $d\eta = \pi^{\ast}\omega_{N}$, where $[\omega_{N}] \in H^{2}(N,\mathbb{Z})$ is a $G$-invariant Hodge metric, so $L$ is ample. Now, suppose that $[\omega_{N}] \in H^{2}(N,\mathbb{Z})$ is a $G$-invariant K\"{a}hler-Einstein metric, i.e., ${\text{Ric}}(\omega_{N}) = k\omega_{N}$, for some $k \in \mathbb{Z}_{>0}$. From this, we take the connection on $M = Q(L)$ defined by

\begin{center}

$\eta' = \sqrt{-1} \eta$.    
    
\end{center}
Since the curvature $F_{\nabla^{L}}$ of the connection $\nabla^{L}$ induced by $\eta'$ on $L$ satisfies $d\eta' = \pi^{\ast}F_{\nabla^{L}}$, a straightforward computation shows that  
\begin{center}
$\displaystyle \mathrm{e}(Q(L)) = \Big [ \frac{\sqrt{-1}}{2\pi} F_{\nabla^{L}}\Big ] = - \ell \frac{c_{1}(N)}{I(N)}$,
\end{center}
thus $c_{1}(L)$ is a multiple of the indivisible class $\frac{1}{I(N)}c_{1}(N)$. Now, since $M = Q(L)$ is a simply connected manifold, it follows from \cite{TOROIDAL} that $L = K_{N}^{\otimes \frac{1}{I(N)}}$.
\end{proof}

\begin{remark}
\label{amplecase}
Notice that the proposition above tells us that every compact simply connected homogeneous contact manifold is in fact a Sasaki manifold. Actually, for any circle bundle $Q(L)$ defined by an ample line bundle $L^{-1} \in {\text{Pic}}(N)$, we can construct a ${\text{K}}$-contact structure by proceeding similarly as in Example \ref{EXAMPLESASAKIAN}. This ${\text{K}}$-contact structure is in fact Sasaki, see \cite[Theorema 2]{HATAKEYMA}, and the Einstein condition for this Sasaki structure is equivalent to the proportionality of $L$ and $K_{N}$.
\end{remark}

\begin{definition}
A Sasakian manifold $(M,g_{M})$ with structure tensors $(\phi,\xi,\eta)$ is said to be homogeneous  if there is a connected Lie group $G$ acting transitively and effectively as a group of isometries on $M$ preserving the Sasakian structure.
\end{definition}

The next result together with the last proposition allows us to describe all compact homogeneous contact manifolds, the proof for the result  below can be found in \cite{BOYERGALICKI}.

\begin{theorem}
\label{contacthomogeneousclass}
Let $(M,\eta,G)$ be a compact homogeneous contact manifold. Then:
\begin{enumerate}

    \item $M$ admits a homogeneous Sasakian structure with contact form $\eta$,\\
    
    \item $M$ is a non-trivial circle bundle over a complex flag manifold,\\
    
    \item $M$ has finite fundamental group, and the universal cover $\widetilde{M}$ of $M$ is compact with a homogeneous Sasakian structure. 
\end{enumerate}

\end{theorem}

The result above provides a complete description of a compact homogeneous contact manifold $(M,\eta,G)$ as being a quotient space 

\begin{center}
    
$M = \widetilde{M}/\Gamma$,     
    
\end{center}
where $\widetilde{M} = Q(L)$, for some ample line bundle $L \in {\text{Pic}}(N)$, and $\Gamma = \mathbb{Z}_{\ell} \subset {\rm{U}}(1) \hookrightarrow \widetilde{M}$ is a cyclic group given by the deck transformations of the universal cover $\widetilde{M}$, see for instance \cite{BOYERGALICKI}. From this, we have 

\begin{center}

$M = Q(L)/\mathbb{Z}_{\ell} = \ell \cdot Q(L) = \underbrace{Q(L) + \cdots + Q(L)}_{\ell-{\text{times}}}$,    

\end{center}
see for instance \cite{TOROIDAL}, \cite[Chapter 2]{BLAIR}. In this paper we also shall use the notation $M = Q(L^{\otimes \ell})$.\\

As we have seen, in order to describe the compact homogeneous contact manifolds which defines homogeneous Sasaki-Einstein manifolds we need to understand the homogeneous contact manifold determined by the circle bundle $Q(L)$, where $L$ is the line bundle
\begin{equation}
\label{maslovcovering}
L = K_{N}^{\otimes \frac{1}{I(N)}},
\end{equation}
over a complex flag manifold $(N,\omega_{N})$. Hence, our main task in the next section will be to provide a complete description of holomorphic line bundles and its associated circle bundles over complex flag manifolds.

%%%%%%%%%%%%%%%%%%%%%%%%%%%%%%%%%%%%%%%%%%%%%%%%%%%%%%%%%%%%%%%%%%%%%%%%%%%%%%%%%%%%%%%%%%%%%%%%%%%%%%%%%%%%%%%%%%%%%%%%%%%%%%%%%%%%%%%%%%%%%%%%%%%%%%%%%%%%%%%%%%%%%%%%%%%%%%%%%%%%%%%%%%%%%%%%%%%%%%%%%%%%%%%%%%%%%%%%%%%%%%%%%%%%%%%%%%%%%%%%%%%%%%%%%%%%%%%%%%%%%%%%%%%%%%%%%%%%%%%%%%%%%%%%%%%%%%%%%%%%%%%%%%%%%%%%%%

\section{Line bundles and circle bundles over complex flag manifolds} 
\label{sec3}

This section is devoted to provide some basic results about holomorphic line bundles and principal circle bundles over flag manifolds. The subjects will be presented as follow:

In Subsection \ref{subsec3.1}, we describe how we can compute the Chern class for homolorphic line bundles over flag manifolds. The main idea is to find a suitable Chern connection for homogeneous holomorphic line bundles. In Subsection \ref{subsec3.2}, we describe how we can compute the Cartan-Ehresmann connection (gauge field) for principal circle bundles over flag manifolds. The main idea is to use the characterization of holomorphic line bundles as associated bundles of principal ${\rm{U}}(1)$-bundles. In Subsection \ref{subsec3.3}, we provide concrete examples which illustrate the content developed in the previous subsections. The main references for this section are \cite{PARABOLICTHEORY}, \cite{AZAD}, \cite{EDER}, \cite{TOROIDAL}, \cite[Chapter 2]{BLAIR}, \cite[Appendix D.1]{EDERTHESIS}. 

\subsection{Line bundles over complex flag manifolds}
\label{subsec3.1}
We start by collecting some basic facts about simple Lie algebras and simple Lie groups. Let $\mathfrak{g}^{\mathbb{C}}$ be a complex simple Lie algebra, by fixing a  Cartan subalgebra $\mathfrak{h}$ and a simple root system $\Sigma \subset \mathfrak{h}^{\ast}$, we have a decomposition of $\mathfrak{g}^{\mathbb{C}}$ given by
\begin{center}
$\mathfrak{g}^{\mathbb{C}} = \mathfrak{n}^{-} \oplus \mathfrak{h} \oplus \mathfrak{n}^{+}$, 
\end{center}
where $\mathfrak{n}^{-} = \sum_{\alpha \in \Pi^{-}}\mathfrak{g}_{\alpha}$ and $\mathfrak{n}^{+} = \sum_{\alpha \in \Pi^{+}}\mathfrak{g}_{\alpha}$, here we denote by $\Pi = \Pi^{+} \cup \Pi^{-}$ the root system associated to the simple root system $\Sigma = \{\alpha_{1},\ldots,\alpha_{l}\} \subset \mathfrak{h}^{\ast}$. We also denote by $\kappa$ the Cartan-Killing form of $\mathfrak{g}^{\mathbb{C}}$.

Now, given $\alpha \in \Pi^{+}$, we have $h_{\alpha} \in \mathfrak{h}$, such  that $\alpha = \kappa(\cdot,h_{\alpha})$. We can choose $x_{\alpha} \in \mathfrak{g}_{\alpha}$ and $y_{\alpha} \in \mathfrak{g}_{-\alpha}$ such that $[x_{\alpha},y_{\alpha}] = h_{\alpha}$. For every $\alpha \in \Sigma$, we can set 
$$h_{\alpha}^{\vee} = \frac{2}{\kappa(h_{\alpha},h_{\alpha})}h_{\alpha},$$ 
from this we have the fundamental weights $\{\omega_{\alpha} \ | \ \alpha \in \Sigma\} \subset \mathfrak{h}^{\ast}$, where $\omega_{\alpha}(h_{\beta}^{\vee}) = \delta_{\alpha \beta}$, $\forall \alpha, \beta \in \Sigma$. We denote by $$\Lambda_{\mathbb{Z}_{\geq 0}}^{\ast} = \bigoplus_{\alpha \in \Sigma}\mathbb{Z}_{\geq 0}\omega_{\alpha}$$ the set of integral dominant weights of $\mathfrak{g}^{\mathbb{C}}$.

From the Lie algebra representation theory, given $\mu \in \Lambda_{\mathbb{Z}_{\geq 0}}^{\ast}$ we have an irreducible $\mathfrak{g}^{\mathbb{C}}$-module $V(\mu)$ with highest weight $\mu$, we denote by $v_{\mu}^{+} \in V(\mu)$ the highest weight vector associated to $\mu \in  \Lambda_{\mathbb{Z}_{\geq 0}}^{\ast}$.

Let $G^{\mathbb{C}}$ be a connected, simply connected, and complex Lie group with simple Lie algebra $\mathfrak{g}^{\mathbb{C}}$, and consider $G \subset G^{\mathbb{C}}$ as being a compact real form of $G^{\mathbb{C}}$. Given a parabolic Lie subgroup $P \subset G^{\mathbb{C}}$, without loss of generality we can suppose

\begin{center}
$P  = P_{\Theta}$, \ for some \ $\Theta \subseteq \Sigma$.
\end{center}

By definition we have $P_{\Theta} = N_{G^{\mathbb{C}}}(\mathfrak{p}_{\Theta})$, where ${\text{Lie}}(P_{\Theta}) = \mathfrak{p}_{\Theta} \subset \mathfrak{g}^{\mathbb{C}}$ is given by

\begin{center}

$\mathfrak{p}_{\Theta} = \mathfrak{n}^{+} \oplus \mathfrak{h} \oplus \mathfrak{n}(\Theta)^{-}$, \ with \ $\mathfrak{n}(\Theta)^{-} = \displaystyle \sum_{\alpha \in \langle \Theta \rangle^{-}} \mathfrak{g}_{\alpha}$.

\end{center}
For our purposes it will be useful to consider the following basic subgroups:

\begin{center}

$T^{\mathbb{C}} \subset B \subset P \subset G^{\mathbb{C}}$.

\end{center}
We have for each element in the chain above of Lie subgroups the following characterization: 

\begin{itemize}

\item $T^{\mathbb{C}} = \exp(\mathfrak{h})$,  \ \ (complex torus)

\item $B = N^{+}T^{\mathbb{C}}$, where $N^{+} = \exp(\mathfrak{n}^{+})$, \ \ (Borel subgroup)

\item $P = P_{\Theta} = N_{G^{\mathbb{C}}}(\mathfrak{p}_{\Theta})$, for some $\Theta \subset \Sigma \subset \mathfrak{h}^{\ast}$. \ \ (parabolic subgroup)

\end{itemize}
Associated to the data above we will be concerned to study the {\it generalized complex flag manifold} defined by  
$$
X_{P} = G^{\mathbb{C}} / P = G /G \cap P.
$$
The following theorem allows us to describe all $G$-invariant K\"{a}hler structures on $X_{P}$.
\begin{theorem}[Azad-Biswas, \cite{AZAD}]
\label{AZADBISWAS}
Let $\omega \in \Omega^{1,1}(X_{P})^{G}$ be a closed invariant real $(1,1)$-form, then we have

\begin{center}

$\pi^{\ast}\omega = \sqrt{-1}\partial \overline{\partial}\varphi$,

\end{center}
where $\pi \colon G^{\mathbb{C}} \to X_{P}$, and $\varphi \colon G^{\mathbb{C}} \to \mathbb{R}$ is given by 
\begin{center}
$\varphi(g) = \displaystyle \sum_{\alpha \in \Sigma \backslash \Theta}c_{\alpha}\log||gv_{\omega_{\alpha}}^{+}||$, 
\end{center}
with $c_{\alpha} \in \mathbb{R}_{\geq 0}$, $\forall \alpha \in \Sigma \backslash \Theta$. Conversely, every function $\varphi$ as above defines a closed invariant real $(1,1)$-form $\omega_{\varphi} \in \Omega^{1,1}(X_{P})^{G}$. Moreover, if $c_{\alpha} > 0$,  $\forall \alpha \in \Sigma \backslash \Theta$, then $\omega_{\varphi}$ defines a K\"{a}hler form on $X_{P}$.

\end{theorem}

\begin{remark}
\label{innerproduct}
It is worth pointing out that the norm $|| \cdot ||$ in the last theorem is a norm induced by some fixed $G$-invariant inner product $\langle \cdot, \cdot \rangle_{\alpha}$ on $V(\omega_{\alpha})$, $\forall \alpha \in \Sigma \backslash \Theta$. 
\end{remark}

Let $X_{P}$ be a flag manifold associated to some parabolic Lie subgroup $P = P_{\Theta} \subset G^{\mathbb{C}}$. According to Theorem \ref{AZADBISWAS}, by taking a fundamental weight $\omega_{\alpha} \in \Lambda_{\mathbb{Z}_{\geq0}}^{\ast}$, such that $\alpha \in \Sigma \backslash \Theta$, we can associate to this weight a closed real $G$-invariant $(1,1)$-form $\Omega_{\alpha} \in \Omega^{1,1}(X_{P})^{G}$ which satisfies 
\begin{equation}
\label{basicforms}
\pi^{\ast}\Omega_{\alpha} = \sqrt{-1}\partial \overline{\partial} \varphi_{\omega_{\alpha}},
\end{equation}
where $\pi \colon G^{\mathbb{C}} \to G^{\mathbb{C}} / P = X_{P}$, and $\varphi_{\omega_{\alpha}}(g) = \displaystyle \frac{1}{2\pi}\log||gv_{\omega_{\alpha}}^{+}||^{2}$, $\forall g \in G^{\mathbb{C}}$. 

The characterization for $G$-invariant real $(1,1)$-forms on $X_{P}$ provided by Theorem \ref{AZADBISWAS} can be used to compute the Chern class of holomorphic line bundles over flag manifolds, let us briefly describe how it can be done. Since each $\omega_{\alpha} \in \Lambda_{\mathbb{Z}_{\geq 0}}^{\ast}$ is an integral dominant weight, we can associate to it a holomorphic character $\chi_{\omega_{\alpha}} \colon T^{\mathbb{C}} \to \mathbb{C}^{\times}$, such that $(d\chi_{\omega_{\alpha}})_{e} = \omega_{\alpha}$, see for instance \cite[p. 466]{TAYLOR}. Given a parabolic Lie subgroup $P \subset G^{\mathbb{C}}$, we can take an extension $\chi_{\omega_{\alpha}} \colon P \to \mathbb{C}^{\times}$ and define a holomorphic line bundle by

\begin{equation}
\label{C8S8.2Sub8.2.1Eq8.2.4}
L_{\chi_{\omega_{\alpha}}} =  G^{\mathbb{C}} \times_{\chi_{\omega_{\alpha}}} \mathbb{C}_{-\omega_{\alpha}},
\end{equation}\\
i.e., as a vector bundle associated to the principal $P$-bundle $P \hookrightarrow G^{\mathbb{C}} \to G^{\mathbb{C}}/P$.
\begin{remark}
\label{remarkcocycle}
In the description above we consider $\mathbb{C}_{-\omega_{\alpha}}$ as a $P$-space with the action $pz = \chi_{\omega_{\alpha}}(p)^{-1}z$, $\forall p \in P$, and $\forall z \in \mathbb{C}$. Therefore, in terms of $\check{C}$ech cocycles, if we consider an open cover $X_{P} = \bigcup_{i \in I}U_{i}$ and
$G^{\mathbb{C}} = \{(U_{i})_{i \in I}, \psi_{ij} \colon U_{i} \cap U_{j} \to P\}$, then we have 

$$L_{\chi_{\omega_{\alpha}}} = \Big \{(U_{i})_{i \in I},\chi_{\omega_{\alpha}}^{-1} \circ \psi_{i j} \colon U_{i} \cap U_{j} \to \mathbb{C}^{\times} \Big \}.$$
Thus, it follows that $L_{\chi_{\omega_{\alpha}}} = \{g_{ij}\} \in \check{H}^{1}(X_{P},\mathscr{O}_{X_{P}}^{\ast})$, with $g_{ij} = \chi_{\omega_{\alpha}}^{-1} \circ \psi_{i j}$, where $i,j \in I$.
\end{remark}
For us it will be important to consider the following results, see for instance \cite{AZAD} and \cite{TOROIDAL}.
\begin{proposition}
\label{C8S8.2Sub8.2.3P8.2.7}
Let $X_{P}$ be a flag manifold associated to some parabolic Lie subgroup $P = P_{\Theta}\subset G^{\mathbb{C}}$. Then for every fundamental weight $\omega_{\alpha} \in \Lambda_{\mathbb{Z}_{\geq 0}}^{\ast}$, such that $\alpha \in \Sigma \backslash \Theta$, we have
\begin{equation}
\label{C8S8.2Sub8.2.3Eq8.2.28}
c_{1}(L_{\chi_{\omega_{\alpha}}}) = [\Omega_{\alpha}].
\end{equation}

\end{proposition}

\begin{proof}
Consider an open cover $X_{P} = \bigcup_{i \in I} U_{i}$ which trivializes both $P \hookrightarrow G^{\mathbb{C}} \to X_{P}$ and $L_{\chi_{\omega_{\alpha}}} \to X_{P}$, such that $\alpha \in \Sigma \backslash \Theta$. Now, take a collection of local sections $(s_{i})_{i \in I}$, such that $s_{i} \colon U_{i} \to G^{\mathbb{C}}$. From these, we define $q_{i} \colon U_{i} \to \mathbb{R}_{+}$ by setting
\begin{equation}
\label{functionshermitian}
q_{i} =  {\mathrm{e}}^{-2\pi \varphi_{\omega_{\alpha}} \circ s_{i}} = \frac{1}{||s_{i}v_{\omega_{\alpha}}^{+}||^{2}},
\end{equation}
for every $i \in I$. These functions $(q_{i})_{i \in I}$ satisfy $q_{j} = |\chi_{\omega_{\alpha}}^{-1} \circ \psi_{ij}|^{2}q_{i}$ on $U_{i} \cap U_{j} \neq \emptyset$, here we have used that $s_{j} = s_{i}\psi_{ij}$ on $U_{i} \cap U_{j} \neq \emptyset$, and $pv_{\omega_{\alpha}}^{+} = \chi_{\omega_{\alpha}}(p)v_{\omega_{\alpha}}^{+}$, for every $p \in P$ and $\alpha \in \Sigma \backslash \Theta$. Hence, we have a collection of functions $(q_{i})_{i \in I}$ which satisfy on $U_{i} \cap U_{j} \neq \emptyset$ the following 
\begin{equation}
\label{collectionofequ}
q_{j} = |g_{ij}|^{2}q_{i},
\end{equation}
such that $g_{ij} = \chi_{\omega_{\alpha}}^{-1} \circ \psi_{i j}$, where $i,j \in I$.

From the collection of smooth functions described above we can define a Hermitian structure $H$ on $L_{\chi_{\omega_{\alpha}}}$ by taking on each trivialization $f_{i} \colon L_{\chi_{\omega_{\alpha}}} \to U_{i} \times \mathbb{C}$ a metric defined by
\begin{equation}
\label{hermitian}
H((x,v),(x,w)) = q_{i}(x) v\overline{w},
\end{equation}
for $(x,v),(x,w) \in L_{\chi_{\omega_{\alpha}}}|_{U_{i}} \cong U_{i} \times \mathbb{C}$. The Hermitian metric above induces a Chern connection $\nabla = d + \partial \log H$ with curvature $F_{\nabla}$ satisfying 
$$
\displaystyle \frac{\sqrt{-1}}{2\pi}F_{\nabla} = \Omega_{\alpha},
$$
thus it follows that $c_{1}(L_{\chi_{\omega_{\alpha}}}) = [\Omega_{\alpha}]$.
\end{proof}
\begin{proposition}
\label{C8S8.2Sub8.2.3P8.2.6}
Let $X_{P}$ be a flag manifold associated to some parabolic Lie subgroup $P = P_{\Theta}\subset G^{\mathbb{C}}$. Then, we have
\begin{equation}
\label{picardeq}
{\text{Pic}}(X_{P}) = H^{1,1}(X_{P},\mathbb{Z}) = H^{2}(X_{P},\mathbb{Z}) = \displaystyle \bigoplus_{\alpha \in \Sigma \backslash \Theta}\mathbb{Z}[\Omega_{\alpha}].
\end{equation}
\end{proposition}

\begin{proof}

Let us sketch the proof. The last equality in the right side of $\ref{picardeq}$ follows from the following facts:

\begin{itemize}

\item $\pi_{2}(X_{P}) \cong \pi_{1}(T(\Sigma \backslash \Theta)^{\mathbb{C}}) = \mathbb{Z}^{\#(\Sigma \backslash \Theta)}$, where

$$T(\Sigma \backslash \Theta)^{\mathbb{C}} = \exp \Big \{ \displaystyle \sum_{\alpha \in  \Sigma \backslash \Theta}a_{\alpha}h_{\alpha} \ \Big | \ a_{\alpha} \in \mathbb{C} \Big \};$$

\item Since $X_{P}$ is simply connected, it follows that $H_{2}(X_{P},\mathbb{Z}) \cong \pi_{2}(X_{P})$ (Hurewicz's theorem);

\item By taking $\mathbb{P}_{\alpha}^{1} \hookrightarrow X_{P}$, such that 

$$\mathbb{P}_{\alpha}^{1} = \overline{\exp(\mathfrak{g}_{-\alpha})x_{0}} \subset X_{P},$$ 
$\forall \alpha \in \Sigma \backslash \Theta$, where $x_{0} = eP \in X_{P}$, it follows that 

\begin{center}

$\Big \langle c_{1}(L_{\chi_{\omega_{\alpha}}}), \big [ \mathbb{P}_{\beta}^{1}\big] \Big \rangle = \displaystyle \int_{\mathbb{P}_{\beta}^{1}} c_{1}(L_{\chi_{\omega_{\alpha}}}) = \delta_{\alpha \beta},$

\end{center}
$\forall \alpha,\beta \in \Sigma \backslash \Theta$. Hence, we obtain
\begin{center}

$\pi_{2}(X_{P}) = \displaystyle \bigoplus_{\alpha \in \Sigma \backslash \Theta} \mathbb{Z}\big [ \mathbb{P}_{\alpha}^{1}\big]$ \ \ and \ \ $H^{2}(X_{P},\mathbb{Z}) = \displaystyle \bigoplus_{\alpha \in \Sigma \backslash \Theta}  \mathbb{Z} c_{1}(L_{\chi_{\omega_{\alpha}}})$.

\end{center}
\end{itemize}
Therefore, $H^{1,1}(X_{P},\mathbb{Z}) = H^{2}(X_{P},\mathbb{Z})$. Now, from the Lefschetz theorem on (1,1)-classes \cite[p. 133]{DANIEL}, and the fact that $0 = b_{1}(X_{P}) = {\text{rk}}({\text{Pic}}^{0}(X_{P}))$, we obtain the first equality in \ref{picardeq}.
\end{proof}

\begin{remark}
\label{C8S8.2Sub8.2.3R8.2.10}
In the previous results and comments we have restricted our attention just to fundamental weights $\omega_{\alpha} \in \Lambda_{\mathbb{Z}_{\geq0}}^{\ast}$ for which $\alpha \in \Sigma \backslash \Theta$. Actually, if we have a parabolic Lie subgroup $P \subset G^{\mathbb{C}}$, such that $P = P_{\Theta}$, the decomposition 

\begin{center}

$P_{\Theta} = [P_{\Theta},P_{\Theta}]T(\Sigma \backslash \Theta)^{\mathbb{C}}$,

\end{center}
shows us that ${\text{Hom}}(P,\mathbb{C}^{\times}) = {\text{Hom}}(T(\Sigma \backslash \Theta)^{\mathbb{C}},\mathbb{C}^{\times})$. Therefore, if we take $\omega_{\alpha} \in \Lambda_{\mathbb{Z}_{\geq 0}}^{\ast}$, such that $\alpha \in \Theta$, we obtain $L_{\chi_{\omega_{\alpha}}} = X_{P} \times \mathbb{C}$, i.e., the associated holomorphic line bundle $L_{\chi_{\omega_{\alpha}}}$ is trivial.

\end{remark}

As we have seen in the previous section, it will be important for us to compute $c_{1}(X_{P})$. In order to do this, we notice that $c_{1}(X_{P}) = c_{1}(K_{X_{P}}^{-1})$, where

\begin{center}
$K_{X_{P}}^{-1} = \det \big(T^{1,0}X_{P} \big) = \bigwedge^{n}(T^{1,0}X_{P} \big)$,
\end{center}
here we suppose $\dim_{\mathbb{C}}(X_{P}) = n$.

In the context of complex flag manifolds the anticanonical line bundle can be described as follows. Consider the identification

\begin{center}

$\mathfrak{m} = \displaystyle \sum_{\alpha \in \Pi^{+} \backslash \langle \Theta \rangle^{+}} \mathfrak{g}_{-\alpha} = T_{x_{0}}^{1,0}X_{P}$,
\end{center}
where $x_{0} = eP \in X_{P}$. We have the following characterization for $T^{1,0}X_{P}$

\begin{center}

$T^{1,0}X_{P} = G^{\mathbb{C}} \times_{P} \mathfrak{m}$,

\end{center}
such that the twisted product on the right side above is obtained from the isotropy representation ${\rm{Ad}} \colon P \to {\rm{GL}}(\mathfrak{m})$ as an associated holomorphic vector bundle.

Let us introduce $\delta_{P} \in \mathfrak{h}^{\ast}$ by setting

\begin{center}

$\delta_{P} = \displaystyle \sum_{\alpha \in \Pi^{+} \backslash \langle \Theta \rangle^{+}} \alpha$.

\end{center}
Since $P = [P,P]T(\Sigma \backslash \Theta)^{\mathbb{C}}$, a straightforward computation shows that
\begin{equation}
\label{charactercanonical}
\det \circ {\rm{Ad}} = \chi_{\delta_{P}}^{-1},    
\end{equation}
from this we have the following result. 

\begin{proposition}
\label{C8S8.2Sub8.2.3Eq8.2.35}
Let $X_{P}$ be a flag manifold associated to some parabolic Lie subgroup $P = P_{\Theta}\subset G^{\mathbb{C}}$, then we have

$$K_{X_{P}}^{-1} = \det \big(T^{1,0}X_{P} \big) =\det \big ( G^{\mathbb{C}} \times_{P} \mathfrak{m} \big )= L_{\chi_{\delta_{P}}}.$$

\end{proposition}
\begin{proof}
The proof follows from the fact that

\begin{center}

$\det \big ( G^{\mathbb{C}} \times_{P} \mathfrak{m} \big ) = \bigg \{(U_{i})_{i \in I}, \det\big ({\rm{Ad}}(\psi_{i j})\big )\colon U_{i} \cap U_{j} \to \mathbb{C}^{\times} \bigg \}$,    
    
\end{center}
thus from \ref{charactercanonical} we have the desired result.
\end{proof}

The result above allows us to write 

$$K_{X_{P}}^{-1} = \bigg \{(U_{i})_{i \in I},\chi_{\delta_{P}}^{-1} \circ \psi_{i j} \colon U_{i} \cap U_{j} \to \mathbb{C}^{\times} \bigg \},$$
see Remark \ref{remarkcocycle}. Moreover, since the holomorphic character associated to $\delta_{P}$ can be written as
$$\chi_{\delta_{P}} = \displaystyle \prod_{\alpha \in \Sigma \backslash \Theta} \chi_{\omega_{\alpha}}^{\langle \delta_{P},h_{\alpha}^{\vee} \rangle},$$
we have the following characterization

$$K_{X_{P}}^{-1} =  L_{\chi_{\delta_{P}}} = \displaystyle \bigotimes_{\alpha \in \Sigma \backslash \Theta} L_{\chi_{\omega_{\alpha}}}^{\otimes \langle \delta_{P},h_{\alpha}^{\vee} \rangle}.$$
Therefore, we have the following description for $c_{1}(X_{P})$

\begin{equation}
\label{Cherncanonical}
c_{1}(X_{P}) = \displaystyle \sum_{\alpha \in \Sigma \backslash \Theta} \langle \delta_{P},h_{\alpha}^{\vee} \rangle \big [ \Omega_{\alpha} \big],
\end{equation}
thus from Theorem \ref{AZADBISWAS} we have a K\"{a}hler-Einstein structure $\omega_{X_{P}}$ defined by

\begin{equation}
\label{canonicalmetric}
\omega_{X_{P}} =  \displaystyle \sum_{\alpha \in \Sigma \backslash \Theta} \langle \delta_{P},h_{\alpha}^{\vee} \rangle \Omega_{\alpha},
\end{equation}
notice that ${\text{Ric}}(\omega_{X_{P}}) = 2\pi \omega_{X_{P}}$. It is worth pointing out that, also from Theorem \ref{AZADBISWAS}, we have $\omega_{X_{P}}$ determined by the quasi-potential $\varphi \colon G^{\mathbb{C}} \to \mathbb{R}$ defined by
\begin{equation}
\label{quasipotential}
\varphi(g) = \displaystyle \frac{1}{2\pi} \log \Big ( \prod_{\alpha \in \Sigma \backslash \Theta} ||gv_{\omega_{\alpha}}^{+}||^{2  \langle \delta_{P},h_{\alpha}^{\vee} \rangle}\Big), 
\end{equation}
for every $g \in G^{\mathbb{C}}$. Hence, given a local section $s_{U} \colon U \subset X_{P} \to G^{\mathbb{C}}$ we obtain the following local expression for $\omega_{X_{P}}$

\begin{equation}
\label{localform}
\omega_{X_{P}} = \displaystyle \frac{\sqrt{-1}}{2 \pi} \partial \overline{\partial}\log \Big ( \prod_{\alpha \in \Sigma \backslash \Theta} ||s_{U}v_{\omega_{\alpha}}^{+}||^{2  \langle \delta_{P},h_{\alpha}^{\vee} \rangle}\Big).    
\end{equation}

\begin{remark}
In order to do some local computations it will be convenient for us to consider the open set defined by the ``opposite" big cell in $X_{P}$. This open set is a distinguished coordinate neighbourhood $U \subset X_{P}$ of $x_{0} = eP \in X_{P}$ defined by the maximal Schubert cell. A brief description for the opposite big cell can be done as follows: Let $\Pi = \Pi^{+} \cup \Pi^{-}$ be the root system associated to the simple root system $\Sigma \subset \mathfrak{h}^{\ast}$, from this we can define the opposite big cell $U \subset X_{P}$ by

\begin{center}

 $U =  B^{-}x_{0} = R_{u}(P_{\Theta})^{-}x_{0} \subset X_{P}$,  

\end{center}
 where $B^{-} = \exp(\mathfrak{h} \oplus \mathfrak{n}^{-})$, and
 
 \begin{center}
 
 $R_{u}(P_{\Theta})^{-} = \displaystyle \prod_{\alpha \in \Pi^{-} \backslash \langle \Theta \rangle^{-}}N_{\alpha}^{-}$, \ \ (opposite unipotent radical)
 
 \end{center}
with $N_{\alpha}^{-} = \exp(\mathfrak{g}_{\alpha})$, $\forall \alpha \in \Pi^{-} \backslash \langle \Theta \rangle^{-}$. The opposite big cell defines a contractible open dense subset of $X_{P}$ and thus the restriction of any vector bundle over this open set is trivial . For further results about Schubert cells and Schubert varieties we suggest \cite{MONOMIAL}.
\end{remark}

\begin{remark} Unless otherwise stated, in the examples which we shall describe throughout this work we will use the conventions of \cite{SMA} for the realization of classical simple Lie algebras as matrix Lie algebras. \end{remark}

Let us illustrate the ideas described so far by means of basic examples.

\begin{example}
\label{exampleP1}
Considering $G^{\mathbb{C}} = {\rm{SL}}(2,\mathbb{C})$, we fix the triangular decomposition for $\mathfrak{sl}(2,\mathbb{C})$ given by

\begin{center}

$\mathfrak{sl}(2,\mathbb{C}) = \Big \langle x = \begin{pmatrix}
 0 & 1 \\
 0 & 0
\end{pmatrix} \Big \rangle_{\mathbb{C}} \oplus  \Big \langle h = \begin{pmatrix}
 1 & 0 \\
 0 & -1
\end{pmatrix} \Big \rangle_{\mathbb{C}} \oplus \Big \langle y = \begin{pmatrix}
 0 & 0 \\
 1 & 0
\end{pmatrix} \Big \rangle_{\mathbb{C}}$.

\end{center}
Notice that all the information about the decomposition above is codified in $\Sigma = \{\alpha\}$, $\Pi = \{\alpha,-\alpha\}$, and in our set of integral dominant weights, which in this case is given by

\begin{center}

$\Lambda_{\mathbb{Z}_{\geq 0}}^{\ast} = \mathbb{Z}_{\geq 0}\omega_{\alpha} $.

\end{center} 
We take $P = B$ (Borel subgroup) and from this we obtain
\begin{center}
$X_{B} = {\rm{SL}}(2,\mathbb{C})/B = \mathbb{C}{\rm{P}}^{1}$.
\end{center}
By considering the cellular decomposition

\begin{center}

$X_{B} = \mathbb{C}{\rm{P}}^{1} = N^{-}x_{0} \cup \pi \Big ( \begin{pmatrix}
 0 & 1 \\
 -1 & 0
\end{pmatrix} \Big ),$

\end{center}
we take the open set defined by the opposite big cell $U =  N^{-}x_{0} \subset X_{B}$ and the local section $s_{U} \colon U \subset \mathbb{C}{\rm{P}}^{1} \to {\rm{SL}}(2,\mathbb{C})$ defined by

\begin{center}

$s_{U}(nx_{0}) = n$, \ \ $\forall n \in N^{-}$.
\end{center}
It is worthwhile to observe that in this case we have the open set  $U =  N^{-}x_{0} \subset \mathbb{C}{\rm{P}}^{1}$ parameterized by 

\begin{center}
    
$z \in \mathbb{C} \mapsto \begin{pmatrix}
 1 & 0 \\
 z & 1
\end{pmatrix} x_{0} \subset \mathbb{C}{\rm{P}}^{1}.$
    
\end{center}
Since $V(\omega_{\alpha}) = \mathbb{C}^{2}$, $v_{\omega_{\alpha}}^{+} = e_{1}$, and $ \langle \delta_{B},h_{\alpha}^{\vee} \rangle = 2$, it follows from Equation \ref{localform} that over the opposite big cell $U = N^{-}x_{0} \subset X_{B}$ we have

\begin{center}

$\omega_{\mathbb{C}{\rm{P}}^{1}} = \displaystyle  \frac{\sqrt{-1}}{2\pi} \langle \delta_{B},h_{\alpha}^{\vee} \rangle \partial \overline{\partial} \log \Bigg ( \Big |\Big |\begin{pmatrix}
 1 & 0 \\
 z & 1
\end{pmatrix} e_{1} \Big| \Big|^{2} \Bigg ) = \frac{\sqrt{-1}}{\pi} \partial \overline{\partial} \log (1+|z|^{2}).$

\end{center}
Notice that in this case we have $K_{\mathbb{C}{\rm{P}}^{1}}^{-1} =  T^{(1,0)}\mathbb{C}{\rm{P}}^{1} = T\mathbb{C}{\rm{P}}^{1}$, and 
$K_{\mathbb{C}{\rm{P}}^{1}} = T^{\ast}\mathbb{C}{\rm{P}}^{1}$. Further, we have
\begin{center}

${\text{Pic}}(\mathbb{C}{\rm{P}}^{1}) = \mathbb{Z}c_{1}(L_{\chi_{\omega_{\alpha}}})$,    
    
\end{center}
thus $K_{\mathbb{C}{\rm{P}}^{1}}^{-1} = L_{\chi_{\omega_{\alpha}}}^{\otimes 2}$. If we denote $L_{\chi_{\omega_{\alpha}}}^{\otimes \ell} = \mathscr{O}(\ell)$, $\forall \ell \in \mathbb{Z}$, we have $K_{\mathbb{C}{\rm{P}}^{1}} = \mathscr{O}(-2)$, and ${\text{Pic}}(\mathbb{C}{\rm{P}}^{1})$ generated by $\mathscr{O}(1)$. Moreover, in this case we have the Fano index given by $I(\mathbb{C}{\rm{P}}^{1}) = 2$, which in turn implies that

\begin{center}

$K_{\mathbb{C}{\rm{P}}^{1}}^{\otimes \frac{1}{2}} = \mathscr{O}(-1).$
    
\end{center}
The computation above is an interesting exercise to understand how the approach by elements of the Lie theory, especially representation theory, can be useful to describe geometric structures.
\end{example}

\begin{example}
\label{examplePn}
Let us briefly describe the generalization of the previous example for $X_{P} = \mathbb{C}{\rm{P}}^{n}$. At first, we recall some basic data related to the Lie algebra $\mathfrak{sl}(n+1,\mathbb{C})$. By fixing the Cartan subalgebra $\mathfrak{h} \subset \mathfrak{sl}(n+1,\mathbb{C})$ given by diagonal matrices whose the trace is equal to zero, we have the set of simple roots given by
$$\Sigma = \Big \{ \alpha_{l} = \epsilon_{l} - \epsilon_{l+1} \ \Big | \ l = 1, \ldots,n\Big\},$$
here $\epsilon_{l} \colon {\text{diag}}\{a_{1},\ldots,a_{n+1} \} \mapsto a_{l}$, $ \forall l = 1, \ldots,n+1$. Therefore, the set of positive roots is given by
$$\Pi^+ = \Big \{ \alpha_{ij} = \epsilon_{i} - \epsilon_{j} \ \Big | \ i<j  \Big\}. $$
In this example we consider $\Theta = \Sigma \backslash \{\alpha_{1}\}$ and $P = P_\Theta$. Now, we take the open set defined by the opposite big cell $U =  R_{u}(P_{\Theta})^{-}x_{0} \subset \mathbb{C}{\rm{P}}^{n}$, where $x_0=eP$ (trivial coset), and  
\begin{center}
 $ R_{u}(P_{\Theta})^{-} = \displaystyle \prod_{\alpha \in \Pi^{-} \backslash \langle \Theta \rangle^{-}}N_{\alpha}^{-}$, \ with \ $N_{\alpha}^{-} = \exp(\mathfrak{g}_{\alpha})$, $\forall \alpha \in \Pi^{-} \backslash \langle \Theta \rangle^{-}$.
 \end{center}
We remark that in this case the open set $U =  R_{u}(P_{\Theta})^{-}x_{0}$ is parameterized by
\begin{center}
$(z_{1},\ldots,z_{n}) \in \mathbb{C}^{n} \mapsto \begin{pmatrix}
1 & 0 &\cdots & 0 \\
z_{1} & 1  &\cdots & 0 \\                  
\ \vdots  & \vdots &\ddots  & \vdots  \\
z_{n} & 0 & \cdots &1 
 \end{pmatrix}x_{0} \in U =  R_{u}(P_{\Theta})^{-}x_{0}$.
\end{center}
Also, notice that the coordinate system above is induced directly from the exponential map $\exp \colon {\text{Lie}}(R_{u}(P)^{-}) \to R_{u}(P)^{-}$.

From the data above, we can take a local section $s_{U} \colon U \subset \mathbb{C}{\rm{P}}^{n} \to {\rm{SL}}(n+1,\mathbb{C})$, such that 
$$s_{U}(nx_{0}) = n \in {\rm{SL}}(n+1,\mathbb{C}).$$ 
Since $V(\omega_{\alpha}) = \mathbb{C}^{n+1}$, $v_{\omega_{\alpha}}^{+} = e_{1}$, and $ \langle \delta_{P_{\Sigma \backslash \{\alpha_{1}\}}},h_{\alpha_{1}}^{\vee} \rangle = n+1$, it follows from Equation \ref{localform} that over the opposite big cell $U =  R_{u}(P_{\Theta})^{-}x_{0} \subset \mathbb{C}{\rm{P}}^{n}$ we have the expression of $\omega_{\mathbb{C}{\rm{P}}^{n}}$ given by
$$\omega_{\mathbb{C}{\rm{P}}^{n}} = \displaystyle \frac{(n+1)}{2\pi} \sqrt{-1}\partial \overline{ \partial} \log \Big (1 + \sum_{l = 1}^{n}|z_{l}|^{2} \Big ).$$
Notice that in this case we have
\begin{center}

${\text{Pic}}(\mathbb{C}{\rm{P}}^{n}) = \mathbb{Z}c_{1}(L_{\chi_{\omega_{\alpha_{1}}}})$,    
    
\end{center}
thus $K_{\mathbb{C}{\rm{P}}^{n}}^{-1} = L_{\chi_{\omega_{\alpha_{1}}}}^{\otimes (n+1)}$. If we denote by $L_{\chi_{\omega_{\alpha_{1}}}}^{\otimes \ell} = \mathscr{O}(\ell)$, $\forall \ell \in \mathbb{Z}$, we have $K_{\mathbb{C}{\rm{P}}^{n}} = \mathscr{O}(-n-1)$, and ${\text{Pic}}(\mathbb{C}{\rm{P}}^{n})$ generated by $\mathscr{O}(1)$. Moreover, in this case we have the Fano index given by $I(\mathbb{C}{\rm{P}}^{n}) = n+1$, which implies that

\begin{center}

$K_{\mathbb{C}{\rm{P}}^{n}}^{\otimes \frac{1}{n+1}} = \mathscr{O}(-1).$

\end{center}

\end{example}

\begin{example}
\label{grassmanian}
Consider $G^{\mathbb{C}} = {\rm{SL}}(4,\mathbb{C})$, here we use the same choice of Cartan subalgebra and conventions for the simple root system as in the previous example. Since our simple root system is given by
$$\Sigma = \Big \{ \alpha_{1} = \epsilon_{1}-\epsilon_{2}, \alpha_{2} = \epsilon_{2} - \epsilon_{3}, \alpha_{3} = \epsilon_{3} - \epsilon_{4}\Big \},$$
by taking $\Theta = \Sigma \backslash \{\alpha_{2}\}$ we obtain  for $P = P_{\Theta}$ the flag manifold $X_{P} = {\rm{Gr}}(2,\mathbb{C}^{4})$ (Klein quadric). Notice that in this case we have
$${\text{Pic}}({\rm{Gr}}(2,\mathbb{C}^{4})) = \mathbb{Z}c_{1}(L_{\chi_{\alpha_{2}}}),$$
thus from Proposition \ref{C8S8.2Sub8.2.3Eq8.2.35} it follows that
$$K_{{\rm{Gr}}(2,\mathbb{C}^{4})}^{-1} = L_{\chi_{\omega_{\alpha_{2}}}}^{\otimes \langle \delta_{P},h_{\alpha_{2}}^{\vee} \rangle}.$$

By considering our Lie-theoretical conventions, we have
$$\Pi^{+} \backslash \langle \Theta \rangle^{+} = \Big \{ \alpha_{2}, \alpha_{1}+\alpha_{2},\alpha_{2}+\alpha_{3},\alpha_{1} + \alpha_{2} + \alpha_{3}\Big \},$$
hence

$$\delta_{P} = \displaystyle \sum_{\alpha \in \Pi^{+} \backslash \langle \Theta \rangle^{+}} \alpha = 2\alpha_{1} + 4 \alpha_{2} + 2\alpha_{3}.$$
By means of the Cartan matrix of  $\mathfrak{sl}(4,\mathbb{C})$ we obtain 
$$ \langle \delta_{P},h_{\alpha_{2}}^{\vee} \rangle = 4 \implies K_{{\rm{Gr}}(2,\mathbb{C}^{4})}^{-1} = L_{\chi_{\omega_{\alpha_{2}}}}^{\otimes4}.$$  
In what follows we shall use the following notation: 
$$L_{\chi_{\omega_{\alpha_{2}}}}^{\otimes \ell} := \mathscr{O}_{\alpha_{2}}(\ell),$$
for every $\ell \in \mathbb{Z}$, therefore we have $K_{{\rm{Gr}}(2,\mathbb{C}^{4})} = \mathscr{O}_{\alpha_{2}}(-4)$. In order to compute the local expression of $\omega_{{\rm{Gr}}(2,\mathbb{C}^{4})} \in c_{1}(\mathscr{O}_{\alpha_{2}}(-4))$, we observe that in this case the quasi-potential $\varphi \colon  {\rm{SL}}(4,\mathbb{C}) \to \mathbb{R}$ is given by

\begin{center}

$\varphi(g) = \displaystyle \frac{\langle \delta_{P}, h_{\alpha_{2}}^{\vee}\rangle}{2\pi} \log ||g v_{\omega_{\alpha_{2}}}^{+}||^{2} = \displaystyle \frac{2}{\pi} \log ||g v_{\omega_{\alpha_{2}}}^{+}||^{2}$,

\end{center}
where $V(\omega_{\alpha_{2}}) = \bigwedge^{2}(\mathbb{C}^{4})$ and $v_{\omega_{\alpha_{2}}}^{+} =  e_{1} \wedge e_{2}$, thus we fix the basis $\{e_{i} \wedge e_{j}\}_{i<j}$ for $V(\omega_{\alpha_{2}}) = \bigwedge^{2}(\mathbb{C}^{4})$. Similarly to the previous examples, we consider the open set defined by the opposite big cell $U = B^{-}x_{0} \subset {\rm{Gr}}(2,\mathbb{C}^{4})$. In this case we have the local coordinates $nx_{0} \in U$ given by

\begin{center}

$(z_{1},z_{2},z_{3},z_{4}) \in \mathbb{C}^{4} \mapsto\begin{pmatrix}
1 & 0 & 0 & 0\\
0 & 1 & 0 & 0 \\                  
z_{1}  & z_{3} & 1 & 0 \\
z_{2}  & z_{4} & 0 & 1
 \end{pmatrix} x_{0} \in U = B^{-}x_{0}.$

\end{center}
Also, notice that the coordinates above are obtained directly from the exponential map $\exp \colon {\text{Lie}}(R_{u}(P)^{-}) \to R_{u}(P)^{-}$. From this, by taking the local section $s_{U} \colon U \subset {\rm{Gr}}(2,\mathbb{C}^{4}) \to {\rm{SL}}(4,\mathbb{C})$, such that $s_{U}(nx_{0}) = n$, we obtain

\begin{center}

$\varphi(s_{U}(nx_{0})) = \displaystyle \frac{2}{\pi} \log \Big ( 1 + \sum_{k = 1}^{4}|z_{k}|^{2} + \bigg |\det \begin{pmatrix}
 z_{1} & z_{3} \\
 z_{2} & z_{4}
\end{pmatrix} \bigg |^{2} \Big)$,

\end{center}
and the following local expression for  $\omega_{{\rm{Gr}}(2,\mathbb{C}^{4})} \in c_{1}(\mathscr{O}_{\alpha_{2}}(-4))$
\begin{equation}
\label{C8S8.3Sub8.3.2Eq8.3.21}
\omega_{{\rm{Gr}}(2,\mathbb{C}^{4})} =  \displaystyle \frac{2 \sqrt{-1}}{\pi}  \partial \overline{\partial} \log \Big (1+ \sum_{k = 1}^{4}|z_{k}|^{2} + \bigg |\det \begin{pmatrix}
 z_{1} & z_{3} \\
 z_{2} & z_{4}
\end{pmatrix} \bigg |^{2} \Big).
\end{equation}
It is worthwhile to observe that in this case we have the Fano index of ${\rm{Gr}}(2,\mathbb{C}^{4})$ given by $I({\rm{Gr}}(2,\mathbb{C}^{4})) = 4$, thus we obtain

\begin{center}
    
$K_{{\rm{Gr}}(2,\mathbb{C}^{4})}^{\otimes \frac{1}{4}} = \mathscr{O}_{\alpha_{2}}(-1).$    
    
\end{center}

\end{example}

\begin{remark}
Observe that from Proposition \ref{C8S8.2Sub8.2.3Eq8.2.35} we have for a complex flag manifold $X_{P}$ its Fano index given by 

\begin{center}   
$I(X_{P}) = {\text{gcd}} \Big (  \langle \delta_{P}, h_{\alpha}^{\vee} \rangle \ \Big | \ \alpha \in \Sigma \backslash \Theta \Big )$,    
\end{center}
here we suppose $P = P_{\Theta} \subset G^{\mathbb{C}}$, for some $\Theta \subset \Sigma$. Thus, $I(X_{P})$ can be completely determined by the Cartan matrix of $\mathfrak{g}^{\mathbb{C}}$.
\end{remark}

\subsection{Principal circle bundles over complex flag manifolds}
\label{subsec3.2}
As we have seen previously, given a complex manifold $X$ and a line bundle $L \to X$, we can define a circle bundle by taking

\begin{center}
$Q(L) = \Big \{ u \in L \ \Big | \ \sqrt{H(u,u)} = 1 \Big\}$,
\end{center}
where $H$ denotes some Hermitian structure on $L$. The action of ${\rm{U}}(1)$ on $Q(L)$ is defined by $u \cdot \theta = u \mathrm{e}^{2\pi \theta \sqrt{-1}}$, $\forall \theta \in {\rm{U}}(1)$ and $\forall u \in Q(L)$. Furthermore, a straightforward computation shows that

\begin{center}
$L = Q(L) \times_{{\rm{U}}(1)} \mathbb{C}$.
\end{center}
Conversely, given a circle bundle

\begin{center}
    
${\rm{U}}(1) \hookrightarrow Q \to X$,    
    
\end{center}
we can construct a line bundle $L(Q) \to X$ as an associated vector bundle defined by 

\begin{center}
$L(Q) = Q \times_{{\rm{U}}(1)} \mathbb{C}$,
\end{center}
where the twisted product is taken with respect to the action
\begin{center}
$\theta \cdot (u,z) = (u \cdot \theta, \mathrm{e}^{-2\pi \theta \sqrt{-1}}z)$, 
\end{center}
$\forall \theta \in {\rm{U}}(1)$ and $\forall (u,z) \in Q \times \mathbb{C}$. If we denote the set of isomorphism classes of circle bundles over $X$ by

\begin{center}
$\mathscr{P}(X,{\rm{U}}(1))$,    
\end{center}
the previous idea provides the following correspondences: 

\begin{itemize}
    
\item ${\text{Pic}}^{\infty}(X) \to \mathscr{P}(X,{\rm{U}}(1))$, $L \mapsto Q(L)$;

\item $\mathscr{P}(X,{\rm{U}}(1)) \to {\text{Pic}}^{\infty}(X)$, $Q \mapsto L(Q)$; 
    
\end{itemize}
where ${\text{Pic}}^{\infty}(X)$ denotes the smooth Picard group of $X$, i.e., the set of isomorphism classes of complex vector bundles of rank $1$. Furthermore, we have 

\begin{itemize}

\item $L(Q(L)) = L$, \ \ $[u,z] \mapsto z u$;

\item $Q(L(Q)) = Q$, \ \ $u \mapsto [u,1]$.

\end{itemize}
It will be important in this work to consider the following well-known results, the details about the proofs can be found in \cite{TOROIDAL}, \cite[Theorem 2.1]{BLAIR}.

\begin{theorem}

The set $\mathscr{P}(X,{\rm{U}}(1))$ of isomorphism classes of all principal circle bundles over $X$ forms an additive group. The identity element is given by the trivial bundle.
\end{theorem}

\begin{remark}
\label{productcircle}
From the previous comments, it will be suitable to consider the following characterization for the group structure of  $\mathscr{P}(X,{\rm{U}}(1))$

\begin{center}
    
$Q_{1} + Q_{2} = Q(L(Q_{1}) \otimes L(Q_{2})),$    
    
\end{center}
for $Q_{1},Q_{2} \in \mathscr{P}(X,{\rm{U}}(1))$. 
\end{remark}

Given $Q \in \mathscr{P}(X,{\rm{U}}(1))$, we can consider its associated homotopy exact sequence:

\begin{center}
\begin{tikzcd}

\cdots \arrow[r] & \pi_{2}(Q) \arrow[r] & \pi_{2}(X) \arrow[r,"\Delta_{Q}"] & \pi_{1}({\rm{U}}(1)) \arrow[r] & \cdots, 

\end{tikzcd}
\end{center}
since $\pi_{1}({\rm{U}}(1)) \cong \mathbb{Z} \Longrightarrow \Delta_{Q} \in {\rm{Hom}}(\pi_{2}(X),\mathbb{Z})$. From this, we have the following result. 

\begin{theorem}

Let $h \colon \pi_{2}(X) \to H_{2}(X,\mathbb{Z})$ be the natural homomorphism and $\ell$ an integer given by $\Delta_{Q}c = \ell b_{0}$, where $b_{0}$ is the generator of $\pi_{1}({\rm{U}}(1))$ and $\Delta_{Q}$ is the boundary operator of the exact homotopy sequence of a bundle $Q \in \mathscr{P}(X,{\rm{U}}(1))$. Then,

\begin{center}
    
$\Big \langle \mathrm{e}(Q),h(c) \Big \rangle = \displaystyle \int_{h(c)} \mathrm{e}(Q) = \ell,$    
    
\end{center}
where $\mathrm{e}(Q)$ denotes the Euler class of $Q \in \mathscr{P}(X,{\rm{U}}(1))$.
\end{theorem}

For our purposes the following corollary will be important.

\begin{corollary}
\label{fundcorollary}
If $X$ is simply connected, then $\mathscr{P}(X,{\rm{U}}(1))$ is isomorphic to ${\rm{Hom}}(\pi_{2}(X),\mathbb{Z})$. The isomorphism is given by $Q \mapsto \Delta_{Q}$, where $\Delta_{Q}$ is the boundary operator of the exact homotopy sequence of a bundle $Q \in \mathscr{P}(X,{\rm{U}}(1))$.
\end{corollary}

Now, let $X$ be a complex manifold. From Hurewicz's theorem, if $X$ is simply connected it follows that $h \colon \pi_{2}(X) \to H_{2}(X,\mathbb{Z})$ is an isomorphism, thus we obtain
\begin{equation}
\label{picinfty}
\mathscr{P}(X,{\rm{U}}(1)) \cong H^{2}(X,\mathbb{Z}) \cong {\text{Pic}}^{\infty}(X),
\end{equation}
where the first isomorphism is given by $\Delta_{Q} \mapsto \mathrm{e}(Q)$, $\forall Q \in \mathscr{P}(X,{\rm{U}}(1)) $ and the second isomorphism follows from the exponential exact sequence of sheaves

\begin{center}
\begin{tikzcd}

0 \arrow[r] & \underline{\mathbb{Z}} \arrow[r,"2\pi \sqrt{-1}"] & \underline{\mathbb{C}} \arrow[r,"\exp"] & \underline{\mathbb{C}}^{\times} \arrow[r] & 0, 

\end{tikzcd}
\end{center}
notice that $ {\text{Pic}}^{\infty}(X) \cong H^{1}(X,\underline{\mathbb{C}}^{\times})$, see for instance \cite[Chapter 2]{LOOP}.

The isomorphism \ref{picinfty} allows us to see that

\begin{center}
    
$\mathrm{e}(Q) = c_{1}(L(Q))$, \ \ and \ \ $c_{1}(L) = \mathrm{e}(Q(L))$,    
    
\end{center}
$\forall Q \in \mathscr{P}(X,{\rm{U}}(1)) $, $\forall L \in {\text{Pic}}^{\infty}(X)$. 

\begin{remark}
It is worth pointing out that, in the setting above, if $X$ is not simply connected we can also obtain the isomorphism \ref{picinfty}. Actually, if we consider the natural exact sequence of sheaves
\begin{center}
\begin{tikzcd}

0 \arrow[r] & \underline{\mathbb{Z}} \arrow[r,"2\pi \sqrt{-1}"] & \underline{\sqrt{-1}\mathbb{R}} \arrow[r,"\exp"] & \underline{S}^{1} \arrow[r] & 0, 

\end{tikzcd}
\end{center}
the result follows from the associated cohomology sequence 
\begin{center}
\begin{tikzcd}

\cdots \arrow[r] & H^{1}(X, \underline{\sqrt{-1}\mathbb{R}}) \arrow[r] & H^{1}(X,\underline{S}^{1}) \arrow[r] & H^{2}(X,\underline{\mathbb{Z}}) \arrow[r] & H^{2}(X, \underline{\sqrt{-1}\mathbb{R}}) \arrow[r] &\cdots 

\end{tikzcd}
\end{center}
notice that $\mathscr{P}(X,{\rm{U}}(1)) \cong H^{1}(X,\underline{S}^{1})$, see for instance \cite[Chapter 2, page 18]{BLAIR}. We also observe that in the particular case of a complex flag manifold $X_{P}$ we have a concrete description for $\pi_{2}(X_{P})$, so the result of Corollary \ref{fundcorollary} combined with the ideas of the proof of Proposition \ref{C8S8.2Sub8.2.3P8.2.6} provide a suitable approach which takes into account some interesting elements of Lie theory. 
\end{remark}
Now, from Proposition \ref{C8S8.2Sub8.2.3P8.2.6} and the last comments above, we have the following result. 

\begin{theorem}[ Kobayashi, \cite{TOROIDAL}]

Let $X_{P}$ be a complex flag manifold defined by a parabolic Lie subgroup $P = P_{\Theta} \subset G^{\mathbb{C}}$. Then, we have

\begin{center}
    
$\mathscr{P}(X_{P},{\rm{U}}(1)) = \displaystyle \bigoplus_{\alpha \in \Sigma \backslash \Theta}  \mathbb{Z}\mathrm{e}(Q(L_{\chi_{\omega_{\alpha}}})).$
    
\end{center}

\end{theorem}

\begin{remark}
It is worthwhile to point out that this last result which we presented above is stated slightly different in \cite{TOROIDAL}. We proceed in this way because our approach is concerned to describe connections and curvature of line bundles and principal circle bundles, therefore we use characteristic classes to describe $\mathscr{P}(X_{P},{\rm{U}}(1))$.
\end{remark}

\begin{remark}
Notice that, particularly, we have
\begin{center}
$\mathscr{P}(X_{P},{\rm{U}}(1)) \cong {\text{Pic}}(X_{P})$.
\end{center}
\end{remark}
In what follows, given a complex flag manifold $X_{P}$, where $P = P_{\Theta}$, we denote
\begin{equation}
\label{notationsimple}    
Q(\mu) := Q(L_{\chi_{\mu}}),
\end{equation}
for every $\mu \in \Lambda_{\mathbb{Z}_{\geq 0}}^{\ast}$. In some cases we also shall denote by $\pi_{Q(\mu)}  \colon Q(\mu) \to X_{P}$ the associated projection map.

Our next task will be to compute $\mathrm{e}(Q(\omega_{\alpha})) \in H^{2}(X_{P},\mathbb{Z})$, $\forall \alpha \in \Sigma \backslash \Theta$. In order to do so, it will be important to consider Proposition \ref{C8S8.2Sub8.2.3P8.2.7} and the fact that $\mathrm{e}(Q(\omega_{\alpha})) = c_{1}(L_{\chi_{\omega_{\alpha}}})$, $\forall \alpha \in \Sigma \backslash \Theta$.

Consider an open cover $X_{P} = \bigcup_{i \in I} U_{i}$ which trivializes both $P \hookrightarrow G^{\mathbb{C}} \to X_{P}$ and $L_{\chi_{\omega_{\alpha}}} \to X_{P}$, such that $\alpha \in \Sigma \backslash \Theta$, and take a collection of local sections $(s_{i})_{i \in I}$, such that $s_{i} \colon U_{i} \subset X_{P} \to G^{\mathbb{C}}$. As we have seen, associated to this data we can define $q_{i} \colon U_{i} \to \mathbb{R}_{+}$ by setting
\begin{center}

$q_{i} =  {\mathrm{e}}^{-2\pi \varphi_{\omega_{\alpha}} \circ s_{i}} = \displaystyle \frac{1}{||s_{i}v_{\omega_{\alpha}}^{+}||^{2}},$

\end{center}
$\forall i \in I$. From these functions we obtain a Hermitian structure $H$ on $L_{\chi_{\omega_{\alpha}}}$ by taking on each trivialization $f_{i} \colon L_{\chi_{\omega_{\alpha}}} \to U_{i} \times \mathbb{C}$ a Hermitian metric defined by
\begin{center}
$H((x,v),(x,w)) = q_{i}(x) v\overline{w},$
\end{center}
for $(x,v),(x,w) \in L_{\chi_{\omega_{\alpha}}}|_{U_{i}} \cong U_{i} \times \mathbb{C}$. Hence, for the pair $(L_{\chi_{\omega_{\alpha}}},H)$ we have the associated principal circle bundle

\begin{center}
    
$Q(\omega_{\alpha}) = \Big \{ u \in L_{\chi_{\omega_{\alpha}}} \ \Big | \ \sqrt{H(u,u)} = 1 \Big\}.$    
    
\end{center}
In terms of cocycles the principal circle bundle $Q(\omega_{\alpha})$ is determined by 

\begin{center}
    
$t_{ij} \colon U_{i} \cap U_{j} \to {\rm{U}}(1)$, \ \ $t_{ij} = \frac{g_{ij}}{||g_{ij}||},$
    
\end{center}
where $g_{ij} = \chi_{\omega_{\alpha}}^{-1} \circ \psi_{ij}$, $\forall i,j \in I$, see the proof of Proposition \ref{C8S8.2Sub8.2.3P8.2.7}. Therefore, if we take a local chart $h_{i} \colon \pi_{Q(\omega_{\alpha})}^{-1}(U_{i}) \subset Q(\omega_{\alpha}) \to U_{i} \times {\rm{U}}(1)$, on the transition $U_{i} \cap U_{j} \neq \emptyset$ we obtain
\begin{equation}
\label{transition}
(h_{i} \circ h_{j}^{-1})(x,a_{j}) = (x,a_{j}t_{ij}(x)) = (x,a_{i}),    
\end{equation}
thus we have $a_{i} = a_{j}t_{ij}$, on $U_{i} \cap U_{j} \neq \emptyset$. If we set 
\begin{equation}
\mathcal{A}_{i} = - \displaystyle \frac{1}{2}\big ( \partial - \overline{\partial} \big ) \log||s_{i} v_{\omega_{\alpha}}^{+}||^{2},    
\end{equation}
$\forall i \in I$, we obtain the following result.

\begin{proposition}
\label{locconnection}
The collection of local $\mathfrak{u}(1)$-valued 1-forms defined by 
\begin{equation}
\label{connection}    
\eta_{i}' = \pi_{Q(\omega_{\alpha})}^{\ast}\mathcal{A}_{i} + \frac{da_{i}}{a_{i}},
\end{equation}
where 
\begin{center}
$\mathcal{A}_{i} = - \displaystyle \frac{1}{2}\big ( \partial - \overline{\partial} \big ) \log||s_{i} v_{\omega_{\alpha}}^{+}||^{2},$
\end{center}
$\forall i \in I$, provides a connection $\eta'_{\alpha}$ on $Q(\omega_{\alpha})$ which satisfies $\eta'_{\alpha} = \eta_{i}'$ on $Q(\omega_{\alpha})|_{U_{i}}$, and
\begin{equation}
\label{eqconnections}
\displaystyle \frac{\sqrt{-1}}{2\pi}d\eta'_{\alpha} = \pi_{Q(\omega_{\alpha})}^{\ast} \Omega_{\alpha}.
\end{equation}
\end{proposition}

\begin{proof}
The proof follows from the following facts:

\begin{enumerate}

\item $\mathcal{A}_{i} = \mathcal{A}_{j} - t_{ij}^{-1}dt_{ij}$, on $U_{i} \cap U_{j}$;

\item $a_{i}^{-1}da_{i} = (a_{j}t_{ij})^{-1}d(a_{j}t_{ij}) = a_{j}^{-1}da_{j} + t_{ij}^{-1}dt_{ij}$, on $U_{i} \cap U_{j}.$

\end{enumerate}
The first fact is a consequence of the definition of $\mathcal{A}_{i}$. Actually, we have that 
\begin{center}
    
$\mathcal{A}_{i} = \displaystyle \frac{1}{2} \big (A_{i} - \overline{A}_{i}\big ) = \sqrt{-1} {\rm{Im}}(A_{i}),$
    
\end{center}
where $A_{i} \in \Omega^{1,0}(U_{i})$ is defined by
\begin{center}
    
$A_{i} = \partial \log H = - \partial \log ||s_{i}v_{\omega_{\alpha}}^{+}||^{2}.$  
    
\end{center}
These last comments just say that the set of gauge potentials $(\mathcal{A}_{i})_{i \in I}$ are induced by the Chern connection on $L_{\chi_{\omega_{\alpha}}}$ defined by $(A_{i})_{i \in I}$. Hence, we have on $U_{i} \cap U_{j}$

\begin{center}
    
$A_{i} = A_{j} - g_{ij}^{-1}dg_{ij} \Longrightarrow \mathcal{A}_{i} = \mathcal{A}_{j} - t_{ij}^{-1}dt_{ij},$
    
\end{center}
here we recall that $t_{ij} = \frac{g_{ij}}{||g_{ij}||}$. The second fact above follows from \ref{transition}.

Therefore, from $(1)$ and $(2)$, we have $\eta'_{\alpha} \in \Omega^{1}(Q(\omega_{\alpha});\mathfrak{u}(1))$ such that

\begin{center}

$\eta'_{\alpha} |_{ \pi_{Q(\omega_{\alpha})}^{-1}(U_{i})} = \pi_{Q(\omega_{\alpha})}^{\ast}\mathcal{A}_{i} + a_{i}^{-1}da_{i},$
    
\end{center}
notice that $\eta'_{\alpha}$ defines a connection one-form by definition. Now, we observe that

\begin{center}
    
$d\eta'_{\alpha} =  \pi_{Q(\omega_{\alpha})}^{\ast}d\mathcal{A}_{i} = \pi_{Q(\omega_{\alpha})}^{\ast}dA_{i},$
    
\end{center}
since $dA_{i} = F_{\nabla}$, for $\nabla = d + A_{i}$ (locally), it follows that

\begin{center}

$\displaystyle \frac{\sqrt{-1}}{2\pi}d\eta'_{\alpha} = \frac{\sqrt{-1}}{2\pi} \pi_{Q(\omega_{\alpha})}^{\ast} F_{\nabla},$

\end{center}
thus from Proposition \ref{C8S8.2Sub8.2.3P8.2.7} we obtain Equation \ref{eqconnections}.
\end{proof}

\begin{remark}
In what follows we shall denote by $\mathcal{A} = (\mathcal{A}_{i})_{i \in I}$ the collection of (gauge) potentials obtained by the result above. We also will denote by $d\mathcal{A} \in \Omega^{1,1}(X_{P})$ the globally defined $(1,1)$-form associated to $\mathcal{A}$.
\end{remark}

The description provided by Proposition \ref{locconnection} will be fundamental for our next step to describe contact structures on homogeneous contact manifolds.

\subsection{Examples}
\label{subsec3.3}
Let us illustrate the previous results, especially Proposition \ref{locconnection}, by means of basic examples.

\begin{example}[Hopf bundle]
\label{HOPFBUNDLE}
Consider $G^{\mathbb{C}} = {\rm{SL}}(2,\mathbb{C})$ and $P = B \subset {\rm{SL}}(2,\mathbb{C})$ as in Example \ref{exampleP1}. As we have seen, in this case we have

\begin{center}
    
$X_{B} = \mathbb{C}{\rm{P}}^{1}$ \ \ and \ \ $\mathscr{P}(\mathbb{C}{\rm{P}}^{1},{\rm{U}}(1)) = \mathbb{Z}\mathrm{e}(Q(\omega_{\alpha}))$,
    
\end{center}
where $Q(\omega_{\alpha}) = Q(\mathscr{O}(1))$. Since $K_{\mathbb{C}{\rm{P}}^{1}}^{\otimes \frac{1}{2}} = \mathscr{O}(-1)$, it follows that 

\begin{center}
    
$Q(-\omega_{\alpha}) = S^{3}$.    
    
\end{center}
By considering the opposite big cell $U =  N^{-}x_{0} \subset X_{B}$ and the local section $s_{U} \colon U \subset \mathbb{C}{\rm{P}}^{1} \to {\rm{SL}}(2,\mathbb{C})$ defined by

\begin{center}

$s_{U}(nx_{0}) = n$, \ \ $\forall n \in N^{-}$,
\end{center}
we obtain from Proposition \ref{locconnection} the following local expression for the gauge potential

\begin{center}

$\mathcal{A}_{U} = \displaystyle \frac{1}{2}\big ( \partial - \overline{\partial} \big ) \log||s_{U} v_{\omega_{\alpha}^{+}}||^{2},$    
    
\end{center}
on the opposite big cell $U \subset \mathbb{C}{\rm{P}}^{1}$, thus we have

\begin{center}
    
$\mathcal{A}_{U} =  \displaystyle  \frac{1}{2} \big ( \partial - \overline{\partial} \big ) \log \Bigg ( \Big |\Big |\begin{pmatrix}
 1 & 0 \\
 z & 1
\end{pmatrix} e_{1} \Big| \Big|^{2} \Bigg ) = \displaystyle  - \frac{1}{2}\frac{zd\overline{z} - \overline{z}dz}{(1+|z|^{2})}.$    
    
\end{center}
Hence, we have a principal ${\rm{U}}(1)$-connection on $Q(-\omega_{\alpha}) = S^{3}$ (locally) defined by

\begin{center}

$\eta'_{\alpha} = \displaystyle  - \frac{1}{2}\frac{zd\overline{z} - \overline{z}dz}{(1+|z|^{2})} + \frac{da_{U}}{a_{U}}.$    
    
\end{center}
Therefore, we obtain 

\begin{center}
$\mathrm{e}(S^{3}) = \displaystyle \Big [\frac{\sqrt{-1}}{2\pi}d\mathcal{A} \Big ] \in H^{2}( \mathbb{C}{\rm{P}}^{1},\mathbb{Z}).$
\end{center}
It is worth pointing out that from the ideas above, given $Q \in \mathscr{P}(\mathbb{C}{\rm{P}}^{1},{\rm{U}}(1))$, it follows that $Q = Q(-\ell \omega_{\alpha})$, for some $\ell \in \mathbb{Z}$, thus we have

\begin{center}
    
$Q = S^{3}/\mathbb{Z}_{\ell}$ \ \  and \ \ $\mathrm{e}(Q) = \displaystyle \Big [\frac{\ell \sqrt{-1}}{2\pi}d\mathcal{A} \Big ] \in H^{2}( \mathbb{C}{\rm{P}}^{1},\mathbb{Z}).$
    
\end{center}
Hence, we obtain the Euler class of the principal circle bundle defined by $ Q(-\ell \omega_{\alpha}) = S^{3}/\mathbb{Z}_{\ell}$ (Lens space).
\end{example}

\begin{example}[Complex Hopf fibrations]
\label{COMPLEXHOPF}
The previous example can be easily generalized. In fact, consider the basic data as in Example \ref{examplePn}, namely, the complex simple Lie group $G^{\mathbb{C}} = {\rm{SL}}(n+1,\mathbb{C})$ and the parabolic Lie subgroup $P = P_{\Sigma \backslash \{\alpha_{1}\}}$. As we have seen, in this case we have

\begin{center}
    
$X_{P_{\Sigma \backslash \{\alpha_{1}\}}} = \mathbb{C}{\rm{P}}^{n}$ \ \  and  \ \ $\mathscr{P}(\mathbb{C}{\rm{P}}^{n},{\rm{U}}(1)) = \mathbb{Z}\mathrm{e}(Q(\omega_{\alpha_{1}}))$,
    
\end{center}
where $Q(\omega_{\alpha_{1}}) = Q(\mathscr{O}(1))$. Since $K_{\mathbb{C}{\rm{P}}^{n}}^{\otimes \frac{1}{n+1}} = \mathscr{O}(-1)$, it follows that 

\begin{center}
    
$Q(-\omega_{\alpha_{1}}) = S^{2n+1}$.    
    
\end{center}
From Proposition \ref{locconnection} and a similar computation as in the previous example, we have

\begin{center}
    
$\mathcal{A}_{U} = \displaystyle \frac{1}{2} \big ( \partial - \overline{\partial} \big )\log \Big (1 + \sum_{l = 1}^{n}|z_{l}|^{2} \Big ),$
    
\end{center}
on the opposite big cell $U \subset \mathbb{C}{\rm{P}}^{n}$. Therefore, we obtain a principal ${\rm{U}}(1)$-connection on $Q(-\omega_{\alpha_{1}}) = S^{2n+1}$ (locally) defined by

\begin{center}

$\eta'_{\alpha_{1}} = \displaystyle  - \frac{1}{2} \sum_{l = 1}^{n}\frac{z_{l}d\overline{z}_{l} - \overline{z}_{l}dz_{l}}{\big (1 + \sum_{l = 1}^{n}|z_{l}|^{2} \big )} + \frac{da_{U}}{a_{U}},$    
    
\end{center}
thus we have 

\begin{center}
$\mathrm{e}(S^{2n+1}) = \displaystyle \Big [\frac{\sqrt{-1}}{2\pi}d\mathcal{A} \Big ] \in H^{2}( \mathbb{C}{\rm{P}}^{n},\mathbb{Z}).$
\end{center}
It is worth pointing out that given $Q \in \mathscr{P}(\mathbb{C}{\rm{P}}^{n},{\rm{U}}(1))$, it follows that $Q = Q(-\ell \omega_{\alpha_{1}})$, for some $\ell \in \mathbb{Z}$, thus we have

\begin{center}
    
$Q = S^{2n+1}/\mathbb{Z}_{\ell}$ \ \  and \ \ $\mathrm{e}(Q) = \displaystyle \Big [\frac{\ell \sqrt{-1}}{2\pi}d\mathcal{A} \Big ] \in H^{2}( \mathbb{C}{\rm{P}}^{n},\mathbb{Z}).$
    
\end{center}
Hence, we obtain the Euler class of the principal circle bundle defined by the Lens space $ Q(-\ell \omega_{\alpha_{1}}) = S^{2n+1}/\mathbb{Z}_{\ell}$.
\end{example}

\begin{example}[Stiefel manifold]
\label{STIEFEL}
Now, consider $G^{\mathbb{C}} = {\rm{SL}}(4,\mathbb{C})$ and $P = P_{\Sigma \backslash \{\alpha_{2}\}}$ as in Example \ref{grassmanian}. In this case we have

\begin{center}
$X_{P_{\Sigma \backslash \{\alpha_{2}\}}} = {\rm{Gr}}(2,\mathbb{C}^{4})$ \ \ and \ \ $\mathscr{P}({\rm{Gr}}(2,\mathbb{C}^{4}),{\rm{U}}(1)) = \mathbb{Z}\mathrm{e}(Q(\omega_{\alpha_{2}})),$
\end{center}
where $Q(\omega_{\alpha_{2}}) = Q(\mathscr{O}_{\alpha_{2}}(1))$. Since $K_{{\rm{Gr}}(2,\mathbb{C}^{4})}^{\otimes \frac{1}{4}} = \mathscr{O}_{\alpha_{2}}(-1)$, it follows that 

\begin{center}
    
$Q(-\omega_{\alpha_{2}}) = \mathscr{V}_{2}(\mathbb{R}^{6})$ \ \ (Stiefel manifold).    
    
\end{center}
From Proposition \ref{locconnection}, and the computations in Example \ref{grassmanian}, we obtain

\begin{center}
    
$\mathcal{A}_{U} = \displaystyle \frac{1}{2}\big ( \partial - \overline{\partial} \big ) \log \Big ( 1 + \sum_{k = 1}^{4}|z_{k}|^{2} + \bigg |\det \begin{pmatrix}
 z_{1} & z_{3} \\
 z_{2} & z_{4}
\end{pmatrix} \bigg |^{2} \Big),$
    
\end{center}
on the opposite big cell $U \subset {\rm{Gr}}(2,\mathbb{C}^{4})$. Hence, we have a principal ${\rm{U}}(1)$-connection on $Q(-\omega_{\alpha_{2}}) = \mathscr{V}_{2}(\mathbb{R}^{6})$ (locally) defined by

\begin{center}

$\eta'_{\alpha_{2}} = \displaystyle \frac{1}{2}\big ( \partial - \overline{\partial} \big ) \log \Big ( 1 + \sum_{k = 1}^{4}|z_{k}|^{2} + \bigg |\det \begin{pmatrix}
 z_{1} & z_{3} \\
 z_{2} & z_{4}
\end{pmatrix} \bigg |^{2} \Big) + \frac{da_{U}}{a_{U}}.$    
    
\end{center}
Therefore, we have 

\begin{center}
$\mathrm{e}(\mathscr{V}_{2}(\mathbb{R}^{6})) = \displaystyle \Big [\frac{\sqrt{-1}}{2\pi}d\mathcal{A} \Big ] \in H^{2}(  {\rm{Gr}}(2,\mathbb{C}^{4}),\mathbb{Z}).$
\end{center}
Notice that, given $Q \in \mathscr{P}({\rm{Gr}}(2,\mathbb{C}^{4}),{\rm{U}}(1))$, it follows that $Q = Q(-\ell \omega_{\alpha_{2}})$, for some $\ell \in \mathbb{Z}$, thus we have

\begin{center}
    
$Q =\mathscr{V}_{2}(\mathbb{R}^{6})/\mathbb{Z}_{\ell}$ \ \  and \ \ $\mathrm{e}(Q) = \displaystyle \Big [\frac{\ell \sqrt{-1}}{2\pi}d\mathcal{A} \Big ] \in H^{2}( {\rm{Gr}}(2,\mathbb{C}^{4}),\mathbb{Z}).$
    
\end{center}
Hence, we obtain the Euler class of the principal circle bundle defined by $ Q(-\ell \omega_{\alpha}) =\mathscr{V}_{2}(\mathbb{R}^{6})/\mathbb{Z}_{\ell}$.
\end{example}

Let us explain how the examples above fit inside of a more general setting. Let $G^{\mathbb{C}}$ be a complex simply connected simple Lie group, and consider $P \subset G^{\mathbb{C}}$ as being a parabolic Lie subgroup. If we suppose that $P = P_{\Sigma \backslash \{\alpha\}}$, i.e. $P$ is a maximal parabolic subgroup, then we have 

\begin{center}
    
$\mathscr{P}(X_{ P_{\Sigma \backslash \{\alpha\}}},{\rm{U}}(1)) = \mathbb{Z}{\mathrm{e}}(Q(\omega_{\alpha})).$
    
\end{center}
In order to simplify the notation, let us denote $P_{\Sigma \backslash \{\alpha\}}$ by $P_{\omega_{\alpha}}$. A straightforward computation shows that 

\begin{center}
    
$I(X_{P_{\omega_{\alpha}}}) = \langle \delta_{P_{\omega_{\alpha}}},h_{\alpha}^{\vee} \rangle$ \ \ and \ \ $K_{X_{P_{\omega_{\alpha}}}}^{ \otimes \frac{1}{ \langle \delta_{P_{\omega_{\alpha}}},h_{\alpha}^{\vee} \rangle}} = L_{\chi_{\omega_{\alpha}}}^{-1},$

\end{center}
thus we have
\begin{equation}
\label{maxparaboliccase}
Q(K_{X_{P_{\omega_{\alpha}}}}^{ \otimes \frac{1}{ \langle \delta_{P_{\omega_{\alpha}}},h_{\alpha}^{\vee} \rangle}}) = Q(-\omega_{\alpha}).    
\end{equation}
Now, consider the following definition. 
\begin{definition}[\cite{MINUSCULE}, \cite{SARA}]
A fundamental weight $\omega_{\alpha}$ is called minuscule if it satisfies the condition

\begin{center}

$\langle \omega_{\alpha},h_{\beta}^{\vee} \rangle \in \{0,1\}, \ \forall \beta \in \Pi^{+}.$
    
\end{center}
A flag manifold $X_{P_{\omega_{\alpha}}}$ associated to a maximal parabolic subgroup $P_{\omega_{\alpha}}$ is called minuscule flag manifold if $\omega_{\alpha}$ is a minuscule weight.
\end{definition}
The flag manifolds of the previous examples are particular cases of flag manifolds defined by maximal parabolic Lie subgroups. Being more specific, they are examples of minuscule flag manifolds. Moreover, examples of flag manifolds associated to maximal parabolic Lie subgroups include Grassmannian manifolds ${\rm{Gr}}(k,\mathbb{C}^{n})$, odd dimensional quadrics $\mathbb{Q}^{2n-1}$, even dimensional quadrics $\mathbb{Q}^{2n-2}$, Lagrangian Grassmannian manifolds ${\rm{LGr}}(n,2n)$, Orthogonal Grassmannian manifolds ${\rm{OGr}}(n,2n)$, Cayley plane $\mathbb{O}{\rm{P}}^{2}$ and the Freudental variety $ {\rm{E}}_{7}/P_{\omega_{7}}$.

\section{Homogeneous contact structures and Sasaki-Einstein metrics}
\label{sec4}

In this section we provide the proofs for the main results of this paper, which are essentially based on the description of contact structures on homogeneous contact manifolds by means of the transversal K\"{a}hler geometry of flag manifolds. 

The results are presented as follows: In Subsection \ref{basiccase}, we provide an outline of how the results developed in the previous sections can be combined in order to obtain an expression for the contact structure in the particular case of flag manifolds defined by maximal parabolic Lie subgroups. In Subsection \ref{subsec4.2}, we prove Theorem \ref{Theo1}. This theorem provides a complete description of contact structures on compact homogeneous contact manifolds. In Subsection \ref{subsec4.3}, we show how the result of Theorem \ref{Theo1} can be used to describe Sasakian-Einstein structures on compact homogeneous contact manifolds and the induced Calabi-Yau structures on their symplectizations. The main goal of this last subsection is to prove Theorem \ref{Theo2} and Theorem \ref{Theo3}.

\subsection{Basic model}
\label{basiccase}
As mentioned above, in this section we will prove the main results of this work. In order to motivate the ideas involved in the proofs which we shall cover in the next subsections, let us start by recalling some basic facts.

As we have seen, given a compact homogeneous contact manifold $(M,\eta,G)$ we have that 

\begin{center}
    
$M = Q(L),$
    
\end{center}
for some ample line bundle $L^{-1} \in {\text{Pic}}(X_{P})$, and under the assumption that $c_{1}(L^{-1})$ defines a K\"{a}hler-Einstein metric on $X_{P} = G^{\mathbb{C}}/P$, we have

\begin{center}
$L = K_{X_{P}}^{\otimes \frac{\ell}{I(X_{P})}},$
\end{center}
for some $\ell \in \mathbb{Z}_{>0}$.

The examples of compact homogeneous contact manifolds associated to flag manifolds defined by maximal parabolic Lie subgroups will be useful for us in the next subsections. In what follows we shall further explore these particular examples. As we have seen, from \ref{maxparaboliccase}, if $P = P_{\omega_{\alpha}}$ it follows that 

\begin{center}
    
$M = Q(- \ell \omega_{\alpha}) = Q(- \omega_{\alpha})/\mathbb{Z}_{\ell},$
    
\end{center}
for some $\ell \in \mathbb{Z}_{>0}$. Hence, from Proposition \ref{locconnection} we have a connection $\eta'_{\alpha}$ defined on $Q(- \ell \omega_{\alpha})$ by

\begin{center}
$\eta_{\alpha}' =  \displaystyle \frac{\ell}{2}\big ( \partial - \overline{\partial} \big ) \log||s_{i} v_{\omega_{\alpha}}^{+}||^{2}+ \frac{da_{i}}{a_{i}},$
\end{center}
thus a contact structure on $M = Q(-\ell \omega_{\alpha})$ is obtained from $\eta = - \sqrt{-1} \eta'_{\alpha}$. If we consider $a_{i} = \mathrm{e}^{\sqrt{-1}\theta_{i}}$, where $\theta_{i}$ is real and defined up to an integral multiple of $2 \pi$, we have that 

\begin{center}

$\eta = \displaystyle - \frac{\ell \sqrt{-1}}{2}\big ( \partial - \overline{\partial} \big ) \log||s_{i} v_{\omega_{\alpha}}^{+}||^{2} + d\theta_{i},$    
    
\end{center}
it is not difficult to see that

\begin{center}

$d\eta = 2\pi \ell \Omega_{\alpha},$    
    
\end{center}
for the sake of simplicity we omitted the pullback of the projection map in the equality above. This particular case turns out to be the basic model for all the cases which we have described in the examples of the previous sections. As we will see later, the ideas developed above are essentially the model for the general case of compact homogeneous contact manifolds. In the next subsections we will come back to this basic example several times.

\subsection{Homogeneous contact structures}
\label{subsec4.2}
In this subsection we prove the following result. 

\begin{theorem}
\label{main1}
Let $(M,\eta,G)$ be a compact connected homogeneous contact manifold, then $M$ is the principal $S^{1}$-bundle given by the sphere bundle 
\begin{equation}
\label{contacthomo}
M = \Big \{ u \in L \ \Big | \ \sqrt{H(u,u)} = 1 \Big\} ,
\end{equation}
for some ample line bundle $L^{-1} \in {\text{Pic}}(X_{P})$, where $X_{P} = G^{\mathbb{C}}/P$ is a flag manifold defined by some parabolic Lie subgroup $P \subset G^{\mathbb{C}}$. Furthermore, if $c_{1}(L^{-1})$ defines a K\"{a}hler-Einstein metric on $X_{P}$, it follows that $M = Q(K_{X_{P}}^{\otimes \frac{\ell}{I(X_{P})}})$, for some $\ell \in \mathbb{Z}_{>0}$, and its contact structure $\eta$ is (locally) given by
\begin{equation}
\label{structure}
\eta = \displaystyle - \frac{\ell \sqrt{-1}}{2I(X_{P})}\big ( \partial - \overline{\partial} \big )\log \big | \big |s_{U}v_{\delta_{P}}^{+} \big| \big |^{2} + d\theta_{U},
\end{equation}
for some local section $s_{U} \colon U \subset X_{P} \to G^{\mathbb{C}}$, where $v_{\delta_{P}}^{+}$ denotes the highest weight vector of weight $\delta_{P}$ associated to the irreducible $\mathfrak{g}^{\mathbb{C}}$-module $V(\delta_{P})$.
\end{theorem}

\begin{proof}
The characterization \ref{contacthomo} follows from the Boothby-Wang fibration \ref{BWhomo}, see also Remark \ref{amplecase}. Thus, from Theorem \ref{AZADBISWAS}, and Proposition \ref{C8S8.2Sub8.2.3P8.2.6}, we have $M = Q(L)$, such that 
\begin{equation}
\label{veryample}
L^{-1} = \bigotimes_{\alpha \in \Sigma \backslash \Theta}L_{\chi_{\omega_{\alpha}}}^{\otimes \ell_{\alpha}},
\end{equation}
where $\ell_{\alpha} \in \mathbb{Z}_{>0}$, $\forall \alpha \in \Sigma \backslash \Theta$. Notice that $\eta \in \Omega^{1}(M)^{G}$ such that 
\begin{equation}
\label{curvatureveryample}
d\eta = \sum_{\alpha \in \Sigma \backslash \Theta} 2 \pi \ell_{\alpha}\pi^{\ast}\Omega_{\alpha},
\end{equation}
where $\Omega_{\alpha} \in \Omega^{1,1}(X_{P})^{G}$, $\forall \alpha \in \Sigma \backslash \Theta$, see \ref{basicforms}. Thus, we have that $c_{1}(L^{-1}) > 0$, which provides an explicit description of the $G$-invariant Hodge metric induced by $M$ on $X_{P}$ via Theorem \ref{AZADBISWAS}.

Now, If we suppose that $c_{1}(L^{-1})$ defines a K\"{a}hler-Einstein metric on $X_{P}$, it follows from Proposition \ref{CONTACTSIMPLY}, and Theorem \ref{contacthomogeneousclass}, that  

\begin{center}
    
$M = Q(K_{X_{P}}^{\otimes \frac{\ell}{I(X_{P})}}).$
    
\end{center}
From Proposition \ref{C8S8.2Sub8.2.3Eq8.2.35}, and Remark \ref{productcircle}, we obtain

\begin{center}

$M = Q(K_{X_{P}}^{\otimes \frac{\ell}{I(X_{P})}}) = \displaystyle \sum_{\alpha \in \Sigma \backslash \Theta}Q \big ( \textstyle{- \frac{\ell \langle \delta_{P},h_{\alpha}^{\vee} \rangle}{I(X_{P})}} \omega_{\alpha} \big).$
    
\end{center}
Therefore, from Proposition \ref{connection} we have a connection one-form on $M$ defined by

\begin{center}
    
$\eta' = \displaystyle \sum_{\alpha \in \Sigma \backslash \Theta}  \frac{\ell  \langle \delta_{P},h_{\alpha}^{\vee} \rangle}{2 I(X_{P})} ( \partial - \overline{\partial} \big ) \log||s_{U} v_{\omega_{\alpha}}^{+}||^{2}+ a_{U}^{-1}da_{U},$
    
\end{center}
thus our contact structure is $\eta = - \sqrt{-1} \eta'$. If we consider $a_{U} = \mathrm{e}^{\sqrt{-1}\theta_{U}}$, where $\theta_{U}$ is real and is defined up to an integral multiple of $2 \pi$, by rearranging the expression above we obtain

\begin{center}
    
$\eta = \displaystyle - \frac{\ell \sqrt{-1}}{2 I(X_{P})}\big ( \partial - \overline{\partial} \big )\log \Big ( \prod_{\alpha \in \Sigma \backslash \Theta} ||s_{U}v_{\omega_{\alpha}}^{+}||^{2  \langle \delta_{P},h_{\alpha}^{\vee} \rangle}\Big) + d\theta_{U}.$
    
\end{center}
Now, we recall some basic facts about representation theory of simple Lie algebras \cite[p. 186]{PARABOLICTHEORY}.

\begin{enumerate}

    \item $V(\delta_{P}) \subset \bigotimes_{\alpha \in \Sigma \backslash \Theta} V(\omega_{\alpha})^{ \otimes \langle \delta_{P},h_{\alpha}^{\vee} \rangle}$; 
    
    \item $v_{\delta_{P}}^{+} = \bigotimes_{\alpha \in \Sigma \backslash \Theta}v_{\omega_{\alpha}}^{+ \otimes \langle \delta_{P},h_{\alpha}^{\vee} \rangle}$, where $v_{\omega_{\alpha}}^{+} \in V(\omega_{\alpha})$ is the highest weight vector of highest weight $\omega_{\alpha}$, $\forall \alpha \in \Sigma \backslash \Theta$.
    
\end{enumerate}
From these two facts, by considering  the $G$-invariant inner product $\langle \cdot, \cdot \rangle_{\alpha}$ on each fundamental $\mathfrak{g}^{\mathbb{C}}$-module $V(\omega_{\alpha})$, see Remark \ref{innerproduct}, we have a $G$-invariant inner product on the Cartan product of fundamental representations 

\begin{center}
$\displaystyle \bigotimes_{\alpha \in \Sigma \backslash \Theta} V(\omega_{\alpha})^{ \otimes \langle \delta_{P},h_{\alpha}^{\vee} \rangle},$ 
\end{center}
defined naturally by

\begin{center}
    
$\langle \cdot, \cdot \rangle =  \displaystyle \prod_{\alpha \in \Sigma \backslash \Theta} \langle \cdot, \cdot \rangle_{\alpha}^{\langle \delta_{P},h_{\alpha}^{\vee} \rangle}.$
    
\end{center}
The inner product described above restricted to $V(\delta_{P})$ defines a norm such that

\begin{center}
$||v_{\delta_{P}}^{+}||^{2} = \displaystyle \prod_{\alpha \in \Sigma \backslash \Theta}||v_{\omega_{\alpha}}^{+}||^{2  \langle \delta_{P},h_{\alpha}^{\vee} \rangle},$
\end{center}
thus we obtain

\begin{center}
$\eta = \displaystyle - \frac{\ell \sqrt{-1}}{2I(X_{P})}\big ( \partial - \overline{\partial} \big )\log ||s_{U}v_{\delta_{P}}^{+}||^{2} + d\theta_{U},$
\end{center}
from this we have the desired expression \ref{structure}.
\end{proof}

\begin{remark}
\label{contactample}
Notice that in the general case that $M = Q(L)$, for some ample line bundle $L^{-1} \in \text{Pic}(X_{P})$, it follows from \ref{veryample}, and \ref{curvatureveryample}, that
\begin{equation}
\label{contactveryample}
\eta = - \displaystyle \frac{\sqrt{-1}}{2}\big ( \partial - \overline{\partial} \big ) \log \Big ( \prod_{\alpha \in \Sigma \backslash \Theta} ||s_{U}v_{\omega_{\alpha}}^{+}||^{2  \ell_{\alpha}}\Big) + d\theta_{U},
\end{equation}
for some local section $s_{U} \colon U \subset X_{P} \to G^{\mathbb{C}}$. Therefore, Theorem \ref{main1} allows us to describe explicitly the contact structure for any compact homogeneous contact manifold.
\end{remark}

As we can see, in Equation \ref{contactveryample} the contact structure of a compact homogeneous contact manifold can be completely described by elements of representation theory and some geometric structures associated to the parabolic Cartan geometry $(G^{\mathbb{C}},P)$.

We recall that when $(M,\eta,G)$ is a simply connected compact homogeneous contact manifold, from Proposition \ref{C8S8.2Sub8.2.3Eq8.2.35}, and the convention \ref{notationsimple}, under the assumption of the Einstein condition, we have

\begin{center}
    
$M = Q(-\frac{\delta_{P}}{I(X_{P})})$.
    
\end{center}
For the sake of simplicity, we shall denote
\begin{equation}
\label{maximalroot}
\mathcal{Q}_{P} := Q(-\textstyle{\frac{\delta_{P}}{I(X_{P})}}),
\end{equation}
to stand for a simply connected homogeneous contact manifold associated to the maximal root of the canonical bundle of $X_{P}$. The next section will be devoted to study the contact manifold $(\mathcal{Q}_{P},\eta,G)$ and some immediate consequences of Theorem \ref{main1}.

\subsection{Sasaki-Einstein structures and Calabi-Yau cones} 
\label{subsec4.3}
Let $(\mathcal{Q}_{P},\eta,G)$ be the simply connected compact homogeneous contact manifold as in \ref{maximalroot}. From Theorem \ref{structure} and Equation \ref{Cherncanonical} we have

\begin{center}
    
$\displaystyle \frac{d\eta}{2 \pi} = \frac{1}{I(X_{P})} \pi_{\mathcal{Q}_{P}}^{\ast} \omega_{X_{P}},$
    
\end{center}
recall the expression of $\omega_{X_{P}}$ from \ref{localform}. The equation above essentially tells us that $\mathrm{e}(\mathcal{Q}_{P}) = -\frac{1}{I(X_{P})}c_{1}(X_{P})$. Since ${\text{Ric}}(\omega_{X_{P}}) = 2 \pi \omega_{X_{P}}$, we consider the following rescaled K\"{a}hler metric on $X_{P}$

\begin{equation}
\label{newkahler}
\widetilde{\omega}_{X_{P}} = \frac{\pi}{n+1} \omega_{X_{P}},     
\end{equation}
where $\dim_{\mathbb{C}}(X_{P}) = n$. From this, since we have ${\text{Ric}}(c \omega_{X_{P}}) = {\text{Ric}}(\omega_{X_{P}})$, $\forall c>0$, it follows that 

\begin{center}
    
${\text{Ric}}(\widetilde{\omega}_{X_{P}}) = 2(n+1)\widetilde{\omega}_{X_{P}}$,
    
\end{center}
thus the metric induced by $\widetilde{\omega}_{X_{P}}$ has scalar curvature $S_{\widetilde{\omega}_{X_{P}}} = 4n(n+1)$. Hence, if we take the connection $\eta' = \sqrt{-1}\eta$ on $\mathcal{Q}_{P}$, we obtain
\begin{equation}
\label{YM}
d\eta' = \displaystyle \frac{2(n+1)}{I(X_{P})}\sqrt{-1}\pi_{\mathcal{Q}_{P}}^{\ast}\widetilde{\omega}_{X_{P}}.
\end{equation}
From these we have the following result.

\begin{theorem}
\label{main2}
Let $(M = \mathcal{Q}_{P} / \mathbb{Z}_{\ell},\eta,G)$ be a compact connected homogeneous contact manifold. Then, $(M = \mathcal{Q}_{P} / \mathbb{Z}_{\ell},\eta,G)$ admits a homogeneous Sasaki-Einstein structure $(g_{M}, \phi,\xi = \frac{\ell(n+1)}{I(X_{P})}\frac{\partial}{\partial \theta}, \frac{I(X_{P})}{\ell(n+1)}\eta )$, such that 
\begin{equation}
g_{M} = \displaystyle \frac{I(X_{P})}{\ell(n+1)} \Bigg ( \frac{1}{2}d \eta ({\rm{id}} \otimes \phi)  + \frac{I(X_{P})}{\ell(n+1)}\eta \otimes \eta \Bigg ),
\end{equation}
where 
\begin{center}

$\eta = \displaystyle - \frac{\ell\sqrt{-1}}{2I(X_{P})}\big ( \partial - \overline{\partial} \big )\log \big | \big |s_{U}v_{\delta_{P}}^{+} \big| \big |^{2} + d\theta_{U},$
\end{center}
for some local section $s_{U} \colon U \subset X_{P} \to G^{\mathbb{C}}$, where $v_{\delta_{P}}^{+}$ denotes the highest weight vector of weight $\delta_{P}$ associated to the irreducible $\mathfrak{g}^{\mathbb{C}}$-module $V(\delta_{P})$. Furthermore, we also have $\phi \in {\text{End}}(TM)$ completely determined by the invariant complex structure of $X_{P}$ and the horizontal lift of the Cartan-Ehresmann connection $ \frac{I(X_{P})\sqrt{-1}}{\ell(n+1)}\eta \in \Omega^{1}(M;\mathfrak{u}(1))$.

\end{theorem}

\begin{proof}
We first consider the case where $M$ is simply connected, i.e. $M = \mathcal{Q}_{P}$. The proof is essentially an application of the general construction described in Example \ref{EXAMPLESASAKIAN}. Let us outline the main ideas involved. Consider $\xi \in \Gamma(T\mathcal{Q}_{P})$ as being the Reeb vector field defined by the homogeneous contact structure $\frac{I(X_{P})}{(n+1)}\eta \in \Omega^{1}(\mathcal{Q}_{P})^{G}$, namely $\xi = \frac{n+1}{I(X_{P})} \frac{\partial}{\partial \theta}$. We define $\phi \in {\text{End}}(T\mathcal{Q}_{P})$ by setting
\begin{center}
$ \phi(X) := \begin{cases}
    (J\pi_{\ast}X)^{H}, \ \ \ {\text{if}}  \ \ X \bot \xi.\\
    \ \ \ \ \  0 \  \  \ \ \ \ \ ,  \ \  \ {\text{if}} \ \ X \parallel \xi .                   \\
  \end{cases}$
    
\end{center}
Here we denote by $(J\pi_{\ast}X)^{H}$ the horizontal lift of $J\pi_{\ast}X$ relative to the connection $\frac{I(X_{P}) \sqrt{-1}}{(n+1)}\eta$, $\forall X \in \Gamma(T\mathcal{Q}_{P})$.

Notice that the metric $g_{M}$ is given by

\begin{center}
$g_{M} = \pi^{\ast}\widetilde{g}_{X_{P}} + \frac{I(X_{P})^{2}}{(n+1)^{2}}\eta \otimes \eta,$
\end{center}
where $\widetilde{g}_{X_{P}} = \widetilde{\omega}_{X_{P}}({\text{id}} \otimes J)$ is the K\"{a}hler metric on $X_{P}$ as in \ref{newkahler}, thus we also have $\frac{I(X_{P})}{(n+1)}\eta = g_{M}(\xi,\cdot)$.

In order to simplify the notation, let us denote $\overline{\eta} = \frac{I(X_{P})}{(n+1)}\eta$. The fact that $\phi \circ \phi = - {\rm{id}} + \overline{\eta} \otimes \xi$ follows from its definition. Moreover, by definition of $\phi \in {\text{End}}(T\mathcal{Q}_{P})$ it is straightforward to check that
\begin{center}
    
$g_{M}(\phi \otimes \phi) = g_{M} - \overline{\eta} \otimes \overline{\eta}$ \ \ and \ \ $d \overline{\eta} = 2g_{M}(\phi \otimes {\text{id}}),$
    
\end{center}
the identities above follow from the fact that $\overline{\eta}(\phi(X)) = 0$, $\forall X \in \Gamma(T\mathcal{Q}_{P})$, and $\phi \circ (\overline{\eta} \otimes \xi) \equiv 0$. Hence, we have that $(g_{M},\phi,\xi,\overline{\eta})$ defines a contact metric structure on $\mathcal{Q}_{P}$.

Now, since $\mathscr{L}_{\xi} d \overline{\eta} ({\rm{id}} \otimes \phi) = 0$, and

\begin{center}
    
$\mathscr{L}_{\xi} \overline{\eta} = \iota_{\xi}(d\overline{\eta}) + d(\iota_{\xi} \overline{\eta}) = d \overline{\eta}(\xi,\cdot) = 0,$
    
\end{center}
it follows that $\mathscr{L}_{\xi} g_{M} = 0\Longrightarrow \xi$ is a Killing vector. Therefore, we have that $(g_{M},\phi,\xi,\overline{\eta})$ defines a ${\text{K}}$-contact structure on $\mathcal{Q}_{P}$. The fact that $(g_{M},\phi,\xi,\overline{\eta})$ is a Sasaki-Einstein structure is a consequence of the fact that $(X_{P},\widetilde{\omega}_{X_{P}})$ is a K\"{a}hler-Einstein manifold with scalar curvature $S_{\widetilde{\omega}_{X_{P}}} = 4n(n+1)$, see for instance  \cite[Theorem 7.3.12]{BOYERGALICKI}, and \cite[Theorema 2]{HATAKEYMA}.

Now, if we consider the compact homogeneous contact manifold $(M,\eta,G)$ such that

\begin{center}
    
$M = \mathcal{Q}_{P} / \mathbb{Z}_{\ell},$
    
\end{center}
where $\pi_{1}(M) = \mathbb{Z}_{\ell}$, for some parabolic Lie subgroup $P \subset G^{\mathbb{C}}$. In order to obtain the desired structure, we just rescale the metric $\omega_{X_{P}}$, as in \ref{newkahler}, and take $\overline{\eta} = \frac{I(X_{P})}{\ell (n+1)} \eta$, notice that here we consider $\eta$ as in \ref{structure}. From these, the result follows from the same arguments as in the simply connected case.
\end{proof}

\begin{remark}

As we have seen in the proof of the result above, since we have $\frac{I(X_{P})}{\ell(n+1)}d\eta = 2 \widetilde{\omega}_{X_{P}}$, see \ref{newkahler}, it follows that $\frac{I(X_{P}) \sqrt{-1}}{\ell(n+1)}\eta$ is a Yang-Mills connection. Therefore, from O'Neill's formulas for Riemannian submersions (e.g. \cite{ONEILL}) we can show that $g_{M}$ is the unique Einstein metric on $M = \mathcal{Q}_{P}/\mathbb{Z}_{\ell}$ naturally defined by $\widetilde{\omega}_{X_{P}}$ and $\eta$, i.e., horizontally determined by $\widetilde{\omega}_{X_{P}}$ and vertically determined by the length of ${\rm{U}}(1) = S^{1}$, see for instance \cite{KOBAYASHI}.
\end{remark}

\begin{remark}
\label{sasakiample}
It is worth pointing out that in the general case when $M = Q(L)$, for some ample line bundle $L^{-1} \in \text{Pic}(X_{P})$, we can use the connection induced by \ref{contactveryample} to obtain an explicit Sasaki structure on $M = Q(L)$ described in terms of Lie theory, see also Remark \ref{amplecase}. 

\end{remark}

In what follows we provide some examples in order to illustrate Theorem \ref{main2}.

\begin{example}[Basic model] As in Subsection \ref{basiccase}, consider the principal circle bundle

\begin{center}

${\rm{U}}(1) \hookrightarrow Q(-\ell \omega_{\alpha}) \to X_{P_{\omega_{\alpha}}}.$

\end{center}
As we have seen, in this case we have a connection $\eta'_{\alpha}$ defined on $Q(- \ell \omega_{\alpha})$ such that 

\begin{center}
$\eta_{\alpha}' =  \displaystyle \frac{\ell}{2}\big ( \partial - \overline{\partial} \big ) \log||s_{i} v_{\omega_{\alpha}}^{+}||^{2}+ a_{i}^{-1}da_{i}.$
\end{center}
Now, by applying Theorem \ref{main1}, we obtain a contact $1$-form on $Q(-\ell \omega_{\alpha})$ given by

\begin{center}

$\eta = \displaystyle - \frac{\ell \sqrt{-1}}{2}\big ( \partial - \overline{\partial} \big ) \log||s_{i} v_{\omega_{\alpha}}^{+}||^{2} + d\theta_{i},$    
    
\end{center}
recall that $I(X_{P_{\omega_{\alpha}}}) = \langle \delta_{P_{\omega_{\alpha}}},h_{\alpha}^{\vee} \rangle$. By taking $\overline{\eta} = \frac{\langle \delta_{P_{\omega_{\alpha}}},h_{\alpha}^{\vee} \rangle}{\ell (n+1)}\eta$, a straightforward computation shows that

\begin{center}

$\displaystyle \frac{d\overline{\eta}}{2} = \frac{\pi}{n+1}\omega_{X_{P_{\omega_{\alpha}}}} = \widetilde{\omega}_{X_{P_{\omega_{\alpha}}}},$

\end{center}
see Equation \ref{localform} for the expression of $\omega_{X_{P_{\omega_{\alpha}}}}$. Thus, since we have 

\begin{center}
$ \widetilde{\omega}_{X_{P_{\omega_{\alpha}}}} = - {\text{Im}}(H_{X_{P_{\omega_{\alpha}}}})$ \ \ and \ \ $\widetilde{g}_{X_{P_{\omega_{\alpha}}}} = {\text{Re}}(H_{X_{P_{\omega_{\alpha}}}})$, 
\end{center}
where $H_{X_{P_{\omega_{\alpha}}}}$ is the Hermitian structure given by 

\begin{center}

$H_{X_{P_{\omega_{\alpha}}}} = \displaystyle  \frac{\langle \delta_{P_{\omega_{\alpha}}},h_{\alpha}^{\vee} \rangle}{2(n+1)} \sum_{k,j}\frac{\partial^{2}}{\partial z_{k} \partial \overline{z}_{j}}\big ( \log||s_{i} v_{\omega_{\alpha}}^{+}||^{2} \big )dz_{k} \otimes d\overline{z}_{j},$
\end{center}
from Theorem \ref{main2} we obtain a Sasaki-Einstein metric on $Q(- \ell \omega_{\alpha})$ given by

\begin{center}

$g_{Q(- \ell \omega_{\alpha})} = \widetilde{g}_{X_{P_{\omega_{\alpha}}}} + \displaystyle \frac{\langle \delta_{P_{\omega_{\alpha}}},h_{\alpha}^{\vee} \rangle^{2}}{\ell^{2}(n+1)^{2}} \eta \otimes \eta$,

\end{center}
where the basic metric $\widetilde{g}_{X_{P_{\omega_{\alpha}}}} = \frac{\langle \delta_{P_{\omega_{\alpha}}},h_{\alpha}^{\vee} \rangle}{2(n+1)}d\eta ({\rm{id}} \otimes \phi)$ is given by

\begin{center}

$\widetilde{g}_{X_{P_{\omega_{\alpha}}}}  = \displaystyle  \frac{\langle \delta_{P_{\omega_{\alpha}}},h_{\alpha}^{\vee} \rangle}{2(n+1)} \sum_{k,j}\frac{\partial^{2}}{\partial z_{k} \partial \overline{z}_{j}}\big ( \log||s_{i} v_{\omega_{\alpha}}^{+}||^{2} \big ){\text{Re}}(dz_{k} \otimes d\overline{z}_{j}).$

\end{center}
Therefore, we obtain a Sasaki-Einstein structure on $Q(- \ell \omega_{\alpha})$ given by

\begin{center}

$(g_{Q(- \ell \omega_{\alpha})}, \overline{\eta} = \frac{\langle \delta_{P_{\omega_{\alpha}}},h_{\alpha}^{\vee} \rangle}{\ell (n+1)}\eta, \xi = \frac{\ell (n+1)}{\langle \delta_{P_{\omega_{\alpha}}},h_{\alpha}^{\vee} \rangle}\frac{\partial}{\partial \theta}, \phi),$ 

\end{center}
completely determined by $\eta$, notice that $\phi \in {\text{End}}(TQ(- \ell \omega_{\alpha}))$ is determined by the horizontal lift of the connection $\sqrt{-1}\overline{\eta} \in \Omega^{1}(Q(- \ell \omega_{\alpha});\mathfrak{u}(1))$ and the complex structure of $X_{P_{\omega_{\alpha}}}$.
\end{example}

The example above cover a huge class of homogeneous contact manifolds. Let us give two explicit examples which fit in this last context. 

\begin{example}[Hopf bundle]
As in Example \ref{HOPFBUNDLE}, consider the principal circle bundle

\begin{center}

${\rm{U}}(1) \hookrightarrow S^{3} \to \mathbb{C}{\rm{P}}^{1}$.

\end{center}
In this case we have a principal ${\rm{U}}(1)$-connection on $Q(-\omega_{\alpha}) = S^{3}$ (locally) defined by

\begin{center}

$\eta'_{\alpha} = \displaystyle  - \frac{1}{2}\frac{zd\overline{z} - \overline{z}dz}{(1+|z|^{2})} + a_{U}^{-1}da_{U}.$    
    
\end{center}
From Theorem \ref{main1} we have a contact $1$-form on $S^{3}$ given by

\begin{center}

$\eta =  \displaystyle \frac{\overline{z}dz - zd\overline{z}}{2\sqrt{-1}(1 + |z|^{2})} + d\theta_{U}$,

\end{center}
notice that $I(\mathbb{C}{\rm{P}}^{1}) = 2$ and $\eta = \overline{\eta}$. It is straightforward to check that 

\begin{center}

$\displaystyle \frac{d\eta}{2} = \frac{\sqrt{-1}}{2} \partial \overline{\partial} \log (1 + |z|^{2}) = \frac{\pi}{2} \omega_{\mathbb{C}{\rm{P}}^{1}} = \widetilde{\omega}_{\mathbb{C}{\rm{P}}^{1}},$

\end{center}
see Example \ref{exampleP1} to recall the expression of $\omega_{\mathbb{C}{\rm{P}}^{1}}$, and see also Equation \ref{newkahler}. Therefore, since we have $\widetilde{\omega}_{\mathbb{C}{\rm{P}}^{1}} = - {\text{Im}}(H_{\mathbb{C}{\rm{P}}^{1}})$ and $\widetilde{g}_{\mathbb{C}{\rm{P}}^{1}} = {\text{Re}}(H_{\mathbb{C}{\rm{P}}^{1}})$, where $H_{\mathbb{C}{\rm{P}}^{1}}$ is the Hermitian structure given by 
\begin{center}

$H_{\mathbb{C}{\rm{P}}^{1}} = \displaystyle \frac{dz \otimes d\overline{z}}{2(1 + |z|^{2})^{2}},$
\end{center}
from Theorem \ref{main2} we obtain a Sasaki-Einstein metric on $S^{3}$ given by

\begin{center}

$g_{S^{3}} = \displaystyle \frac{{\text{Re}}(dz \otimes d\overline{z})}{2(1 + |z|^{2})^{2}} + \bigg ( \displaystyle \frac{\overline{z}dz - zd\overline{z}}{2\sqrt{-1}(1 + |z|^{2})} + d\theta_{U}\bigg) \otimes \bigg ( \displaystyle \frac{\overline{z}dz - zd\overline{z}}{2\sqrt{-1}(1 + |z|^{2})} + d\theta_{U}\bigg).$
\end{center}
Hence, we have a Sasaki-Einstein structure $(g_{S^{3}}, \eta, \xi = \frac{\partial}{\partial \theta}, \phi)$ on $S^{3}$ completely determined by $\eta$.
\end{example}

\begin{remark}
It is worthwhile to notice that the computations above can be naturally generalized to the case provided by the principal ${\rm{U}}(1)$-bundle
\begin{center}
${\rm{U}}(1) \hookrightarrow S^{2n+1}/\mathbb{Z}_{\ell} \to \mathbb{C}{\rm{P}}^{n}$,
\end{center}
$\forall \ell \in \mathbb{Z}_{>0}$, see Example \ref{COMPLEXHOPF} for the Lie-theoretical approach.
\end{remark}

\begin{example}[Stiefel manifold] As in Example \ref{STIEFEL}, consider the principal ${\rm{U}}(1)$-bundle

\begin{center}

${\rm{U}}(1) \hookrightarrow  \mathscr{V}_{2}(\mathbb{R}^{6}) \to {\rm{Gr}}(2,\mathbb{C}^{4})$.

\end{center}
As we have seen, in this case we have a principal ${\rm{U}}(1)$-connection on $\mathscr{V}_{2}(\mathbb{R}^{6}) = Q(-\omega_{\alpha_{2}})$ defined by

\begin{center}

$\eta'_{\alpha_{2}} = \displaystyle \frac{1}{2}\big ( \partial - \overline{\partial} \big ) \log \Big ( 1 + \sum_{k = 1}^{4}|z_{k}|^{2} + \bigg |\det \begin{pmatrix}
 z_{1} & z_{3} \\
 z_{2} & z_{4}
\end{pmatrix} \bigg |^{2} \Big) + a_{U}^{-1}da_{U},$        
\end{center}
Thus, from Theorem \ref{main1} we have a contact $1$-form on $\mathscr{V}_{2}(\mathbb{R}^{6})$ given by
\begin{center}

$\eta = -\displaystyle \frac{\sqrt{-1}}{2}\big ( \partial - \overline{\partial} \big ) \log \Big ( 1 + \sum_{k = 1}^{4}|z_{k}|^{2} + \bigg |\det \begin{pmatrix}
 z_{1} & z_{3} \\
 z_{2} & z_{4}
\end{pmatrix} \bigg |^{2} \Big) + d\theta_{U}.$        
\end{center}
Now, by taking $\overline{\eta} = \frac{4}{5} \eta$ as in Theorem \ref{main2}, it follows that

\begin{center}

$\displaystyle \frac{d \overline{\eta}}{2} = \frac{2 \sqrt{-1}}{5} \partial \overline{\partial}  \log \Big ( 1 + \sum_{k = 1}^{4}|z_{k}|^{2} + \bigg |\det \begin{pmatrix}
 z_{1} & z_{3} \\
 z_{2} & z_{4}
\end{pmatrix} \bigg |^{2} \Big)$.

\end{center}
Thus, from Equations \ref{C8S8.3Sub8.3.2Eq8.3.21}, \ref{newkahler}, and the computation above, we have that

\begin{center}

$\displaystyle \frac{d \overline{\eta}}{2} = \frac{\pi}{5} \omega_{{\rm{Gr}}(2,\mathbb{C}^{4})} = \widetilde{\omega}_{{\rm{Gr}}(2,\mathbb{C}^{4})},$

\end{center}
it is worthwhile to notice that $I({\rm{Gr}}(2,\mathbb{C}^{4})) = 4$. Since we have $\widetilde{\omega}_{{\rm{Gr}}(2,\mathbb{C}^{4})} = - {\text{Im}}(H_{{\rm{Gr}}(2,\mathbb{C}^{4})})$ and $\widetilde{g}_{{\rm{Gr}}(2,\mathbb{C}^{4})} = {\text{Re}}(H_{{\rm{Gr}}(2,\mathbb{C}^{4})})$, where $H_{{\rm{Gr}}(2,\mathbb{C}^{4})}$ is the Hermitian structure given by 
\begin{center}

$H_{{\rm{Gr}}(2,\mathbb{C}^{4})} = \displaystyle  \frac{2}{5} \sum_{i,j= 1}^{4}\frac{\partial^{2}}{\partial z_{i} \partial \overline{z}_{j}} \log \Big ( 1 + \sum_{k = 1}^{4}|z_{k}|^{2} + \bigg |\det \begin{pmatrix}
 z_{1} & z_{3} \\
 z_{2} & z_{4}
\end{pmatrix} \bigg |^{2} \Big)dz_{i} \otimes d\overline{z}_{j},$
\end{center}
from Theorem \ref{main2} we obtain a Sasaki-Einstein metric on $\mathscr{V}_{2}(\mathbb{R}^{6})$ given by

\begin{center}

$g_{\mathscr{V}_{2}(\mathbb{R}^{6})} = \widetilde{g}_{{\rm{Gr}}(2,\mathbb{C}^{4})} + \displaystyle \frac{16}{25} \eta \otimes \eta$,

\end{center}
where the basic metric $\widetilde{g}_{{\rm{Gr}}(2,\mathbb{C}^{4})} = \frac{2}{5}d\eta ({\rm{id}} \otimes \phi)$ is given by

\begin{center}

 $ \widetilde{g}_{{\rm{Gr}}(2,\mathbb{C}^{4})} = \displaystyle  \frac{2}{5}\sum_{i,j= 1}^{4}\frac{\partial^{2}}{\partial z_{i} \partial \overline{z}_{j}} \log \Big ( 1 + \sum_{k = 1}^{4}|z_{k}|^{2} + \bigg |\det \begin{pmatrix}
 z_{1} & z_{3} \\
 z_{2} & z_{4}
\end{pmatrix} \bigg |^{2} \Big){\text{Re}}(dz_{i} \otimes d\overline{z}_{j})$.

\end{center}
Therefore, we have a Sasaki-Einstein structure 
\begin{center}
$(g_{\mathscr{V}_{2}(\mathbb{R}^{6})}, \overline{\eta} = \frac{4}{5} \eta, \xi = \frac{5}{4}\frac{\partial}{\partial \theta}, \phi),$ 
\end{center}
on $\mathscr{V}_{2}(\mathbb{R}^{6})$ completely determined by the contact structure $\eta$.

\end{example}

The last results tell us that the cone $\mathscr{C}(M)$ of a compact homogeneous contact manifold $(M,\eta,G)$ is a K\"{a}hler manifold, see \ref{sasakikahler}. Moreover, a straightforward computation shows that the K\"{a}hler form $\omega_{\mathscr{C}} = g_{\mathscr{C}}(J_{\mathscr{C}} \otimes {\text{id}})$ is given by
\begin{equation}
\label{Eqconekahler}
\omega_{\mathscr{C}} = \frac{1}{2}d \big ( r^{2} \overline{\eta} \big ) = r dr \wedge \overline{\eta} + \frac{r^{2}}{2}d \overline{\eta},
\end{equation}
we shall denote $ \Phi = r^{2} \overline{\eta} \in \Omega^{1}(\mathscr{C}(M))$. It is worthwhile to observe that 
\begin{equation}
g_{\mathscr{C}} = dr \otimes dr + \frac{r^{2}}{2} d \overline{\eta} ({\rm{id}} \otimes \phi) + r^{2}\overline{\eta} \otimes \overline{\eta}.     
\end{equation}
Before we prove our next theorem, we consider the following well-known result, see for instance \cite{BOYERGALICKI}.

\begin{proposition}
\label{CYCONE}
Let $M$ be a Sasaki manifold with ${\text{K}}$-contact structure $(g_{M},\phi,\xi,\eta)$, then $g_{M}$ is Sasaki-Einstein if and only if the cone metric $g_{\mathscr{C}}$ is Ricci-flat, i.e., if and only if $(\mathscr{C}(M),g_{\mathscr{C}})$ is Calabi-Yau.
\end{proposition}

\begin{remark}
Notice that in the last proposition for the case that $g_{\mathscr{C}}$ is Calabi-Yau we have ${\rm{Hol}}^{0}(g_{\mathscr{C}}) \subset {\rm{SU}}(n+1)$, where ${\rm{Hol}}^{0}(g_{\mathscr{C}})$ denotes the restricted holonomy group.

\end{remark}

\begin{theorem}
\label{main3}
Let $(M,\eta,G)$ be a compact homogeneous contact manifold such that $M = \mathcal{Q}_{P} / \mathbb{Z}_{\ell}$, for some parabolic Lie subgroup $P \subset G^{\mathbb{C}}$. Then, the cone $\mathscr{C}(M)$ admits a Calabi-Yau metric $\omega_{\mathscr{C}} = \frac{1}{2}d\Phi$ such that 

\begin{equation}
\Phi = \displaystyle - \frac{r^{2} \sqrt{-1}}{2(n+1)}\big ( \partial - \overline{\partial} \big )\log \big | \big |s_{U}v_{\delta_{P}}^{+} \big| \big |^{2} + \frac{r^{2}I(X_{P})}{\ell(n+1)}d\theta_{U},    
\end{equation}
for some local section $s_{U} \colon U \subset X_{P} \to G^{\mathbb{C}}$, where $v_{\delta_{P}}^{+}$ denotes the highest weight vector of weight $\delta_{P}$ associated to the irreducible $\mathfrak{g}^{\mathbb{C}}$-module $V(\delta_{P})$.
\end{theorem}

\begin{proof}
The proof follows from the following facts: From Theorem \ref{main2}, we obtain a Sasaki-Einstein structure on $M = \mathcal{Q}_{P} / \mathbb{Z}_{\ell}$ defined by

\begin{center}

$(g_{M}, \phi,\xi = \frac{\ell(n+1)}{I(X_{P})}\frac{\partial}{\partial \theta}, \frac{I(X_{P})}{\ell(n+1)}\eta )$.

\end{center}
Thus, from Proposition \ref{CYCONE} we have that the cone $\mathscr{C}(M)$ is a K\"{a}hler Ricci-flat manifold, i.e., it defines a Calabi-Yau manifold. Now, we notice that from Equation \ref{Eqconekahler} it follows that 

\begin{center}

$\Phi = \displaystyle \frac{r^{2}I(X_{P})}{\ell(n+1)} \eta ,$

\end{center}
defines a Calabi-Yau metric $\omega_{\mathscr{C}} = \frac{1}{2}d\Phi$ on $M = \mathcal{Q}_{P} / \mathbb{Z}_{\ell}$.
\end{proof}

\begin{remark}
Notice that in the result above if $M = \mathcal{Q}_{P}$, then we have $\mathscr{C}(\mathcal{Q}_{P})$ simply connected, it follows that ${\rm{Hol}}(g_{\mathscr{C}}) = {\rm{Hol}}^{0}(g_{\mathscr{C}}) \subset {\rm{SU}}(n+1)$. Hence, in this case we have that $\mathscr{C}(\mathcal{Q}_{P})$ admits a spin structure, e.g. \cite[Example 1, p. 84]{BAUM}, \cite{BOYERGALICKI}.
\end{remark}

\begin{remark}
It is worthwhile to observe that we can compute the global K\"{a}hler potential for the Calabi-Yau metric $\omega_{\mathscr{C}} = \frac{1}{2}d\Phi$ provided by Theorem \ref{main3}. In fact, if we consider $L = K_{X_{P}}^{\otimes \frac{\ell}{I(X_{P})}}$, we can take a Hermitian structure $H$ on $L$ such that 
\begin{center}

$H((gP,w),(gP,v)) = w \overline{v}  \big | \big | s_{U}(gP)v_{\delta_{P}}^{+} \big | \big |^{\frac{2 \ell}{I(X_{P})}},$

\end{center}
$\forall (gP,w),(gP,v) \in L|_{U}$ and $s_{U} \colon U \subset X_{P} \to G^{\mathbb{C}}$ (local section). Now, we consider the smooth function defined by $\frac{1}{2}r^{2} \colon L^{\times} \to \mathbb{R}^{+}$, such that $\frac{1}{2}r^{2} = H$. It is straightforward to verify that (locally)

\begin{center}

$\displaystyle \frac{1}{2}r^{2} = b_{U}\overline{b}_{U} \big | \big | s_{U}(z_{U})v_{\delta_{P}}^{+} \big | \big |^{\frac{2 \ell}{I(X_{P})}} = b_{U}\overline{b}_{U}{\mathrm{e}}^{\varphi_{U}(z_{U})},$

\end{center}
where $(z_{U},b_{U}) \in L|_{U}$ are local coordinates, and

\begin{center}

$\varphi_{U} = \displaystyle \frac{\ell}{I(X_{P})}\log \Big (  \big | \big | s_{U}v_{\delta_{P}}^{+} \big | \big |^{2}\Big )$.

\end{center}
From the facts above we notice that $b_{U} = \frac{r}{\sqrt{2}}{\mathrm{e}}^{-\frac{1}{2}\varphi_{U}(z_{U}) + \sqrt{-1}\theta_{U}}$ and we can verify that 
\begin{equation}
\displaystyle \frac{\sqrt{-1}}{2} \partial \overline{\partial} r^{2} = \frac{\ell (n+1)}{I(X_{P})} \omega_{\mathscr{C}}.
\end{equation}
Thus, we obtain a globally defined K\"{a}hler potential $\frac{I(X_{P})r^{2}}{2\ell(n+1)}$ for $\omega_{\mathscr{C}}$. Here we have used the identification $L^{\times} = \mathscr{C}(M)$, for $M = \mathcal{Q}_{P} / \mathbb{Z}_{\ell}$.
\end{remark}

\begin{remark}
Notice that in the general case when $M = Q(L)$, for some ample line bundle $L^{-1} \in \text{Pic}(X_{P})$, we can use the Sasaki structure obtained in \ref{sasakiample} to get an explicit K\"{a}hler structure on $\mathscr{C}(M)$ described in terms of Lie theory.
\end{remark}

In the next section we shall explore the application of the last theorem in the study of crepant resolutions of Calabi-Yau cones over homogeneous Sasaki-Einstein manifolds. Thus, explicit examples which illustrate Theorem \ref{main3} will be given in the next section.

\section{Applications in Crepant resolutions of Calabi-Yau cones}
\label{sec5}

This section is devoted to provide a concrete application of the results developed in the previous sections in the study of crepant resolutions of Calabi-Yau cones.

In Subsection \ref{subsec5.1}, we shall discuss how the Cartan-Remmert reduction of canonical bundles of K\"{a}hler-Einstein Fano manifolds can be used to produce concrete examples for the conjecture introduced in \cite{ADSCFT}. The main goal is to prove Theorem \ref{Theo4}.

In Subsection \ref{subsec5.2}, we provide several concrete examples of resolutions of Calabi-Yau cones, many of these concrete examples are new in the literature. Our references for this section are  \cite{COLON}, \cite{RESOLUTIONCOMPSUP}, \cite{GOTO}, \cite{CALABIANSATZ}, \cite{GRAUERT}, \cite{LAUFER}.

\subsection{Crepant resolution of Calabi-Yau cones and Calabi Ansatz}
\label{subsec5.1}
As we have seen so far, for every compact homogeneous contact manifold $(M,\eta,G)$, such that ($\mathscr{C}(M),\omega_{\mathscr{C}})$ is K\"{a}hler Ricci-flat, we have that $M = \mathcal{Q}_{P}/\mathbb{Z}_{\ell}$, for some $\ell \in \mathbb{Z}_{>0}$. Therefore, we obtain

\begin{center}
    
$M \in \mathscr{P}(X_{P},{\rm{U}}(1))  \longleftrightarrow L(M) \in {\text{Pic}}(X_{P}),$

\end{center}
such that 
\begin{center}

$c_{1}(L(M)) = - \displaystyle \frac{\ell}{I(X_{P})}c_{1}(X_{P}),$
    
\end{center}
notice that $L(M) = K_{X_{P}}^{\otimes \frac{\ell}{I(X_{P})}}$. From this, we can take a Hermitian structure $H$ on $L(M)$ and define $\rho \colon L(M) \to \mathbb{R}_{\geq 0}$ such that
\begin{center}
$\rho(u) = \sqrt{H(u,u)},$ 
\end{center}
$\forall u \in L(M)$. Since $L(M)$ is a negative line bundle, the function $\rho$ is strictly plurisubharmonic away from the zero section $X_{P} \subset L(M)$ (cf. \cite[p. 341]{NEGATIVELINEBUNDLE}), thus $L(M)$ can be exhausted by strictly pseudo-convex domains

\begin{center}
    
$D_{\epsilon} = \Big \{ u \in L(M) \ \ \Big | \ \ \rho(u) < \epsilon \Big\}$,
    
\end{center}
$\forall \epsilon > 0$. It follows that $L(M)$ is holomorphically convex, and we have the Cartan-Remmert reduction \cite{GRAUERT}. Namely,  we have a Stein space $Y$ and a holomorphic map $\mathscr{R} \colon L(M) \to Y$ which contracts (blows down) the maximal compact analytic subset $X_{P} \subset L(M)$ to a point and defines a biholomorphism outside $X_{P}$. 

If we denote

\begin{center}
    $L(M)^{\times} = L(M) \backslash X_{P},$
\end{center}
it follows from the identification $L(M)^{\times} \cong \mathscr{C}(M)$ that

\begin{center}

$Y = \mathscr{C}(M) \cup \{o\},$    
    
\end{center}
where $\{o\} = \mathscr{R}(X_{P}) \subset Y$. Here the manifold $M$ can be identified with the level set $\{r = 1\} \subset Y$ (link).

\begin{remark}
The space $Y = \mathscr{C}(M) \cup \{o\}$ is a normal complex space, and by the Riemann extension theorem, we have $\iota_{\ast} \mathscr{O}_{\mathscr{C}(M)} = \mathscr{O}_{Y}$, where the map $\iota \colon \mathscr{C}(M) \hookrightarrow Y$ denotes the natural inclusion. 
\end{remark}

\begin{figure}[H]
\centering
\includegraphics[scale=0.4]{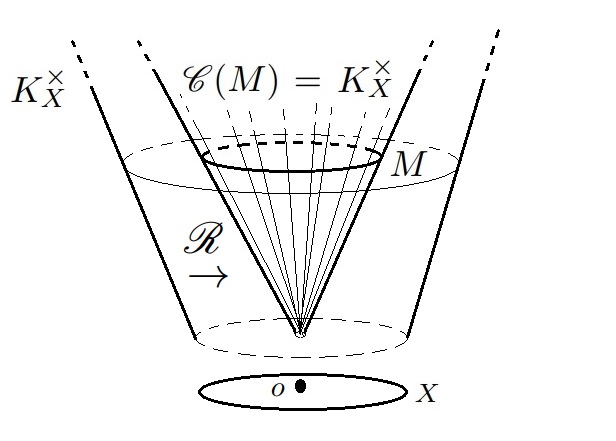}
\caption{Cartan-Remmert reduction of the canonical bundle associated to a K\"{a}hler-Einstein Fano manifold.}
\end{figure}

Since the Calabi-Yau metric provided by Theorem \ref{main3} is singular at the apex $o \in Y$ (conical singularity), this leads us to the following general conjecture.

\begin{conjecture}[\cite{ADSCFT}]
\label{conj}
Let $f \colon \widetilde{Y} \to Y$ be a crepant resolution of an isolated singularity $Y = \mathscr{C}(M) \cup \{o\}$, where $\mathscr{C}(M)$ is the cone over a Sasaki-Einstein manifold $(M,g_{M})$. Then, $\widetilde{Y}$ admits a unique Ricci-flat K\"{a}hler metric in each K\"{a}hler class in $H^{2}(\widetilde{Y},\mathbb{R})$ which is asymptotic to the cone metric on the cone over the Sasaki-Einstein manifold $(M,g_{M})$. 
\end{conjecture}

We recall that a resolution of singularities $f \colon \widetilde{Y} \to Y$ is called crepant if $f^{\ast}K_{Y} = K_{\widetilde{Y}}$, e.g. \cite{RESOLUTIONCOMPSUP}. We have the following important partial solution for \ref{conj}.

\begin{theorem}[\cite{RESOLUTIONCOMPSUP}]
\label{resolutionvan}
Let $f \colon \widetilde{Y} \to Y$ be a crepant resolution of an isolated singularity $Y = \mathscr{C}(M) \cup \{o\}$, where $\mathscr{C}(M)$ admits a Ricci-flat cone metric $\omega_{\mathscr{C}}$. Then, for each cohomology class $b \in H_{c}^{2}(\widetilde{Y},\mathbb{R})$ (cohomology with compact supports) there is a complete Ricci-flat K\"{a}hler metric $\omega_{CY}$ with $[\omega_{CY}] = b$ and $\omega_{CY}$ is asymptotic to the cone metric $\omega_{\mathscr{C}}$ for any order of derivatives.
\end{theorem}

\begin{remark}
\label{conevariety}
It is worth pointing out that in the context of Cartan-Remmert reduction $\mathscr{R} \colon L(M) \to Y = \mathscr{C}(M) \cup \{o\}$, we can show that $Y = \mathscr{C}(M) \cup \{o\}$ is in fact an affine variety, see for instance \cite{VERBITSKY} and \cite{AFFINECONE}.

\end{remark}

The result above together with Theorem \ref{main3} provide the following description for Asymptotically conical (AC) Calabi-Yau manifolds obtained from crepant resolutions of Riemannian cones over homogeneous Sasaki-Einstein manifolds.

\begin{proposition}
\label{resolutionhomogenous}
Let $f \colon \widetilde{Y} \to Y$ be a crepant resolution of $Y = \mathscr{C}(\mathcal{Q}_{P} / \mathbb{Z}_{\ell}) \cup \{o\}$, for some parabolic Lie subgroup $P \subset G^{\mathbb{C}}$.  Then, for each cohomology class $b \in H_{c}^{2}(\widetilde{Y},\mathbb{R})$ there is a complete Ricci-flat K\"{a}hler metric $\omega_{CY}$ with $[\omega_{CY}] = b$ and $\omega_{CY}$ is asymptotic to the cone metric $\omega_{\mathscr{C}} = \frac{1}{2}d\Phi$ (for any order of derivatives) such that 
\begin{center}
$\Phi = \displaystyle - \frac{r^{2} \sqrt{-1}}{2(n+1)}\big ( \partial - \overline{\partial} \big )\log \big | \big |s_{U}v_{\delta_{P}}^{+} \big| \big |^{2} + \frac{r^{2}I(X_{P})}{\ell(n+1)}d\theta_{U},$
\end{center}
for some local section $s_{U} \colon U \subset X_{P} \to G^{\mathbb{C}}$, where $v_{\delta_{P}}^{+}$ denotes the highest weight vector of weight $\delta_{P}$ associated to the irreducible $\mathfrak{g}^{\mathbb{C}}$-module $V(\delta_{P})$.
\end{proposition}

Notice that Theorem \ref{resolutionvan} ensures that there is a Ricci-flat conical K\"{a}hler metric in the K\"{a}hler class which belongs to the compactly supported cohomology group $H_{c}^{2}(\widetilde{Y},\mathbb{R})$. In \cite{GOTO} the following result was shown.

\begin{theorem}[\cite{GOTO}]
\label{Gototheorem}
Let $Y$ be an affine variety with only normal isolated singularity at $p \in Y$. We assume that the complement $Y \backslash \{p\}$ is biholomorphic to the cone $\mathscr{C}(M)$ of a Sasaki-Einstein manifold $M$ of real dimension $2n+1$. If there is a resolution of singularity $f \colon \widetilde{Y} \to Y$ with trivial canonical bundle $K_{\widetilde{Y}}$, then there is a Ricci-flat complete K\"{a}hler metric for every K\"{a}hler class of $\widetilde{Y}$.

\end{theorem}

The result above shows that the restricted condition of compactly supported cohomology in \ref{resolutionvan} is not necessary and it implies that the conjecture on the existence of complete Ricci-flat K\"{a}hler metrics in \cite{ADSCFT} is affirmative.

As we have seen so far, every Calabi-Yau cone $ (\mathscr{C}(\mathcal{Q}_{P} / \mathbb{Z}_{\ell}), \omega_{\mathscr{C}} = \frac{1}{2}d\Phi)$ is obtained from the total space of a negative line bundle $L(\mathcal{Q}_{P} / \mathbb{Z}_{\ell}) \in {\text{Pic}}(X_{P})$, namely, $\mathscr{C}(\mathcal{Q}_{P} / \mathbb{Z}_{\ell}) \cong L( \mathcal{Q}_{P} / \mathbb{Z}_{\ell})^{\times}$. In this case, if we consider the Cartan-Remmert reduction
\begin{equation}
\label{REMMERT}
\mathscr{R} \colon L(\mathcal{Q}_{P} / \mathbb{Z}_{\ell}) \to Y = \mathscr{C}(\mathcal{Q}_{P} / \mathbb{Z}_{\ell}) \cup \{o\},
\end{equation}
we obtain a resolution of the singularity $\{o\}$ which is a crepant resolution if and only if $\ell = I(X_{P})$, see for instance \cite{COLON}. Let us formalize this last statement.

\begin{proposition}
\label{remmert}
Let $X$ be a K\"{a}hler-Einstein Fano manifold and $L \in {\text{Pic}}(X)$ such that $L = K_{X}^{\otimes \frac{\ell}{I(X)}}$, for some $\ell \in \mathbb{Z}_{>0}$. Then the manifold defined by the total space of $L$ admits a global holomorphic volume form if and only if $\ell \ | \ I(X)$.
\end{proposition}
\begin{proof}
Consider the embedding defined by the zero section $\sigma_{0} \colon X \hookrightarrow L$. From the adjunction formula we have 

\begin{center}

$K_{X} = \big (K_{L} \otimes \mathscr{O}{([ X ])} \big )|_{X},$    
    
\end{center}
such that $[X] \in {\text{Div}}(L)$ and $\mathscr{O} \colon {\text{Div}}(L) \to {\text{Pic}}(L)$. Since $[X]$ is a smooth hypersurface, it follows that it defines an effective divisor, in fact, $X \subset L$ is an irreducible divisor. Therefore, $\exists \ 0 \neq  \sigma \in H^{0}(L,\mathscr{O}{([ X ])})$, and from Poincar\'{e}-Lelong formula \cite[p. 281]{VOISIN} we have $c_{1}(\mathscr{O}{([ X])}) = {\text{PD}}[X] \in H_{c}^{2}(L,\mathbb{Z})$. Hence, we get

\begin{center}

$c_{1}(K_{L})|_{X} = \bigg (\displaystyle \frac{I(X)}{\ell} - 1 \bigg ) {\text{PD}}[X],$

\end{center}
notice that $\sigma_{0}^{\ast}{\text{PD}}[X] = c_{1}(L)$. From this, we obtain 

$$ \exists \ 0 \neq \Omega \in H^{0}(L,K_{L}) \Longrightarrow \ell \ | \ I(X).$$
Conversely, we can pullback by the Cartan-Remmert reduction $\mathscr{R} \colon L \to \mathscr{C}(M) \cup \{o\}$  a nowhere vanishing holomorphic volume form $\widetilde{\Omega}$ defined on the cone $\mathscr{C}(M) = L^{\times}$, notice that  $\mathscr{C}(M)$ is a Calabi-Yau cone over the Sasaki-Einstein manifold defined by the principal circle bundle $M = Q(L)$. Since $o$ is a rational singularity of $\mathscr{C}(M) \cup \{o\}$, by extending $\mathscr{R}^{\ast}\widetilde{\Omega}$, see for instance \cite{LAUFER} and \cite{RESOLUTIONCOMPSUP}, if $\ell \ | \ I(X)$ we obtain $0 \neq \Omega \in H^{0}(L,K_{L})$ which vanishes to order $\frac{I(X)}{\ell} - 1$ along the
zero section $X \hookrightarrow L$. 
\end{proof}

From the result above we see that the Cartan-Remmert reduction \ref{REMMERT} defines a crepant resolution if and only if $\ell = I(X_{P})$. Hence, if $\ell = I(X_{P})$ we have a crepant resolution defined by
\begin{equation}
\mathscr{R} \colon K_{X_{P}} \to Y = \mathscr{C}(\mathcal{Q}_{P} / \mathbb{Z}_{I(X_{P})}) \cup \{o\}.
\end{equation}
For canonical bundles of K\"{a}hler-Einstein Fano manifolds Calabi \cite{CALABIANSATZ} constructed many examples of AC Ricci-flat K\"{a}hler manifolds which are in fact almost explicit.

\begin{theorem}[E. Calabi]
\label{CALABI3}
 Let $(X,\omega_{X})$ be a compact K\"{a}hler-Einstein manifold such that $c_{1}(X) > 0$, i.e. a K\"{a}hler-Einstein Fano manifold. Then there exists a complete Ricci-flat metric on the manifold defined by the total space $K_{X} = \det(T^{\ast}X)$. 
\end{theorem}

\begin{remark}
The proof of the result above which we are following can be found in \cite{SALAMONANSATZ}, see also \cite[Appendix D.2]{EDERTHESIS}. It is worth pointing out that in the context of the theorem above we have a nowhere vanishing parallel holomorphic $(n+1,0)$ form $\Omega = d\tau$, such that
\begin{center}

$\tau_{\nu}(X_{1},\ldots,X_{n}) = \nu(\pi_{\ast}X_{1},\ldots,\pi_{\ast}X_{n})$, \ \ (Tautological form)

\end{center}
$\forall \nu \in K_{X}$ and $X_{1},\ldots,X_{n} \in T_{\nu}K_{X}$, where $\pi \colon K_{X} \to X$ is the projection map.
\end{remark}

For the particular case of K\"{a}hler-Einstein Fano manifolds defined by complex flag manifolds $X_{P} = G^{\mathbb{C}}/P$, the Ricci-flat K\"{a}hler metric provided by Theorem \ref{CALABI3} has the following characterization. 

\begin{theorem}[\cite{EDER}]
\label{C8S8.3Sub8.3.2Teo8.3.5}
Let $(X_{P},\omega_{X_{P}})$ be a complex flag manifold associated to some parabolic Lie subgroup $P = P_{\Theta} \subset G^{\mathbb{C}}$, such that $\dim_{\mathbb{C}}(X_{P}) = n$. Then, the total space $K_{X_{P}}$ admits a complete Ricci-flat K\"{a}hler metric with K\"{a}hler form given by

\begin{equation}
\label{C8S8.3Sub8.3.2Eq8.3.15}
\omega_{CY} = (2\pi r^{2} + C)^{\frac{1}{n+1}} \Bigg ( \pi^{\ast}\omega_{X_{P}} - \frac{\sqrt{-1}}{n+1} \frac{\nabla b \wedge \overline{\nabla b}}{(2\pi r^{2} + C)} \Bigg ),
\end{equation}\\
where $C>0$ is some positive constant and $r^{2} \colon K_{X_{P}} \to \mathbb{R}_{\geq 0}$ is given by $r^{2}([g,b]) = |b|^{2}$, $\forall [g,b] \in K_{X_{P}}$. Furthermore, the K\"{a}hler form above is completely determined by the quasi-potential $\varphi \colon G^{\mathbb{C}} \to \mathbb{R}$ defined by
$$\varphi(g) = \displaystyle  \frac{1}{2\pi} \log \Big (\prod_{\alpha \in \Sigma \backslash \Theta} \big | \big |gv_{\omega_{\alpha}}^{+} \big | \big |^{2\langle \delta_{P},h_{\alpha}^{\vee} \rangle} \Big),$$
for every $g \in G^{\mathbb{C}}$. Therefore, $(K_{X_{P}},\omega_{CY})$ is a (complete) noncompact Calabi-Yau manifold with Calabi-Yau metric $\omega_{CY}$ completely determined by $\Theta \subset \Sigma$.
\end{theorem}

 \begin{remark}
It is worthwhile to point out that in the result above the complete Ricci-flat K\"{a}hler metric induced by the K\"{a}hler form $\omega_{CY}$ is given by
\begin{center}
\label{C8S8.3Sub8.3.1Eq8.3.14}
$g_{CY} =  \displaystyle (2\pi r^{2} + C)^{\frac{1}{n+1}}  \Bigg ( \pi^{\ast}g_{X_{P}} + \frac{1}{n+1}\frac{{\text{Re}} \big (\nabla b \otimes \overline{\nabla b} \big )}{(2\pi r^{2} + C)} \Bigg )$,
\end{center}
see \cite{SALAMONANSATZ} for more details.
\end{remark}

\begin{remark}
\label{supportcalabi}
Notice that in the context of the last theorem we have $[\omega_{CY}] \in H_{c}^{2}(K_{X_{P}},\mathbb{R})$. In fact, since 
\begin{center}
$H_{c}^{3}(K_{X_{P}}) \cong H^{3}(K_{X_{P}},K_{X_{P}}^{\times}) \cong H^{1}(X_{P}) = \{0\}$,
\end{center}
we obtain the split exact sequence

\begin{center}
\begin{tikzcd}

0 \arrow[r] & H^{2}(K_{X_{P}},K_{X_{P}}^{\times}) \arrow[r, "\beta^{\ast}"] & H^{2}(K_{X_{P}}) \arrow[r, "\iota^{\ast}"] & H^{2}(K_{X_{P}}^{\times}) \arrow[r] & 0, 

\end{tikzcd}
\end{center}
where by considering $\Omega^{2}(K_{X_{P}},K_{X_{P}}^{\times}) = \Omega^{2}(K_{X_{P}}) \oplus \Omega^{1}(K_{X_{P}}^{\times})$ we have
\begin{center}
$\beta \colon \Omega^{2}(K_{X_{P}},K_{X_{P}}^{\times}) \to \Omega^{2}(K_{X_{P}})$, such that $\beta(a,b) = a$, 
\end{center}
and $\iota \colon K_{X_{P}}^{\times} \hookrightarrow K_{X_{P}}$. From this, since the map $\sigma_{0} \circ \pi \colon K_{X_{P}}^{\times} \to K_{X_{P}}$ is homotopic to the inclusion $\iota$, where $\sigma_{0} \colon X_{P} \hookrightarrow K_{X_{P}}$ denotes the zero section, it follows that 
\begin{center}
 $ \iota^{\ast}[\omega_{CY}] = \pi^{\ast}[\sigma_{0}^{\ast}\omega_{CY}] = C^{\frac{1}{n+1}}[\pi^{\ast}\omega_{X_{P}}] = -  C^{\frac{1}{n+1}}\pi^{\ast}c_{1}(K_{X_{P}})$.   
\end{center}
Therefore, once we have $c_{1}(K_{X_{P}}) = \sigma_{0}^{\ast} \mathcal{T}(1)$, where $\mathcal{T}(1) \in H_{c}^{2}(K_{X_{P}})$ denotes the Thom class of $K_{X_{P}}$, we obtain

\begin{center}
 
$\iota^{\ast}[\omega_{CY}] = - C^{\frac{1}{n+1}} (\sigma_{0} \circ \pi)^{\ast}\mathcal{T}(1) = - C^{\frac{1}{n+1}}  \iota^{\ast} \mathcal{T}(1).$
    
\end{center}
Thus, we have 

\begin{center}
$\iota^{\ast}[\omega_{CY}] =  - C^{\frac{1}{n+1}}  \iota^{\ast} \mathcal{T}(1) = - C^{\frac{1}{n+1}}  (\iota^{\ast} \circ \beta^{\ast})\big[\mathcal{T}(1),-\frac{\eta}{2\pi} \big] = 0,$    
\end{center}
such that $\frac{d\eta}{2\pi} = \pi^{\ast}\omega_{X_{P}}$, notice that $H^{2}(K_{X_{P}},K_{X_{P}}^{\times}) = \mathbb{R}[\mathcal{T}(1),-\frac{\eta}{2\pi}]$. Hence, since $\im(\beta^{\ast}) =\ker{(\iota^{\ast})}$ and $\im(\beta^{\ast}) \cong H_{c}^{2}(K_{X_{P}}) = \mathbb{R}\mathcal{T}(1)$, it follows that $[\omega_{CY}] \in H_{c}^{2}(K_{X_{P}})$. For more details about the ideas used above see for instance \cite{BOTT}. 
\end{remark}

In the last theorem we have $\omega_{X_{P}}$ as in \ref{localform} and the quasi-pontential $\varphi \colon G^{\mathbb{C}} \to \mathbb{R}$ as in \ref{quasipotential}. The $(1,1)$-form $\nabla b \wedge \overline{\nabla b}$ is obtained by patching together $\nabla b_{U} \wedge \overline{\nabla b_{U}}$ such that

\begin{center}
  
$\nabla b_{U} = db_{U} + b_{U} \pi^{\ast} A_{U},$   
    
\end{center}
where $(z_{U},b_{U}) \in K_{X_{P}} |_{U} \cong U \times \mathbb{C}$ are local coordinates and
\begin{equation}
\label{GAU}
A_{U} =  \displaystyle \partial \log \Bigg (\prod_{\alpha \in \Sigma \backslash \Theta} \big | \big |s_{U}v_{\omega_{\alpha}}^{+} \big | \big |^{2\langle \delta_{P},h_{\alpha}^{\vee} \rangle} \Bigg ),    
\end{equation}
for some local section $s_{U} \colon U \subset X_{P} \to G^{\mathbb{C}}$. As we have seen in the proof of Theorem \ref{main1}, we can write the gauge potential $A_{U}$ as 
\begin{equation}
\label{gaugepot}
A_{U} =  \displaystyle \partial \log  ||s_{U}v_{\delta_{P}}^{+} ||^{2},
\end{equation}
where $v_{\delta_{P}}^{+}$ denotes the highest weight vector of $V(\delta_{P})$. Therefore, the metric \ref{C8S8.3Sub8.3.2Eq8.3.15} can be (locally) described by

\begin{itemize}

    \item $\omega_{X_{P}}|_{U} = \displaystyle \frac{\overline{\partial} \partial  \log  ||s_{U}v_{\delta_{P}}^{+} ||^{2}}{2\pi \sqrt{-1}} $; \ \ (Horizontal component) \\

    \item $\nabla b_{U} = db_{U} + b_{U} \displaystyle \partial\log ||s_{U}v_{\delta_{P}}^{+} ||^{2}$. \ \ (Vertical component)
    
\end{itemize}
The key point which allows us to describe the cone metric on $\mathscr{C}(\mathcal{Q}_{P} / \mathbb{Z}_{I(X_{P})})$ and its resolution by means of the Calabi Ansatz metric on $K_{X_{P}}$ is the complete description of the Chern connection 
\begin{equation}
\nabla = d + \partial \log ||s_{U}v_{\delta_{P}}^{+} ||^{2},
\end{equation}
and the principal Cartan-Ehresmann connection (Yang-Mills field)

\begin{equation}
\sqrt{-1}\eta = \frac{1}{2} \big (\partial - \overline{\partial} \big ) \log ||s_{U}v_{\delta_{P}}^{+}||^{2} + \sqrt{-1}d\theta_{U}.
\end{equation}
As we can see, the connections above are both defined through of the gauge potential \ref{gaugepot}.

As observed in \cite{GOTO}, in the context of K\"{a}hler-Einstein Fano manifolds, the Ricci-flat K\"{a}hler metric on $K_{X}$ obtained from the Calabi Ansatz technique \ref{CALABI3} provides a resolution for the singular cone metric defined on the Calabi-Yau cone $K_{X}^{\times} \cong {\mathscr{C}}(Q(K_{X}))$ via Cartan-Remmert reduction. Therefore, from Proposition \ref{resolutionhomogenous} and Theorem \ref{C8S8.3Sub8.3.2Teo8.3.5} we obtain the following general result.

\begin{theorem}
\label{calabiresolution}
Let $(M,\eta,G)$ be a compact homogeneous contact manifold such that $M = \mathcal{Q}_{P} / \mathbb{Z}_{I(X_{P})}$, i.e., $M = Q(K_{X_{P}})$ for some parabolic Lie subgroup $P \subset G^{\mathbb{C}}$. Then, the Cartan-Remmert reduction $\mathscr{R} \colon K_{X_{P}} \to  Y = {\mathscr{C}}(M) \cup \{o\}$ provides a crepant resolution of the Calabi-Yau cone $({\mathscr{C}}(M), \omega_{\mathscr{C}})$ such that the complete Calabi-Yau metric $\omega_{CY}$ on $K_{X_{P}}$, defined by the Calabi Ansatz
\begin{equation}   
\omega_{CY} = \displaystyle (2\pi r^{2} + C)^{\frac{1}{n+1}} \Bigg (\omega_{X_{P}} - \frac{\sqrt{-1}}{n+1} \frac{( db_{U} + b_{U} A_{U})\wedge ( d\overline{b}_{U} + \overline{b}_{U}  \overline{A}_{U})}{(2\pi r^{2} + C)} \Bigg ),
\end{equation}
provides a resolution for the singular cone metric defined on $Y = {\mathscr{C}}(M) \cup \{o\}$ by
\begin{equation} 
\omega_{\mathscr{C}} = \displaystyle r dr \wedge \Bigg (\frac{\sqrt{-1}(\overline{A}_{U} - A_{U})}{2(n+1)} + \frac{d\theta_{U}}{n+1} \Bigg )  + \frac{\pi r^{2}}{n+1} \omega_{X_{P}},        
\end{equation}
such that $\omega_{X_{P}} = -\frac{\sqrt{-1}}{2\pi}dA_{U}$ and 
\begin{equation}
A_{U} =  \displaystyle \partial \log \big | \big |s_{U}v_{\delta_{P}}^{+} \big| \big |^{2},       
\end{equation}
for some local section $s_{U} \colon U \subset X_{P} \to G^{\mathbb{C}}$, where $v_{\delta_{P}}^{+}$ denotes the highest weight vector of weight $\delta_{P}$ associated to the irreducible $\mathfrak{g}^{\mathbb{C}}$-module $V(\delta_{P})$. Furthermore, there is a Ricci-flat complete K\"{a}hler metric for every K\"{a}hler class of $K_{X_{P}}$.

\end{theorem}

\begin{proof}
This result follows from the following facts: 

\begin{itemize}

\item By applying Theorem \ref{C8S8.3Sub8.3.2Teo8.3.5} on 

\begin{center}
$K_{X_{P}} \to (X_{P},\omega_{X_{P}})$,
\end{center}
we get the Calabi-Yau metric $[\omega_{CY}] \in H_{c}^{2}(K_{X_{P}},\mathbb{R})$ provided by the Calabi Ansatz technique, see Remark \ref{supportcalabi};

\item From the Boothby-Wang fibration

\begin{center}
$(\mathcal{Q}_{P} / \mathbb{Z}_{I(X_{P})}, \frac{1}{n+1}\eta) \to (X_{P}, \frac{\pi}{n+1}\omega_{X_{P}})$, 
\end{center}
we obtain a Calabi-Yau metric $\omega_{\mathscr{C}}$ on the cone $\mathscr{C}(\mathcal{Q}_{P} / \mathbb{Z}_{I(X_{P})})$ just like in Theorem \ref{main3}.
\end{itemize}
Since the Cartan-Remmert reduction 
\begin{center}
$\mathscr{R} \colon K_{X_{P}} \to Y = \mathscr{C}(\mathcal{Q}_{P} / \mathbb{Z}_{I(X_{P})}) \cup \{o\},$ 
\end{center}
defines a crepant resolution for the cone $Y = \mathscr{C}(\mathcal{Q}_{P} / \mathbb{Z}_{I(X_{P})}) \cup \{o\}$, from Theorem \ref{resolutionvan} it follows that the metric $\omega_{CY}$ obtained from the Calabi Ansatz provides a resolution for the cone metric $\omega_{\mathscr{C}}$. Moreover, since $Y =  \mathscr{C}(\mathcal{Q}_{P} / \mathbb{Z}_{I(X_{P})}) \cup \{o\}$ is an affine variety, see Remark \ref{conevariety}, from Theorem \ref{Gototheorem} we have that there is a Ricci-flat complete K\"{a}hler metric in every K\"{a}hler class of $K_{X_{P}}$. 
\end{proof}

As we can see in Theorem \ref{calabiresolution}, the gauge potential \ref{gaugepot} plays an important role in our approach. Moreover, the result above allows us to describe explicitly a huge class of examples which illustrate the existence part of Conjecture \ref{conj}. As we will see in the next subsection, the last result provides a constructive method to describe the resolution of Calabi-Yau metrics defined on certain Calabi-Yau cones over homogeneous Sasaki-Einstein manifolds.

\subsection{Examples of resolved Calabi-Yau cones via Lie theory}
\label{subsec5.2}
This subsection is devoted to describe how the result of Theorem \ref{calabiresolution} can be applied in concrete cases. The first example which we describe below covers a huge class of important manifolds obtained from maximal parabolic Lie subgroups (e.g. minuscule flag manifolds). 

\begin{example}[Basic model]
\label{maxparabolic}
As in Subsection \ref{basiccase}, consider $X_{P_{\omega_{\alpha}}} = G^{\mathbb{C}}/P_{\omega_{\alpha}}$, where $P_{\omega_{\alpha}} \subset G^{\mathbb{C}}$ is a maximal parabolic Lie subgroup. As we have seen, in this case we have

\begin{center}
    
$\mathscr{P}(X_{P_{\omega_{\alpha}}}, {\rm{U}}(1)) = \mathbb{Z}\mathrm{e}(\mathcal{Q}_{P_{\omega_{\alpha}}}) \ \ {\text{and}} \ \  K_{X_{P_{\omega_{\alpha}}}}^{-1} = L_{\chi_{\omega_{\alpha}}}^{\otimes\langle \delta_{P_{\omega_{\alpha}}},h_{\alpha}^{\vee} \rangle},$
    
\end{center}
thus we obtain $I(X_{P_{\omega_{\alpha}}}) = \langle \delta_{P_{\omega_{\alpha}}},h_{\alpha}^{\vee} \rangle$. From this, we have 

\begin{center}
    
$Q(K_{X_{P_{\omega_{\alpha}}}}) = \mathcal{Q}_{P_{\omega_{\alpha}}} / \mathbb{Z}_{\langle \delta_{P_{\omega_{\alpha}}},h_{\alpha}^{\vee} \rangle}.$
    
\end{center}
It follows from Theorem \ref{main3} that the Calabi-Yau cone metric $\omega_{\mathscr{C}} = \frac{1}{2}d\Phi$ on $\mathscr{C}( \mathcal{Q}_{P_{\omega_{\alpha}}} / \mathbb{Z}_{\langle \delta_{P_{\omega_{\alpha}}},h_{\alpha}^{\vee} \rangle})$ is determined by

\begin{center}

$\Phi = \displaystyle \frac{\langle \delta_{P_{\omega_{\alpha}}},h_{\alpha}^{\vee} \rangle r^{2}}{2 (n+1)\sqrt{-1}}\big ( \partial - \overline{\partial} \big )\log \big | \big |s_{U}v_{\omega_{\alpha}}^{+} \big| \big |^{2} + \frac{r^{2}}{n+1}d\theta_{U},$

\end{center}
notice that in this case we have $v_{\delta_{P_{\omega_{\alpha}}}}^{+} = (v_{\omega_{\alpha}}^{+})^{\otimes \langle \delta_{P_{\omega_{\alpha}}},h_{\alpha}^{\vee} \rangle}$. Hence, we obtain

\begin{center}
    
$\displaystyle \omega_{\mathscr{C}} = r dr \wedge \Bigg (\frac{\langle \delta_{P_{\omega_{\alpha}}},h_{\alpha}^{\vee} \rangle \big ( \partial - \overline{\partial} \big )\log  ||s_{U}v_{\omega_{\alpha}}^{+} ||^{2}}{2(n+1)\sqrt{-1}} + \frac{d\theta_{U}}{n+1} \Bigg )  + \frac{ r^{2} \langle \delta_{P_{\omega_{\alpha}}},h_{\alpha}^{\vee} \rangle \overline{\partial} \partial\log ||s_{U}v_{\omega_{\alpha}}^{+} ||^{2}}{2(n+1)\sqrt{-1}},$

\end{center}
which defines a singular metric on the cone $\mathscr{C}( \mathcal{Q}_{P_{\omega_{\alpha}}} / \mathbb{Z}_{\langle \delta_{P_{\omega_{\alpha}}},h_{\alpha}^{\vee} \rangle}) \cup \{o\}$.

Now, by taking the crepant resolution provided by the Cartan-Remmert reduction, we obtain a resolution for the singular cone metric above given by the Calabi Ansatz 
    
\begin{center}
    
$\omega_{CY} = \displaystyle (2\pi r^{2} + C)^{\frac{1}{n+1}} \Bigg (\omega_{X_{P_{\omega_{\alpha}}}} - \frac{\sqrt{-1}}{n+1} \frac{\nabla b_{U} \wedge \overline{\nabla b}_{U}}{(2\pi r^{2} + C)} \Bigg ),$
\end{center}
such that $C > 0$ is some positive constant and

\begin{itemize}

\item $\omega_{X_{P_{\omega_{\alpha}}}} = \displaystyle \frac{\langle \delta_{P_{\omega_{\alpha}}},h_{\alpha}^{\vee} \rangle  \overline{\partial} \partial \log \big | \big |s_{U}v_{\omega_{\alpha}}^{+} \big| \big |^{2}}{2\pi \sqrt{-1}}$,

\item $\nabla b_{U} = d b_{U} + \langle \delta_{P_{\omega_{\alpha}}},h_{\alpha}^{\vee} \rangle b_{U}\partial \log \big | \big |s_{U}v_{\omega_{\alpha}}^{+} \big| \big |^{2}.$
    
\end{itemize}
Thus, the Calabi-Yau manifold $(K_{X_{P_{\omega_{\alpha}}}},\omega_{CY})$ provides a resolution for the singular cone over the Sasaki-Einstein manifold $Q(K_{X_{P_{\omega_{\alpha}}}}) = \mathcal{Q}_{P_{\omega_{\alpha}}} / \mathbb{Z}_{I(X_{P_{\omega_{\alpha}}})}$.

\begin{figure}[h]
\centering
\includegraphics[scale=0.5]{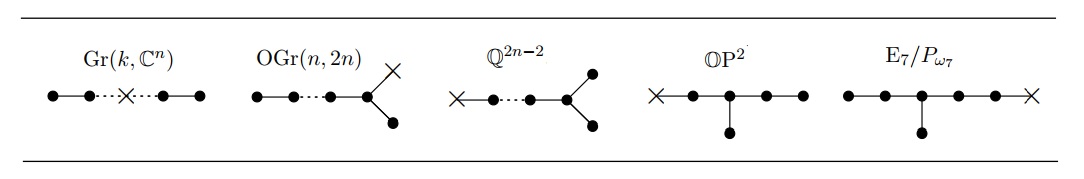}
\caption{Minuscule flag manifolds associated to maximal parabolic Lie subgroups.}
\end{figure}
\end{example}

\begin{remark}
\label{GOTOREMARK}
It is worth pointing out that for $X_{P_{\omega_{\alpha}}}$ as in the example above we have
\begin{center}
\begin{tikzcd}

0 \arrow[r] & H_{c}^{2}(K_{X_{P_{\omega_{\alpha}}}}) \arrow[r, "\beta^{\ast}"] & H^{2}(K_{X_{P_{\omega_{\alpha}}}}) \arrow[r, "\iota^{\ast}"] & H^{2}(Q(K_{X_{P_{\omega_{\alpha}}}})) \arrow[r] & 0, 

\end{tikzcd}
\end{center}
see Remark \ref{supportcalabi}. Thus, since 

\begin{center}

$1 = b_{2}(X_{P_{\omega_{\alpha}}}) = \dim H^{2}(X_{P_{\omega_{\alpha}}}) = \dim H^{2}(K_{X_{P_{\omega_{\alpha}}}}),$

\end{center}
it follows that $H^{2}(K_{X_{P_{\omega_{\alpha}}}},\mathbb{R}) = \mathbb{R}[\omega_{CY}]$, i.e. any K\"{a}hler class is cohomologous to the Calabi Ansatz K\"{a}hler form $\omega_{CY}$. We also observe that for the case which $P = P_{\Theta}$, such that $\# (\Sigma \backslash \Theta) > 1$, we have 
\begin{center}

$ \# (\Sigma \backslash \Theta) = b_{2}(X_{P}) = \dim H^{2}(X_{P}) = \dim H^{2}(K_{X_{P}}).$

\end{center}
Hence, from \cite[Example 6.3]{GOTO} there is a K\"{a}hler class which does not belong to the compactly supported cohomology group of $K_{X_{P}}$. Thus, we obtain a huge class of examples of Ricci-flat K\"{a}hler metrics defined on canonical bundles of flag manifolds associated to parabolic subgroups which satisfy $\# (\Sigma \backslash \Theta) > 1$, e.g. $K_{X_{B}}$, where $B \subset G^{\mathbb{C}}$ is a Borel subgroup. 

\end{remark}
Let us provide some particular examples of the ideas explored above in the setting of minuscule flag manifolds associated to ${\rm{SL}}(n+1,\mathbb{C})$.
\begin{example}[Calabi-Yau cone over $ \mathbb{R}\mathbb{P}^{3} = S^{3}/\mathbb{Z}_{2}$]
Consider $G^{\mathbb{C}} = {\rm{SL}}(2,\mathbb{C})$ and $X_{B} = {\rm{SL}}(2,\mathbb{C})/B$ as in Example \ref{HOPFBUNDLE}. As we have seen previously, in this case we have $X_{B} = \mathbb{C}{\rm{P}}^{1}$ and

\begin{center}
$T\mathbb{C}{\rm{P}}^{1} = K_{\mathbb{C}{\rm{P}}^{1}}^{-1} = L_{\chi_{\omega_{\alpha}}}^{ \otimes \langle \delta_{B},h_{\alpha}^{\vee} \rangle},$    
\end{center}
such that $I(\mathbb{C}{\rm{P}}^{1}) = \langle \delta_{B},h_{\alpha}^{\vee} \rangle = 2$. Now, since $\mathscr{P}(\mathbb{C}{\rm{P}}^{1},{\rm{U}}(1)) = \mathbb{Z}\mathrm{e}(Q(\omega_{\alpha}))$, where $Q(-\omega_{\alpha}) = \mathcal{Q}_{B}$ and $\mathcal{Q}_{B} = S^{3}$, it follows that

\begin{center}
$Q(K_{\mathbb{C}{\rm{P}}^{1}}) = S^{3}/\mathbb{Z}_{2} =  \mathbb{R}\mathbb{P}^{3}.$    
\end{center}
By considering the opposite big cell $U =  N^{-}x_{0} \subset X_{B}$ and the local section $s_{U} \colon U \subset \mathbb{C}{\rm{P}}^{1} \to {\rm{SL}}(2,\mathbb{C})$ defined by

\begin{center}

$s_{U}(nx_{0}) = n$, \ \ $\forall n \in N^{-}$,
\end{center}
since $V(\omega_{\alpha}) = \mathbb{C}^{2}$ and $v_{\omega_{\alpha}}^{+} = e_{1}$, see Example \ref{exampleP1}, it follows that 
\begin{center}
$A_{U} = 2\partial \log \big (1 + |z|^{2} \big) = \displaystyle \frac{2\overline{z}dz}{1 + |z|^{2}}.$
\end{center}
Thus, from Theorem \ref{calabiresolution} we have the Calabi-Yau metric $\omega_{\mathscr{C}}$ on $\mathscr{C}( \mathbb{R}\mathbb{P}^{3})$ given by

\begin{center}
    
$\omega_{\mathscr{C}} = \displaystyle rdr \wedge \Bigg ( \frac{\overline{z}dz - zd\overline{z}}{2\sqrt{-1}(1 + |z|^{2})} + \frac{d\theta_{U}}{2} \Bigg ) + \frac{r^{2}d\overline{z} \wedge dz}{2\sqrt{-1}(1 + |z|^{2})^{2}}.$
    
\end{center}
The metric above defines a singular metric on $\mathscr{C}( \mathbb{R}\mathbb{P}^{3}) \cup \{o\}$ with conical singularity at $r = 0$. By considering the Cartan-Remmert reduction of $K_{\mathbb{C}{\rm{P}}^{1}} = T^{\ast}\mathbb{C}{\rm{P}}^{1}$, we obtain a resolution of the metric above provided by the Calabi Ansatz metric $\omega_{CY}$ such that 

\begin{center}
    
$\omega_{CY} = \displaystyle \sqrt{ 2\pi r^{2} + C} \Bigg (   \frac{d\overline{z} \wedge dz}{\pi \sqrt{-1}(1+|z|^{2})^{2}} - \frac{\Big ( d\overline{b}_{U} + \frac{\overline{b}_{U} zd \overline{z}}{(1 + |z|^{2} )}\Big ) \wedge \Big ( db_{U} + \frac{b_{U} \overline{z}dz}{(1 + |z|^{2} )}\Big )}{2\sqrt{-1}(2\pi r^{2} + C)} \Bigg ).$    
    
\end{center}
It is worth pointing out that the metric above is also asymptotically locally Euclidean. In fact, we have $K_{\mathbb{C}{\rm{P}}^{1}} = {\text{Bl}}_{0}(\mathbb{C}^{2}/\mathbb{Z}_{2})$, i.e., the canonical bundle $K_{\mathbb{C}{\rm{P}}^{1}}$ can be seen as the blow-up of $\mathbb{C}^{2}/\mathbb{Z}_{2}$ at $0$. Thus, we have

\begin{center}
    
$\mathscr{C}(S^{3}/\mathbb{Z}_{2}) \cong K_{\mathbb{C}{\rm{P}}^{1}}^{\times} \cong \big (\mathbb{C}^{2}/\mathbb{Z}_{2} \big ) \backslash \{0\}.$
    
\end{center}
Now, let $R = |z_{1}|^{2} + |z_{2}|^{2}$ be a smooth function on $\mathbb{C}^{2}$ and $F_{s} \colon  (\mathbb{C}^{2}/\mathbb{Z}_{2} \big ) \backslash \{0\} \to \mathbb{R}$ such that 

\begin{center}
    
$ F_{s}(R) = \displaystyle R \sqrt{1 + \frac{s^{2}}{R^{2}}} + s \log \Bigg ( \frac{R}{s + \sqrt{s^{2} + R^{2}}}\Bigg ),$
    
\end{center}
where $0 < s \leq 1$. Then, we have that $\omega_{s} = \sqrt{-1} \partial \overline{\partial}F_{s}$ can be smoothly extended in order to define a complete asymptotically locally Euclidean (ALE) Ricci-flat K\"{a}hler-Einstein metric on $K_{\mathbb{C}{\rm{P}}^{1}} = {\text{Bl}}_{0}(\mathbb{C}^{2}/\mathbb{Z}_{2})$,  called the Eguchi-Hanson metric \cite{EGUCHIHANSON}, \cite[Example 7.2.2]{JOYCE}. Hence, we have that $\omega_{s}$ converges smoothly to the flat metric $\omega_{0}$, namely 

\begin{center}
    
$\omega_{s} \sim \sqrt{-1} \partial \overline{\partial} \Big ( |z_{1}|^{2} + |z_{2}|^{2} \Big ),$
    
\end{center}
when $s \to 0$ and $R \gg 0$. For other constructions on $\mathscr{O}(-k) \to \mathbb{C}{\rm{P}}^{1}$, $\forall k \geq 1$, see for instance \cite{LEBRUN}.

\end{example}

\begin{example}[Calabi-Yau cone over the Lens space $S^{2n+1}/\mathbb{Z}_{n+1}$] The same ideas of the previous example can be generalized to $\mathbb{C}{\rm{P}}^{n}$. Consider $G^{\mathbb{C}} = {\rm{SL}}(n+1,\mathbb{C})$ and $P = P_{\omega_{\alpha_{1}}}$, see Example \ref{examplePn}.

As we have seen in Example \ref{COMPLEXHOPF}, in this case we have

\begin{center}
    
$X_{P_{\omega_{\alpha_{1}}}} = \mathbb{C}{\rm{P}}^{n}$ \ \  and  \ \ $\mathscr{P}(\mathbb{C}{\rm{P}}^{n},{\rm{U}}(1)) = \mathbb{Z}\mathrm{e}(\mathcal{Q}_{P_{\omega_{\alpha_{1}}}})$,
    
\end{center}
such that $\mathcal{Q}_{P_{\omega_{\alpha_{1}}}} = Q(\mathscr{O}(-1))$. Since $K_{\mathbb{C}{\rm{P}}^{n}}^{\otimes \frac{1}{n+1}} = \mathscr{O}(-1)$ and $\mathcal{Q}_{P_{\omega_{\alpha_{1}}}} = S^{2n+1}$, it follows that 

\begin{center}
    
$Q(K_{\mathbb{C}{\rm{P}}^{n}}) = S^{2n+1}/\mathbb{Z}_{n+1}$.    
    
\end{center}
Therefore, we have that the Cartan-Remmert reduction 

\begin{center}
    
$\mathscr{R} \colon K_{\mathbb{C}{\rm{P}}^{n}} \to \mathscr{C}(S^{2n+1}/\mathbb{Z}_{n+1}) \cup \{o\},$

\end{center}
provides a crepant resolution for the singular Calabi-Yau cone over $S^{2n+1}/\mathbb{Z}_{n+1}$. From this, we can apply Theorem \ref{calabiresolution} in order to describe the singular Ricci-flat K\"{a}hler metric $\omega_{\mathscr{C}}$ on $\mathscr{C}(S^{2n+1}/\mathbb{Z}_{n+1}) \cup \{o\}$ and its resolution $\omega_{CY}$ provided by the Calabi Ansatz metric on $K_{\mathbb{C}{\rm{P}}^{n}}$. 

By taking a local section $s_{U} \colon U \subset \mathbb{C}{\rm{P}}^{n} \to {\rm{SL}}(n+1,\mathbb{C})$ on the opposite big cell $U =  R_{u}(P_{\omega_{\alpha_{1}}})^{-}x_{0} \subset \mathbb{C}{\rm{P}}^{n}$, such that 
$$s_{U}(nx_{0}) = n \in {\rm{SL}}(n+1,\mathbb{C}),$$ 
since $V(\omega_{\alpha_{1}}) = \mathbb{C}^{n+1}$,  $v_{\omega_{\alpha_{1}}}^{+} = e_{1}$ and $I(\mathbb{C}{\rm{P}}^{n}) = n+1$, it follows that 

\begin{center}

$A_{U} =  (n+1)\displaystyle \partial \log \Big (1 + \sum_{l = 1}^{n}|z_{l}|^{2} \Big ).$
    
\end{center}
Hence, the singular cone metric $\omega_{\mathscr{C}}$ can be expressed by 

\begin{center}
    
$\omega_{\mathscr{C}} =  \displaystyle r dr \wedge \Bigg (\frac{\big ( \partial - \overline{\partial} \big ) \log \big (1 +  \sum_{l = 1}^{n}|z_{l}|^{2} \big )}{2\sqrt{-1}} + \frac{d\theta_{U}}{n+1} \Bigg )  + \frac{ r^{2}\overline{\partial} \partial\log \big (1 + \sum_{l = 1}^{n}|z_{l}|^{2} \big )}{2\sqrt{-1}},$

\end{center}
and the resolution for the singular cone metric above is given by the Calabi Ansatz 
    
\begin{center}
    
$\omega_{CY} = \displaystyle (2\pi r^{2} + C)^{\frac{1}{n+1}} \Bigg (\omega_{\mathbb{C}{\rm{P}}^{n}} - \frac{\sqrt{-1}}{n+1} \frac{\nabla b_{U} \wedge \overline{\nabla b}_{U}}{(2\pi r^{2} + C)} \Bigg ),$
\end{center}
where $C > 0$ is some positive constant and

\begin{itemize}

\item $\omega_{\mathbb{C}{\rm{P}}^{n}} = \displaystyle \frac{(n+1)  \overline{\partial} \partial  \log \Big (1 + \sum_{l = 1}^{n}|z_{l}|^{2} \Big )}{2\pi \sqrt{-1}}$,\\

\item $\nabla b_{U} = d b_{U} + (n+1) b_{U} \displaystyle \partial \log \Big (1 + \sum_{l = 1}^{n}|z_{l}|^{2} \Big ).$
    
\end{itemize}
Thus, we obtain a resolution for the singular cone over $S^{2n+1}/\mathbb{Z}_{n+1}$ provided by the Calabi-Yau manifold $(K_{\mathbb{C}{\rm{P}}^{n}},\omega_{CY})$. Further results on scalar flat metrics and other constructions on $\mathscr{O}(-k) \to \mathbb{C}{\rm{P}}^{n}$, $\forall k \geq n$, can be found in \cite{PEDERSEN}, \cite{abreu}.

\end{example}

\begin{example}[Calabi-Yau cone over $\mathscr{V}_{2}(\mathbb{R}^{6})/\mathbb{Z}_{4}$] As we have seen in Examples \ref{grassmanian} and \ref{STIEFEL}, if we consider $G^{\mathbb{C}} = {\rm{SL}}(4,\mathbb{C})$ and $P = P_{\omega_{\alpha_{2}}}$, it follows that 

\begin{center}
$X_{P_{\omega_{\alpha_{2}}}} = {\rm{Gr}}(2,\mathbb{C}^{4})$ \ \ and \ \ $\mathscr{P}({\rm{Gr}}(2,\mathbb{C}^{4}),{\rm{U}}(1)) = \mathbb{Z}\mathrm{e}(\mathcal{Q}_{P_{\omega_{\alpha_{2}}}}),$
\end{center}
such that $\mathcal{Q}_{P_{\omega_{\alpha_{2}}}} = Q(\mathscr{O}_{\alpha_{2}}(-1)) = \mathscr{V}_{2}(\mathbb{R}^{6})$. From Example \ref{grassmanian} we have $K_{{\rm{Gr}}(2,\mathbb{C}^{4})}^{\otimes \frac{1}{4}} = \mathscr{O}_{\alpha_{2}}(-1)$, thus
\begin{center}
    
$Q(K_{{\rm{Gr}}(2,\mathbb{C}^{4})}) = \mathscr{V}_{2}(\mathbb{R}^{6})/\mathbb{Z}_{4}$.    
    
\end{center}
Since in this case we have  $V(\omega_{\alpha_{2}}) = \bigwedge^{2}(\mathbb{C}^{4})$ and $v_{\omega_{\alpha_{2}}}^{+} =  e_{1} \wedge e_{2}$, the gauge potential $A_{U}$ over the opposite big cell $U = R_{u}(P_{\omega_{\alpha_{2}}})^{-}x_{0}$ is given by 

\begin{center}

$A_{U} = \displaystyle 4 \partial \log \Big ( 1 + \sum_{k = 1}^{4}|z_{k}|^{2} + \bigg |\det \begin{pmatrix}
 z_{1} & z_{3} \\
 z_{2} & z_{4}
\end{pmatrix} \bigg |^{2} \Big).$
\end{center}
Hence, from Theorem \ref{calabiresolution} and Example \ref{maxparabolic} we have the singular Calabi-Yau metric defined on the cone over $\mathscr{V}_{2}(\mathbb{R}^{6})/\mathbb{Z}_{4}$ given by

$\omega_{\mathscr{C}} = r dr \wedge \Bigg (2 \frac{ \displaystyle \big ( \partial - \overline{\partial} \big )\log \Big ( 1 + \sum_{k = 1}^{4}|z_{k}|^{2} + \bigg |\det \begin{pmatrix}
 z_{1} & z_{3} \\
 z_{2} & z_{4}
\end{pmatrix} \bigg |^{2} \Big)}{5\sqrt{-1}} + \displaystyle \frac{d\theta_{U}}{5} \Bigg ) $
$$ + \ \ \displaystyle \frac{2r^{2}}{5\sqrt{-1}} \overline{\partial} \partial\log \Big ( 1 + \sum_{k = 1}^{4}|z_{k}|^{2} + \bigg |\det \begin{pmatrix}
 z_{1} & z_{3} \\
 z_{2} & z_{4}
\end{pmatrix} \bigg |^{2} \Big).$$

By applying Theorem \ref{calabiresolution} we obtain a resolution for the singular metric above provided by the Calabi Ansatz 

\begin{center}
$\omega_{CY} = \displaystyle (2\pi r^{2} + C)^{\frac{1}{5}} \Bigg (\omega_{{\rm{Gr}}(2,\mathbb{C}^{4})} - \frac{\sqrt{-1}}{5} \frac{\nabla b_{U} \wedge \overline{\nabla b}_{U}}{(2\pi r^{2} + C)} \Bigg ),$
\end{center}
such that 
\begin{itemize}

\item $\omega_{{\rm{Gr}}(2,\mathbb{C}^{4})} =  \displaystyle \frac{2 \sqrt{-1}}{\pi}  \partial \overline{\partial} \log \Big (1+ \sum_{k = 1}^{4}|z_{k}|^{2} + \bigg |\det \begin{pmatrix}
 z_{1} & z_{3} \\
 z_{2} & z_{4}
\end{pmatrix} \bigg |^{2} \Big)$,

\item $\nabla b_{U} = db_{U} + 4 b_{U} \displaystyle \partial \log \Big (1+ \sum_{k = 1}^{4}|z_{k}|^{2} + \bigg |\det \begin{pmatrix}
 z_{1} & z_{3} \\
 z_{2} & z_{4}
\end{pmatrix} \bigg |^{2} \Big)$.
\end{itemize}
Hence, we have that $(K_{{\rm{Gr}}(2,\mathbb{C}^{4})},\omega_{CY})$ provides a resolution for the singular Calabi-Yau cone obtained from $(\mathscr{C}(\mathscr{V}_{2}(\mathbb{R}^{6})/\mathbb{Z}_{4}),\omega_{\mathscr{C}})$.

\end{example}

\begin{example}[Calabi-Yau cone $K_{{\rm{Gr}}(k,\mathbb{C}^{n+1})}^{\times}$] 
\label{grassexample}
The previous example can be generalized as follows. Consider $G^{\mathbb{C}} = {\rm{SL}}(n+1,\mathbb{C})$, by fixing the Cartan subalgebra $\mathfrak{h} \subset \mathfrak{sl}(n+1,\mathbb{C})$ given by diagonal matrices whose the trace is equal to zero, we have the set of simple roots given by
$$\Sigma = \Big \{ \alpha_{l} = \epsilon_{l} - \epsilon_{l+1} \ \Big | \ l = 1, \ldots,n\Big\},$$
here $\epsilon_{l} \colon {\text{diag}}\{a_{1},\ldots,a_{n+1} \} \mapsto a_{l}$, $ \forall l = 1, \ldots,n+1$. Therefore, the set of positive roots is given by
$$\Pi^+ = \Big \{ \alpha_{ij} = \epsilon_{i} - \epsilon_{j} \ \Big | \ i<j  \Big\}. $$
In this case we consider $\Theta = \Sigma \backslash \{\alpha_{k}\}$ and $P = P_{\omega_{\alpha_{k}}}$, thus we have

\begin{center}

${\rm{SL}}(n+1,\mathbb{C})/P_{\omega_{\alpha_{k}}} = {\rm{Gr}}(k,\mathbb{C}^{n+1}).$

\end{center}
A straightforward computation shows that 

\begin{center}

$I({\rm{Gr}}(k,\mathbb{C}^{n+1})) = n+1,$ \ \  and  \ \ $\mathscr{P}({\rm{Gr}}(k,\mathbb{C}^{n+1}),{\rm{U}}(1)) = \mathbb{Z}\mathrm{e}(\mathcal{Q}_{P_{\omega_{\alpha_{k}}}}),$

\end{center}
which imples that

\begin{center}

$Q(K_{{\rm{Gr}}(k,\mathbb{C}^{n+1})}) = \mathcal{Q}_{P_{\omega_{\alpha_{k}}}}/\mathbb{Z}_{n+1}.$

\end{center}
Since we have $V(\omega_{\alpha_{k}}) = \bigwedge^{k}(\mathbb{C}^{n+1})$ and $v_{\omega_{\alpha_{k}}}^{+} = e_{1} \wedge \ldots \wedge e_{k}$, by taking the coordinate neighborhood $U =  R_{u}(P_{\omega_{\alpha_{k}}})^{-}x_{0} \subset {\rm{Gr}}(k,\mathbb{C}^{n+1})$, such that 

\begin{center}

$Z \in \mathbb{C}^{(n+1-k)k} \mapsto n(Z)x_{0} = \begin{pmatrix}
 \ 1_{k} & 0_{k,n+1-k} \\
 Z & 1_{n+1-k}
\end{pmatrix}x_{0},$

\end{center}
here we identified $\mathbb{C}^{(n+1-k)k} \cong {\rm{M}}_{n+1-k,k}(\mathbb{C})$, we can take the local section $s_{U} \colon U \subset {\rm{Gr}}(k,\mathbb{C}^{n+1})\to {\rm{SL}}(n+1,\mathbb{C})$ defined by

\begin{center}

$s_{U}(n(Z)x_{0}) = n(Z) = \begin{pmatrix}
 \ 1_{k} & 0_{k,n+1-k} \\
 Z & 1_{n+1-k}
\end{pmatrix}.$

\end{center}
From the data above we obtain the gauge potential 

\begin{center}

$A_{U} = (n+1)\partial \log \Bigg (\sum_{I} \bigg | \det_{I} \begin{pmatrix}
 \ 1_{k} \\
 Z 
\end{pmatrix} \bigg |^{2} \Bigg ),$

\end{center}
where the sum above is taken over all $k \times k$ submatrices whose the lines are labeled by $I = \{i_{1} < \ldots < i_{k}\} \subset \{1, \ldots, n+1\}$. Thus, we have the singular metric on the cone $\mathscr{C}(\mathcal{Q}_{P_{\omega_{\alpha_{k}}}}/\mathbb{Z}_{n+1}) \cup \{o\}$ given by
    
$\omega_{\mathscr{C}} = \displaystyle r dr \wedge \Bigg (\frac{(n+1)}{2((n+1-k)k + 1)\sqrt{-1}} \big ( \partial - \overline{\partial} \big )\log \textstyle{\Bigg (\sum_{I} \bigg | \det_{I} \begin{pmatrix}
 \ 1_{k} \\
 Z 
\end{pmatrix} \bigg |^{2} \Bigg )} + \displaystyle \frac{d\theta_{U}}{(n+1-k)k + 1} \Bigg ) $ 

$$ + \ \frac{ r^{2}(n+1)}{2((n+1-k)k + 1)\sqrt{-1}}\overline{\partial} \partial \log \textstyle{ \Bigg (\sum_{I} \bigg | \det_{I} \begin{pmatrix}
 \ 1_{k} \\
 Z 
\end{pmatrix} \bigg |^{2} \Bigg )}.$$
From this, the Calabi Ansatz metric on $K_{{\rm{Gr}}(k,\mathbb{C}^{n+1})}$ defined by

\begin{center}

$\omega_{CY} = \displaystyle (2\pi r^{2} + C)^{\frac{1}{(n+1-k)k + 1}} \Bigg (\omega_{{\rm{Gr}}(k,\mathbb{C}^{n+1})} - \frac{\sqrt{-1}}{(n+1-k)k + 1} \frac{\nabla b_{U} \wedge \overline{\nabla b}_{U}}{(2\pi r^{2} + C)} \Bigg ),$

\end{center}
such that 

\begin{itemize}

\item $\omega_{{\rm{Gr}}(k,\mathbb{C}^{n+1})} = \displaystyle \frac{(n+1)}{2\pi \sqrt{-1}}\overline{\partial} \partial\log \textstyle{ \Bigg (\sum_{I} \bigg | \det_{I} \begin{pmatrix}
 \ 1_{k} \\
 Z 
\end{pmatrix} \bigg |^{2} \Bigg )}$,

\item $\nabla b_{U} = d b_{U} + (n+1) b_{U}\partial \log \Bigg (\sum_{I} \bigg | \det_{I} \begin{pmatrix}
 \ 1_{k} \\
 Z 
\end{pmatrix} \bigg |^{2} \Bigg ),$
    
\end{itemize}
provides a resolution for the singular cone metric $\omega_{\mathscr{C}}$. 
\end{example}

As we have seen the examples which we have described so far are given by maximal parabolic subgroups of ${\rm{SL}}(n+1,\mathbb{C})$. In what follows we provide examples of maximal flag manifolds.

\begin{example}[Calabi-Yau cone over $X_{1,1}/\mathbb{Z}_{2} $] Consider $G^{\mathbb{C}} = {\rm{SL}}(3,\mathbb{C})$ and $P_{\emptyset} = B$ (Borel subgroup). In this case we have the Wallach flag manifold

\begin{center}

$X_{B} = {\rm{SL}}(3,\mathbb{C})/B = {\rm{SU}}(3)/T^{2}.$

\end{center}
By keeping the notation of the previous example, we have that $\Sigma = \{\alpha_{1},\alpha_{2}\}$ and 

\begin{center}

$\Pi^{+} = \big \{ \alpha_{1},\alpha_{2}, \alpha_{1} + \alpha_{2}\big \}$,

\end{center}
thus we obtain $\delta_{B} = 2\alpha_{1} + 2\alpha_{2}$. A straightforward computation shows that $I(X_{B}) = 2$ and

\begin{center}
$\mathscr{P}(X_{B},{\rm{U}}(1)) = \mathbb{Z}\mathrm{e}(Q(\omega_{\alpha_{1}})) \oplus \mathbb{Z}\mathrm{e}(Q(\omega_{\alpha_{2}})).$
\end{center}
Hence, we have 

\begin{center}

$\mathcal{Q}_{B} = Q(-\omega_{\alpha_{1}}) + Q(-\omega_{\alpha_{2}}) = {\rm{SU}}(3)/{\rm{U}}(1) = X_{1,1},$

\end{center}
the manifold $X_{1,1}$ is an example of Aloff-Wallach space \cite{ALOFF}. Therefore, from the last comments we obtain 

\begin{center}

$Q(K_{X_{B}}) = X_{1,1}/\mathbb{Z}_{2}.$

\end{center}
In order to compute the cone metric $\omega_{\mathscr{C}}$ on $\mathscr{C}(X_{1,1}/\mathbb{Z}_{2})$ we observe that in this case we have

\begin{center}

$V(\omega_{\alpha_{1}}) = \mathbb{C}^{3}$ \ \  and \ \ $V(\omega_{\alpha_{2}}) = \bigwedge^{2}(\mathbb{C}^{3}),$

\end{center}
where $v_{\omega_{\alpha_{1}}}^{+} = e_{1}$ and $v_{\omega_{\alpha_{2}}}^{+} = e_{1} \wedge e_{2}$. Now, we consider the opposite big cell $U = R_{u}(B)^{-}x_{0} \subset X_{B}$ such that 

\begin{center}

$U = \Bigg \{ \begin{pmatrix}
1 & 0 & 0 \\
z_{1} & 1 & 0 \\                  
z_{2}  & z_{3} & 1
 \end{pmatrix}x_{0} \ \Bigg | \ z_{1},z_{2},z_{3} \in \mathbb{C} \Bigg \}$.

\end{center}
By taking the local section $s_{U} \colon U \subset X_{B} \to {\rm{SL}}(3,\mathbb{C})$, such that $s_{U}(nx_{0}) = n$, a straightforward computation shows that the gauge potential \ref{GAU} is given by

\begin{center}

$A_{U} = \displaystyle 2 \partial \log \bigg ( 1 + \sum_{i = 1}^{2}|z_{i}|^{2} \bigg ) + 2 \partial \log \bigg (1 + |z_{3}|^{2} + \bigg | \det \begin{pmatrix}
z_{1} & 1  \\                  
z_{2}  & z_{3} 
 \end{pmatrix} \bigg |^{2} \bigg ).$

\end{center}
From the expression above we obtain the following formula for $\omega_{\mathscr{C}}$

$\omega_{\mathscr{C}} = r dr \wedge \Bigg ( \frac{ \big ( \partial - \overline{\partial} \big ) \log \bigg [ \bigg ( 1 + \displaystyle \sum_{i = 1}^{2}|z_{i}|^{2} \bigg ) \bigg (1 + |z_{3}|^{2} + \bigg | \det \begin{pmatrix}
z_{1} & 1  \\                  
z_{2}  & z_{3} 
 \end{pmatrix} \bigg |^{2} \bigg ) \bigg ]}{4\sqrt{-1}} + \displaystyle \frac{d\theta_{U}}{4} \Bigg ) $
 
$$ + \ \ \displaystyle \frac{r^{2}}{4\sqrt{-1}} \overline{\partial} \partial\log \bigg [ \bigg ( 1 + \sum_{i = 1}^{2}|z_{i}|^{2} \bigg ) \bigg (1 + |z_{3}|^{2} + \bigg | \det \begin{pmatrix}
z_{1} & 1  \\                  
z_{2}  & z_{3} 
 \end{pmatrix} \bigg |^{2} \bigg ) \bigg ].$$
Therefore, from the Cartan-Remmert reduction we obtain a crepant resolution for the singular cone $(\mathscr{C}(X_{1,1}/\mathbb{Z}_{2}),\omega_{\mathscr{C}})$ provided by $(K_{{\rm{SU}}(3)/T^{2}},\omega_{CY})$ such that 

\begin{center}    
$\omega_{CY} = \displaystyle (2\pi r^{2} + C)^{\frac{1}{4}} \Bigg (\omega_{{\rm{SU}}(3)/T^{2}} - \frac{\sqrt{-1}}{4} \frac{\nabla b_{U} \wedge \overline{\nabla b}_{U}}{(2\pi r^{2} + C)} \Bigg ),$
\end{center}
with 
\begin{itemize}

\item $\omega_{{\rm{SU}}(3)/T^{2}} = \displaystyle \frac{1}{\pi \sqrt{-1}} \overline{\partial} \partial\log \bigg [ \bigg ( 1 + \sum_{i = 1}^{2}|z_{i}|^{2} \bigg ) \bigg (1 + |z_{3}|^{2} + \bigg | \det \begin{pmatrix}
z_{1} & 1  \\                  
z_{2}  & z_{3} 
 \end{pmatrix} \bigg |^{2} \bigg ) \bigg ],$
 
\item $\nabla b_{U} = d b_{U} + 2 b_{U}\partial\log \bigg [ \bigg ( 1 + \displaystyle \sum_{i = 1}^{2}|z_{i}|^{2} \bigg ) \bigg (1 + |z_{3}|^{2} + \bigg | \det \begin{pmatrix}
z_{1} & 1  \\                  
z_{2}  & z_{3} 
 \end{pmatrix} \bigg |^{2} \bigg ) \bigg ].$ 

\end{itemize}
It is worth pointing out that from Remark \ref{GOTOREMARK} we obtain
\begin{center}
$ 2 = b_{2}({\rm{SU}}(3)/T^{2}) = \dim H^{2}({\rm{SU}}(3)/T^{2}) = \dim H^{2}(K_{{\rm{SU}}(3)/T^{2}}).$
\end{center}
Thus, from \cite[Example 6.3]{GOTO} we have a K\"{a}hler class which does not belong to the compactly supported cohomology group of $K_{{\rm{SU}}(3)/T^{2}}$.
\end{example}

\begin{example}[Calabi-Yau cone $K_{{\rm{SU}}(n+1)/T^{n}}^{\times}$] The previous example can be easily generalized. In fact, consider the Lie-theoretical data for the Lie group ${\rm{SL}}(n+1,\mathbb{C})$ as in Example \ref{grassexample}. By taking $P = B \subset {\rm{SL}}(n+1,\mathbb{C})$ (Borel subgroup), we obtain

\begin{center}

$X_{B} = {\rm{SL}}(n+1,\mathbb{C})/B = {\rm{SU}}(n+1)/T^{n}$,

\end{center}
notice that in this case we have $\Theta = \emptyset$. Now, a straightforward computation shows that 

\begin{center}

$\mathscr{P}(X_{B},{\rm{U}}(1)) = \displaystyle \bigoplus_{l = 1}^{n} \mathbb{Z}\mathrm{e}(Q(\omega_{\alpha_{l}})).$

\end{center}
In order to compute the metrics of Theorem \ref{calabiresolution}, we observe that in this case we have
\begin{center}

$V(\omega_{\alpha_{l}}) = \bigwedge^{l}(\mathbb{C}^{n+1})$, \ \ $v_{\omega_{\alpha_{l}}}^{+} = e_{1} \wedge \ldots \wedge e_{l},$

\end{center}
for $l = 1,\ldots,n$. Let $U = R_{u}(B)^{-}x_{0} \subset X_{B}$ be the opposite big cell. For this particular example this open set is parameterized by the holomorphic coordinates 
$$n = \begin{pmatrix}
  1 & 0 & 0 & \cdots & 0 \\
  z_{21} & 1 & 0 & \cdots & 0  \\
  z_{31} & z_{32} & 1 &\cdots & 0 \\
  \vdots & \vdots & \vdots & \ddots & \vdots \\
  z_{n+1,1} & z_{n+1,2} & z_{n+1,3} & \cdots & 1 
 \end{pmatrix},$$
where $ n = n^{-}(z) \in N^{-}$ and $z = (z_{ij}) \in \mathbb{C}^{\frac{n(n+1)}{2}}$. We define for each subset $I = \{i_{1} < \cdots < i_{k}\} \subset \{1,\cdots,n+1\}$, with $1 \leq k \leq n$, the polynomial function $\det_{I} \colon {\rm{SL}}(n+1,\mathbb{C}) \to \mathbb{C}$, such that 

$$\textstyle{\det_{I}}(g) = \det \begin{pmatrix}
  g_{i_{1}1} & g_{i_{1}2} & \cdots & g_{i_{1} k} \\
  g_{i_{2}1} & g_{i_{2}2} & \cdots & g_{i_{2} k}  \\
  \vdots & \vdots& \ddots & \vdots \\
  g_{i_{k}1} & g_{i_{k}2} & \cdots & g_{i_{k}k} 
 \end{pmatrix},$$
for every $g \in {\rm{SL}}(n+1,\mathbb{C})$. From this, we have for every $g \in {\rm{SL}}(n+1,\mathbb{C})$ that 

\begin{center}

$g(e_{1} \wedge \ldots \wedge e_{l}) =  \displaystyle \sum_{i_{1} < \cdots < i_{l}} \textstyle{\textstyle{\det_{I}}}(g)e_{i_{1}} \wedge \ldots \wedge e_{i_{l}},$

\end{center}
notice that the sum above is taken over $I = \{i_{1} < \cdots < i_{l}\} \subset \{1,\cdots,n+1\}$, with $1 \leq l \leq n$. By taking the local section $s_{U} \colon U \subset X_{B} \to {\rm{SL}}(n+1,\mathbb{C})$, such that $s_{U}(n^{-}(z)x_{0}) = n^{-}(z)$, we obtain the gauge potential

\begin{center}

$A_{U} = \displaystyle \sum_{l = 1}^{n} \langle \delta_{B},h_{\alpha_{l}}^{\vee} \rangle \partial \log \Bigg (  \sum_{i_{1} < \cdots < i_{l}} \bigg | \textstyle{\det_{I}} \begin{pmatrix}
  1 & 0 & \cdots & 0 \\
  z_{21} & 1 & \cdots & 0  \\
  \vdots & \vdots & \ddots & \vdots \\
  z_{n+1,1} & z_{n+1,2} & \cdots & 1 
 \end{pmatrix} \bigg |^{2}\Bigg).$ 

\end{center}
For the sake of simplicity we shall denote

\begin{center}

$A_{U} = \displaystyle \partial \log \Bigg ( \prod_{l = 1}^{n} \Big (  \sum_{i_{1} < \cdots < i_{l}} \big | \textstyle{\det_{I}} \big (n^{-}(z) \big ) \big |^{2}\Big)^{\langle \delta_{B},h_{\alpha_{l}}^{\vee} \rangle}\Bigg),$ 

\end{center}
to stand for the previous expression. Hence, by applying Theorem \ref{calabiresolution} we obtain the cone metric $\omega_{\mathscr{C}}$ such that

$\omega_{\mathscr{C}} = \displaystyle r dr \wedge \Bigg (\frac{(\overline{\partial}-\partial)\log \Bigg (\displaystyle \prod_{l = 1}^{n} \Big (  \sum_{i_{1} < \cdots < i_{l}} \big | \textstyle{\det_{I}} \big (n^{-}(z) \big ) \big |^{2}\Big)^{\langle \delta_{B},h_{\alpha_{l}}^{\vee} \rangle}\Bigg)}{\sqrt{-1}(n(n+1)+2)} + \frac{2d\theta_{U}}{n(n+1)} \Bigg )$    
    
$$ + \ \ \displaystyle \frac{r^{2}}{\sqrt{-1}(n(n+1)+2)} {\overline{\partial}}\partial \log \Bigg ( \prod_{l = 1}^{n} \Big (  \sum_{i_{1} < \cdots < i_{l}} \big | \textstyle{\det_{I}} \big (n^{-}(z) \big ) \big |^{2}\Big)^{\langle \delta_{B},h_{\alpha_{l}}^{\vee} \rangle}\Bigg),$$
the metric above defines a Calabi-Yau metric on the cone $K_{{\rm{SU}}(n+1)/T^{n}}^{\times}$. From the Cartan-Remmert reduction we obtain a resolution for the singular Calabi-Yau cone $K_{{\rm{SU}}(n+1)/T^{n}}^{\times} \cup \{o\}$ provided by $(K_{{\rm{SU}}(n+1)/T^{n}},\omega_{CY})$ where

\begin{center}    
$\omega_{CY} = \displaystyle (2\pi r^{2} + C)^{\frac{2}{n(n+1)+2}} \Bigg (\omega_{{\rm{SU}}(n+1)/T^{n}} - \frac{2\sqrt{-1}}{n(n+1)+2} \frac{\nabla b_{U} \wedge \overline{\nabla b}_{U}}{(2\pi r^{2} + C)} \Bigg ),$
\end{center}
such that
\begin{itemize}

\item $\omega_{{\rm{SU}}(n+1)/T^{n}} = \displaystyle \frac{1}{2\pi \sqrt{-1}} \overline{\partial} \partial\log \Bigg (\displaystyle \prod_{l = 1}^{n} \Big (  \sum_{i_{1} < \cdots < i_{l}} \big | \textstyle{\det_{I}} \big (n^{-}(z) \big ) \big |^{2}\Big)^{\langle \delta_{B},h_{\alpha_{l}}^{\vee} \rangle}\Bigg),$
 
\item $\nabla b_{U} = d b_{U} + b_{U}\partial\log \Bigg (\displaystyle \prod_{l = 1}^{n} \Big (  \sum_{i_{1} < \cdots < i_{l}} \big | \textstyle{\det_{I}} \big (n^{-}(z) \big ) \big |^{2}\Big)^{\langle \delta_{B},h_{\alpha_{l}}^{\vee} \rangle}\Bigg).$ 

\end{itemize}
Thus, we obtain an explicit description for the resolution of the Calabi-Yau cone, which has the Sasaki-Einstein manifold $Q(K_{{\rm{SU}}(n+1)/T^{n}})$ as a link of isolated singularity, given by $(K_{{\rm{SU}}(n+1)/T^{n}},\omega_{CY})$. It is worthwhile to observe that, according to Remark \ref{GOTOREMARK}, we have  

\begin{center}

$ n = \# (\Sigma) = \dim H^{2}({\rm{SU}}(n+1)/T^{n}) = \dim H^{2}(K_{{\rm{SU}}(n+1)/T^{n}}).$

\end{center}
Hence, in this case, we obtain a family of Ricci-flat K\"{a}hler metrics on $K_{{\rm{SU}}(n+1)/T^{n}}$ which does not belong to its compactly supported cohomology group. Therefore, this example also provides a concrete realization for \cite[Example 6.3]{GOTO}.

\end{example}

The examples described in this section provide a huge class of concrete nontrivial examples for the existence part of Conjecture \ref{conj}. Many of the computations which we have done for homogeneous contact manifolds associated to ${\rm{SL}}(n+1,\mathbb{C})$ also can be done for other classical groups, namely, $\rm{SO}(n,\mathbb{C})$ and ${\rm{Sp}}(2n,\mathbb{C})$, see for instance \cite{EDER} to some computations of the Calabi Ansatz metric in low dimensional cases. 

For Lie groups associated to exceptional Lie algebras the computation becomes highly nontrivial. The main reason, in this case, is that we do not have a manageable matrix realization for the associated Lie algebras, thus we can not directly derive a suitable local expression for connections and gauge potentials involved in our computations.

\section*{Acknowledgement}
The author would like to thank the anonymous reviewers for their helpful and constructive comments that greatly contributed to improving the final version of the paper.

\end{document}